\documentclass[a4paper,11pt,reqno]{article}

\usepackage{fullpage}

\usepackage{color}
\def\coleins{} 
\def\colzwo{} 

\usepackage[latin1]{inputenc}
\usepackage{graphicx}
\usepackage{amsmath}
\usepackage{amsthm}
\usepackage{amssymb}
\usepackage{bbm}
\usepackage{url}
\usepackage{enumerate}

\usepackage{psfrag}
\usepackage{multirow}
\usepackage{tikz}
\usepackage[centerlast]{subfigure}
\usepackage{listings}

\usepackage{eso-pic}
\usepackage{marvosym}

\numberwithin{equation}{section}

\usepackage{float}
 
\floatstyle{ruled}
\newfloat{algorithm}{thp}{lop}[section]
\floatname{algorithm}{Algorithm}

\theoremstyle{plain}
\newtheorem{theorem}{Theorem}[section]

\newtheorem{lemma}[theorem]{Lemma}

\newtheorem{definition}[theorem]{Definition}

\theoremstyle{definition}

\theoremstyle{remark}
\newtheorem{remark}[theorem]{Remark}

\newcommand{\N}{\ensuremath{\mathbb{N}}}
\newcommand{\Z}{\ensuremath{\mathbb{Z}}}

\renewcommand{\P}{\ensuremath{\mathbb{P}}}
\newcommand{\E}{\ensuremath{\mathbb{E}}}

\newcommand{\1}{\ensuremath{\mathbbm{1}}}
\newcommand{\M}{\ensuremath{\mathcal{M}}}
\renewcommand{\H}{\ensuremath{\mathcal{H}}}

\newcommand{\Q}{\ensuremath{\mathcal{Q}}}

\def\QGT{\Q_\theta^\text{{\ck $\Lambda$-}GT}}
\def\QSD{\Q_\theta^\text{{\ck $\Lambda$-}SD}}
\def\QSingleHUW{\Q_\theta^{\text{\ck {\ck $\Lambda$-}HUW}^1}}
\def\QPairHUWtwoStepA{\Q_\theta^{\text{\ck {\ck $\Lambda$-}HUW}^2\alpha}}
\def\QPairHUWtwoStepB{\Q_\theta^{\text{\ck {\ck $\Lambda$-}HUW}^2\beta}}
\def\QPairHUWoneStepA{\Q_\theta^{\text{\ck {\ck $\Lambda$-}HUW}^2{\ck B}}}
\def\QPairHUWoneStepB{\Q_\theta^{\text{\ck {\ck $\Lambda$-}HUW}^2{\ck A}}}
\def\QPairHUWmixed{\Q_\theta^{\text{\ck {\ck $\Lambda$-}HUW}^{1.5}}}

\def\ck{}

\newlength{\nodedist}

\newlength{\defaultleveldist}

\newlength{\defaultsibdist}
\begin{document}

\title
{Importance sampling for Lambda-coalescents\\ in the infinitely many sites model}

\author{Matthias Birkner\footnote{Johannes-Gutenberg-Universit\"{a}t Mainz, Institut f\"{u}r Mathematik, Staudingerweg 9, 55099 Mainz, Germany,
e-mail: {\tt birkner@mathematik.uni-mainz.de}}\ ,
Jochen Blath\footnote{Technische Universit\"{a}t Berlin, Institut f\"{u}r Mathematik, Strasse des 17.\ Juni 136, 10623 Berlin, Germany,
e-mail: {\tt blath@math.tu-berlin.de}} \hspace{0em} and
Matthias Steinr\"ucken\footnote{Department of Statistics, University of California, 367 Evans Hall MC 3860, Berkeley, CA 94720-3860, USA, e-mail: {\tt steinrue@stat.berkeley.edu} (corresponding author)}}

\date{5th of May, 2011}

\maketitle

\begin{abstract}
 We present and discuss {\ck new} importance sampling schemes for the
 approximate computation of the sample probability of observed
 genetic types in the infinitely many sites model from population
 genetics.  More specifically, we extend the `classical framework',
 where genealogies are assumed to be governed by Kingman's 
 coalescent, to the more general class of Lambda-coalescents {\ck and develop further 
 Hobolth et.~al.'s (2008) idea of deriving importance sampling schemes based on `compressed
 genetrees'.
 The resulting schemes extend} earlier work by 
 Griffiths and Tavar\'e (1994), 
 Stephens and Donnelly (2000), Birkner and Blath (2008) and Hobolth
 et.~al.~(2008). {\ck We conclude with a performance comparison of classical and new schemes for 
 Beta- and Kingman coalescents.}

\end{abstract}

AMS subject classification.
{\em Primary:}
				62F99 
{\em Secondary:}
				62P10; 
				92D10; 
				92D20   

\vspace{2mm}


\vspace{2mm}

\section{Introduction}

\subsection{Aims and outline of the paper}

 In the present paper we derive and discuss importance sampling schemes for the
 approximate computation of the sampling probability of observed
 genetic types in the infinitely many sites model (ISM), which is used for the analysis 
 of DNA sequence data sampled from a population.  

 In particular, we extend earlier results on this classical problem of likelihood estimation 
 in mathematical genetics 
 in two directions. 

 First, we consider genealogies which may be governed by any 
 member of the rather general class of Lambda-coalescents 
 instead of restricting to the
 classical Kingman's coalescent framework only. These genealogies offer more 
 flexibility in the modelling of `exceptional genealogical events' like 
 extreme reproduction and selective sweeps, see e.g.
 \cite{Birkner2008} for a brief discussion. 
 In particular, we derive the analogues of the 
 `Kingman-scenario' based importance sampling schemes of 
 Griffiths and Tavar\'e \cite{Griffiths1994}, 
 Stephens and Donnelly \cite{Stephens2000} and Hobolth, Uyenoyama and Wiuf \cite{Hobolth2008}.

 For the second direction of our investigation, observe that both the schemes derived by 
 Ethier and Griffiths and Stephens and Donnelly do not take any
 specific information about the genealogical distance of types 
 (which is provided by the infinitely many sites model) 
 into account. Indeed, the latter proposal has been explicitly derived by means of optimality 
 for parent-independent mutation models which in particular 
 do not provide information about genealogical distance. 
 Hobolth et.~al.\ \cite{Hobolth2008} proposed a scheme which 
 can be regarded as a starting point to overcome this
simplification. 
 Indeed, for their proposal distribution, they `compress' the observed 
 genealogical tree to a tree where only one segregating site remains, 
 derive optimal proposals for this compressed tree, and show how to combine them to obtain 
 a proposal for the original tree.
 We show how to extend this method to compressed trees with two (and, in principle more)
 segregating sites, 
 which retain information about the topology of the original tree and the 
 genealogical distance of the types of the sample, leading to further improved importance sampling
 schemes (also in the Lambda-coalescent scenario). We `pay' for this additional genealogical information 
 with an increase of complexity in the derivation and of the proposal 
 scheme.
 Along the way, we discuss the optimality of the analogue of the 
 Stephens and Donnelly proposal for the Lambda-coalescent in the
 infinitely many alleles model.\\

 The paper is organised as follows. In Section~\ref{ssn:genetree} we discuss in detail the combinatorial
 framework of samples in the infinitely many sites model. In Section~\ref{sseq:rec} we formulate various
 recursions which form the basis of our importance sampling schemes.
 Section~\ref{sseq:iii} and Section~\ref{ssn:opt} discuss the creation of sample histories and the basic framework
 for importance sampling. We will also briefly discuss the notion of optimality. 
 In Section~\ref{importance_sec_proposal_schemes} we extend earlier and derive new important sampling schemes, 
 whose performance we will analyse in Section~\ref{importance_sec_performance}. 
In the appendix, 
we will provide an algorithm for generating
 sample histories (\ref{sect:algo}), derive some auxiliary 
results on the combinatorics of the infinitely many sites model 
(\ref{sect:combfactor}) and briefly discuss computational aspects 
(\ref{sect:precomp}) as well as estimation of event times, given the 
observed data (\ref{sect:histestimation}). 

\subsection{Genealogies and samples in the infinitely many sites model}
\label{ssn:genetree}

We consider samples taken from a large panmictic population of
constant size evolving due to random mating and mutation according to
the infinitely many sites model. We study the distribution of (neutral) genetic 
variation at a single locus and may therefore assume that the genealogy 
of the sampled genes is described by an exchangeable coalescent process. 
Extending the classical framework of \cite{Ethier1987}, we consider in particular 
genealogies governed by so-called Lambda-coalescents, 
hence allowing multiple, but not simultaneous multiple 
collisions. 

Recall that Pitman (\cite{Pitman1999}) and Sagitov (\cite{Sagitov1999})
introduced and discussed coalescents in which 
more than just two blocks may merge at a time.
Informally, a Lambda-coalescent is a partition-valued Markov process, whose dynamics is as follows:
Whenever there are $b \in \N$ blocks in the partition at present, each
$k$-tuple of blocks (where $2 \le k \le b \le n$) merges to form a
single block at rate $\lambda_{b,k}$, where the rates are given by
\begin{equation}
\label{Lambdarates}
\lambda_{b, k} = \int_0^1 x^{k}(1-x)^{b-k} \frac{1}{x^2} \Lambda(dx),
\end{equation}
for some finite measure $\Lambda$ on the unit interval. Further, denote by
\begin{equation}
\label{Lambdatotalrates}
\lambda_b := \sum_{k=2}^{b} \lambda_{b, k}
\end{equation}
the total rate at which mergers happen while there are $b$ blocks present.

Note that the family of Lambda-coalescents is rather large, and in
particular cannot be parametrised by a few real parameters. Important
examples include $\Lambda=\delta_0$ (Kingman's coalescent) and
$\Lambda=\delta_1$ (leading to star-shaped genealogies, i.e.~one huge
merger into one single block).  Later, we will also be concerned with an
important parametric subclass of $\Lambda$-coalescents, namely the
so-called {\em Beta-coalescents}, where $\Lambda$ has a
Beta$(2-\alpha, \alpha)$-density for some $\alpha \in [1,2]$.
Note that such coalescents occur as limits of genealogies of 
population models, where single individuals
may occasionally be able to produce almost instantaneously a 
non-negligible fraction of the total population size, 
see e.g.~\cite{Birkner2009} for a review. W.l.g.~we assume that 
$\Lambda([0,1]) =1$.\\

We now introduce {\ck detailed} notation to describe samples in the infinitely many sites model.
Note that we represent our data in the form presented in
\cite{Ethier1987} resp.~\cite{Griffiths1994}.  A discussion of how to
transform actual DNA sequence data into this format can be found
e.g.~in \cite[Section~2.1]{Birkner2008} 
(assuming known ancestral types for each segregating site).
{\ck Although the notation for the description of samples in the infinitely many sites model under various equivalence classes seems to 
be relatively standard, we chose to provide full details here, including several formulations of the recursions for observed type probabilities,
since the treatment of the combinatorics of samples is somewhat inconsistent across the literature (see, e.g., Remark~\ref{rem:themess} for some of the subtleties).}

{\ck W}e represent a sample of size $n$ by a
vector ${\bf x} = (x_1,\ldots,x_n)$ of $n$ genetic types, where each type $x_i$ 
is given as a list of positive integers representing mutations
\begin{equation}
\label{eq:type1}
x_i = (x_{i0},\ldots,x_{ij_i}) \in \Z_+^{j_i}.
\end{equation}
Such an ${\bf x}$ is called a \emph{tree} if
\begin{enumerate}
\item for fixed $i \in \{1,\ldots,n\}$ the coordinates $x_{ij}$ are distinct
for all $j \in \Z_+$,
\item whenever for some $i,i' \in \{1,\ldots,n\}$, $j,j' \in \Z_+$, $x_{ij} =
x_{i'j'}$ holds, then $x_{i,j+l} = x_{i',j'+l}$ holds for all $l \in \Z_+$,
\item there exist $j_1,\ldots,j_n \in \Z_+$
such that $x_{1j_1} = x_{2j_2}\ldots =
x_{nj_n}$.
\end{enumerate}
The space of all trees of size $n$ is denoted by $\mathcal{T}_n$. 

Next, we introduce an equivalence relation `$\sim$' on
$\mathcal{T}_n$, where two trees ${\bf x},{\bf y} \in \mathcal{T}_n$
are said to be equivalent if there exists a bijection $\zeta \colon
\Z_+ \to \Z_+$ such that $y_{ij} = \zeta (x_{ij})$ holds for all $i
\in \{1,\ldots,n\}$ and $j \in \Z_+$. Denote by
$(\mathcal{T}_n\slash\!\!\sim)$ the set of equivalence classes under
the relation $\sim$ and by $(\mathcal{T}_n\slash\!\!\sim)_0$ the
restriction of this set to those classes where $x_i \neq x_j$ if
$i\neq j$. The number of segregating sites $s$ is given as the number
of different $x_{i,j}$ that appear in at least one but not all
elements in ${\bf x}$. Note that this does not depend on the actual
representative of the class. Denote by
$(\mathcal{T}_{s,n}\slash\!\!\sim)$ the set of equivalence classes
representing a tree of size $n$ with $s$ segregating sites.
Note that for simplicity, we will always assume $x_{ij} \in \{0,1,\ldots,s\}$. 
Recall that the \emph{complexity} of a sample of size $n$ with $s$ 
segregating sites is defined to be $n+s-1$. 
Elements of $(\mathcal{T}_n\slash\!\!\sim)$ are called \emph{unlabelled} 
trees in \cite[p.~528, l.~-10]{Ethier1987}.
We will sometimes emphasise 
the fact that the order of the samples (equivalently, of the types in the 
case of distinct entries) plays a role by calling them  
\emph{ordered unlabelled} trees. 

A type configuration ${\bf x} = (x_1, \ldots, x_n) \in
(\mathcal{T}_n\slash\!\!\sim)$ can be represented by a pair $({\bf t},{\bf a})$
of a tree ${\bf t} \in (\mathcal{T}_d\slash\!\!\sim)_0$ of the different types
that occur in ${\bf x}$ and an ordered partition ${\bf a} = \big( A_1,\ldots,A_d
\big)$ that specify at which position in the sample the corresponding type occurs 
(i.e.~we think of {\em ordered} types). The number
of distinct types is denoted by $d = \big\vert \{ x_i \colon i=1,\ldots,n \}
\big\vert$. Furthermore, $A_i = \{ j \colon t_i = x_j\}$, $A_i \cap A_j =
\emptyset\;\forall i \neq j$ and $\bigcup_{i=1}^d A_i = \{1,\ldots,n\}$ holds.
Note that this notation introduces an artificial order of the occurring types. In
the sequel the actual sample numbers of the types will play no role, but rather the
multiplicities. For this purpose define ${\bf n}^{({\bf a})} := (\vert A_1
\vert, \ldots, \vert A_d \vert)$, the vector containing the sizes of the sets in
${\bf a}$. 
We denote by 
\[({\bf t}, {\bf n}) \in \cup_{d=1}^\infty 
(\mathcal{T}_d\slash\!\!\sim)_0 \times \N^d =: \mathcal{T}^*
\] 
(where ${\bf n}={\bf n}^{({\bf a})}$) an ordered type configuration with 
multiplicities. Note that for a given $({\bf t}, {\bf n})$ with $d$ 
types, there are $n!/(n_1!\cdots n_d!)$ different choices of ${\bf a}$ 
consistent with ${\bf n}$.

Finally, we define the equivalence relation `$\approx$' by saying that
\begin{equation}
({\bf t},{\bf a}) \approx ({\bf t}',{\bf a}')
\end{equation}
holds if there exist a bijection $\zeta \colon \{1,\ldots,s\} \to
\{1,\ldots,s\}$ and a permutation $\sigma \in S_d$ such that $x_{ij}
= \zeta(x_{\sigma(i)j})$ and ${\bf n}^{({\bf a})} = ({\bf n}^{({\bf
a})})_\sigma$, where ${\bf x}$ is representative of the class ${\bf t}$ and
$\sigma$ is applied to the vector componentwise. 
Note that under this
equivalence class the order of the types is lost. We denote such an equivalence
class by $[{\bf t},{\bf n}]=[{\bf t},{\bf n}^{({\bf a})}]$ and call it an \emph{unnumbered
unlabelled sample configuration with unordered types}, \emph{sample
configuration}, or \emph{genetree}, because it accounts for the fact that in a
sample obtained from a population, the numbering of the types and mutations is
artificially imposed. Summarising, in the following we will 
consider 
equivalence classes
\begin{equation}
\label{equivalenceclasses}
({\bf t},{\bf a}) \in (\mathcal{T}_n\slash\!\!\sim) \quad \mbox{ and } \quad 
[{\bf t},{\bf n}] \in (\mathcal{T}_n\slash\!\!\approx).
\end{equation}
Note that $[{\bf t},{\bf n}]$ in our notation denotes
$\big[\Phi_{\bf n}({\bf t})\big]$ in the notation of \cite{Ethier1987}, 
with $[\cdot]$ referring to the equivalence class under $\approx$. 

However, one should be warned that there are several combinatorial 
conventions present in the literature, 
see Remark~\ref{rem:themess} for a discussion of some 
of the ensuing subtleties.

\begin{remark}
Note that by ignoring the tree structure given by ${\bf t}$ and just considering the partition ${\bf n}$ one can map a sample under the infinitely many sites model to a sample in the infinitely many alleles model.
This observation underlies some of the importance sampling schemes discussed below, see Section \ref{importance_sec_griffiths_tavare}.
However, the additional information provided by the infinitely many sites model can be exploited to find 
more efficient proposal distributions, see Section \ref{sseq:huw}.
\end{remark}

\subsection{Recursion for tree probabilities}
\label{sseq:rec}

In this section we recall from \cite{Birkner2008} 
recursions which allow the computation of the probability of observing a given 
type configuration $({\bf t},{\bf a})$.
In the sequel, we always think of randomly ordered types. 

Indeed, with the above notation, the probability of obtaining a 
given sample $({\bf t},{\bf a})$ from the stationary
distribution of the population under the
infinitely many sites mutation model satisfies the recursion
\begin{equation}
\label{recursion}
\begin{split}
p({\bf t},{\bf a}) = &\frac{1}{rn + \lambda_n} \sum_{i:\vert A_i\vert\geq 2} 
                     \sum_{k=2}^{\vert A_i\vert} {\vert A_i\vert\choose k } 
                      \lambda_{n,k} \, p({\bf t},{\bf a} - (k-1){\bf e}_i)\\
& + \frac{r}{rn + \lambda_n} \sum_{i:\vert A_i\vert=1, x_{i0} 
                      \text{unique}, \atop \mathfrak{s}({\bf x}_i) 
                          \neq {\bf x}_j \forall j} p(\mathfrak{s}_i({\bf t}),{\bf a})\\
& + \frac{r}{rn + \lambda_n} \frac{1}{d} \sum_{i:\vert A_i\vert=1, \atop  x_{i0} 
                    \text{unique}} \sum_{j:\mathfrak{s}({\bf x}_i) = {\bf x}_j} 
                   p(\mathfrak{r}_i({\bf t}),\mathfrak{r}_i({\bf a} + {\bf e}_j))
\end{split}
\end{equation}
with the boundary condition $p\big((0), (\{1\})\big)=1$. Here, 
$x_{i0}$ unique means that mutation $x_{i0}$ occurs only in type $i$. 
The operator $\mathfrak{s}({\bf x})$ [the operator $\mathfrak{s}_i({\bf t})$]
removes the outmost mutation [from type $i$] and $\mathfrak{r}_i({\bf
 t})$ removes the $i$-th component of the vector ${\bf t}$.  By
$\mathbf{a}-(k-1){\bf e}_i$ we mean a partition obtained from $\mathbf{a}$
by removing $k-1$ elements from the {\em set} $a_i$ (with implicit
renumbering of the samples so that the result is a partition of $\{1,
\dots, n-k+1\}$). Note that by symmetry, the type probability $p$ will
not depend on the actual choice. 
Finally, $\mathbf{a}+{\bf e}_j$ denotes the partition obtained from ${\bf a}$
by adding an arbitrary element of $\N$ to the set $a_j$ that is not
yet contained in any other set $a_l, l=1, \dots, d$.

\eqref{recursion} can be seen by conditioning on the most recent event
in the coalescent history (or, equivalently, in the
lookdown-construction into which the so-called `$\Lambda$-Fleming-Viot
process', describing the population forwards in time, can be
embedded), see \cite[Section~4]{Birkner2008} and
\cite[Section~3.3.2]{Steinruecken2009} for details and proofs.

Note that in the `Kingman-case', i.e.~$\Lambda = \delta_0$, this
essentially reduces to the recursion provided by Ethier and Griffiths
in
\cite[Corollary~4.2]{Ethier1987} (see also Remark~\ref{rem:themess} 
below).
The relation between the sampling probabilities of
ordered numbered samples $p({\bf t},{\bf a})$ and the probabilities of
the corresponding unordered unnumbered samples $p[{\bf t},{\bf n}]$ is
given by
\begin{equation}\label{eq_relation_ordered_unordered}
p[{\bf t},{\bf n}] = p({\bf t},{\bf a}) \frac{n!}{n_1!\cdots n_d!}\frac{d!}{c({\bf t},{\bf n})}
= p({\bf t},{\bf n}) \frac{d!}{c({\bf t},{\bf n})}
\end{equation}
Here, ${\bf n} = {\bf n}^{({\bf a})}$, $n!/(n_1!\cdots n_d!)$ is 
the number of ordered partitions of $\{1,\dots,n\}$ into 
$d$ subsets with the given sizes, corresponding to the $d$ types and   
\begin{equation}
\label{eq_ctn_definition}
c({\bf t},{\bf n}) := \Big\vert\big\{ \sigma \in S_d : t \sim t_\sigma \text{ and } {\bf n} = {\bf n}_{\sigma} \forall i\big\}\Big\vert,
\end{equation}
where ${\bf n}_{\sigma}=(n_{\sigma(1)},\dots,n_{\sigma(d)})$. 
There are $d!$ possible orders for the types if the mutations 
carry distinct labels, each of which is equivalent to $c({\bf t},{\bf n})$ 
others if mutation labels are disregarded. Thus, there are 
$d!/c({\bf t},{\bf n})$ different re-orderings of the 
types that cannot be transformed into each other by re-labelling the 
mutations, explaining \eqref{eq_relation_ordered_unordered}. 

Note that $c({\bf t},{\bf n})=c([{\bf t},{\bf n}])$ depends in fact only on
$[{\bf t},{\bf n}]$.
For a constructive way to evaluate 
$c({\bf t},{\bf n})$, see Lemma~\ref{computingct}.

Recursion~\eqref{recursion} can be combined with
relation~\eqref{eq_relation_ordered_unordered} to obtain a recursion for the
unordered sampling probabilities $p[{\bf t},{\bf n}]$:
\begin{equation}\label{real_ims_recursion}
\begin{split}
p[{\bf t},{\bf n}] =& \frac{1}{\lambda_n + nr}\sum_{i:n_i\geq 2} \sum_{k=2}^{n_i} {n \choose k} \lambda_{n,k}\frac{n_i-k+1}{n-k+1} \frac{c\big({\bf t},{\bf n} - (k-1){\bf e}_i\big)}{c({\bf t},{\bf n})} p\big[{\bf t},{\bf n} - (k-1){\bf e}_i\big]\\
& + \frac{r}{\lambda_n + nr} \sum_{i:n_i=1, x_{i0} \text{unique}, \atop \mathfrak{s}({\bf x}_i) \neq {\bf x}_j \forall j} \frac{c\big(\mathfrak{s}_i({\bf t}),{\bf n})\big)}{c({\bf t},{\bf n})} p[\mathfrak{s}_i ({\bf t}), {\bf n}]\\
& + \frac{r}{\lambda_n + nr} \sum_{i:n_i=1,\atop x_{i0} \text{unique}} \sum_{j:\mathfrak{s}({\bf x}_i) = {\bf x}_j}  (n_j + 1) \frac{c\big(\mathfrak{r}_i ({\bf t}), \mathfrak{r}_i ({\bf n} + {\bf e}_j)\big)}{c({\bf t},{\bf n})} p\big[\mathfrak{r}_i ({\bf t}), \mathfrak{r}_i ({\bf n} + {\bf e}_j)\big]
\end{split}
\end{equation}
for the sampling probability of the unordered sample $[{\bf t}, {\bf n}]$.  
Again, we have 
the boundary condition $p\big((0), (1)\big)=1$.
In terms of samples with ordered types $({\bf t}, {\bf n})$, the recursion 
reads 
\begin{equation}\label{ordered_ims_recursion}
\begin{split}
p({\bf t},{\bf n}) =& \frac{1}{\lambda_n + nr}\sum_{i:n_i\geq 2} \sum_{k=2}^{n_i} {n \choose k} \lambda_{n,k}\frac{n_i-k+1}{n-k+1}  p\big({\bf t},{\bf n} - (k-1){\bf e}_i\big)\\
& + \frac{r}{\lambda_n + nr} \sum_{i:n_i=1, x_{i0} \text{unique}, \atop \mathfrak{s}({\bf x}_i) \neq {\bf x}_j \forall j}  p(\mathfrak{s}_i ({\bf t}), {\bf n})\\
& + \frac{r}{\lambda_n + nr} \frac1d 
\sum_{i:n_i=1,\atop x_{i0} \text{unique}} \sum_{j:\mathfrak{s}({\bf x}_i) = {\bf x}_j}  (n_j + 1)  p\big(\mathfrak{r}_i ({\bf t}), \mathfrak{r}_i ({\bf n} + {\bf e}_j)\big),
\end{split}
\end{equation}
with the usual boundary condition.

\begin{remark} \label{rem:themess}
Note that the recursion given by Ethier and Griffiths in \cite[Corollary~4.2]{Ethier1987} closely 
resembles our recursion~\eqref{recursion} in the case $\Lambda=\delta_0$, $r=\theta/2$, 
up to a missing factor $1/d$ in the last term on the right-hand side. 
This subtle discrepancy can be resolved as follows.

As before, let $({\bf t}, {\bf a})$ denote an unlabelled 
ordered sample of $d$ (ordered) types stored in ${\bf t}$ 
together with an ordered partition ${\bf a}=(A_1, \dots, A_d)$ 
and let $[{\bf t}, {\bf n}]$
be the corresponding sample with $d$ unordered types 
stored in ${\bf t}$ and multiplicity vector
${\bf n}=(n_1, \dots, n_d)$.  
Recall that we have 
\[
p[{\bf t}, {\bf n}]= \frac{n!}{n_1! \cdots n_d!} \frac{d!}{c({\bf t}, {\bf n})} 
 p({\bf t}, {\bf a}),
\]
where $p([{\bf t}, {\bf n}])$ solves Recursion~\eqref{recursion} and $p({\bf t}, {\bf a})$
solves \eqref{real_ims_recursion}. In contrast, let $\langle {\bf t}, {\bf a} \rangle$
denote a sample with $d$ (unordered) types and type partition ${\bf a} = \{A_1, \dots, A_d\}$, 
where the types in the vector ${\bf t} \in (\mathcal{T}_d\slash \sim)_0$ are 
ordered by appearance in the sample (any other deterministic recipe of deriving an order 
on the types from the sample would work equally well).
Then, we have
\[
p[{\bf t}, {\bf n}] = \frac{n!}{n_1! \cdots n_d!} \frac{1}{c({\bf t}, {\bf n})} 
 p\langle{\bf t}, {\bf a}\rangle,
\]
which corresponds to (4.12) in \cite{Ethier1987} and is consistent with 
the displayed equation on p.~86, l.~-13 of \cite{Griffiths1995}, and
\[
 p({\bf t}, {\bf a}) = \frac{1}{d!} p\langle{\bf t}, {\bf a}\rangle.
\]
If one interprets the notation $(T, {\bf n})$ of \cite[Corollary~4.2]{Ethier1987}
as a canonical representative of $\langle{\bf t}, {\bf a}\rangle$, then 
$p\langle{\bf t}, {\bf a}\rangle$ solves recursion (4.4) 
in \cite{Ethier1987} without additional factor $1/d$ in front of the last term.

While all the above recursions yield probability weights (resp.~likelihoods),
for practical purposes it is often easier to multiply \eqref{real_ims_recursion} by $c({\bf
   t},{\bf n})$ and thus derive from \eqref{real_ims_recursion} a
 recursion for $p^0[{\bf t},{\bf n}] := c({\bf t},{\bf n}) p[{\bf
   t},{\bf n}]$.  This is the recursion given by
 \cite[Corollary~1]{Birkner2008}, and it is also the recursion implemented
 by \texttt{genetree}\footnote{Version~9.0, available from 
\url{http://www.stats.ox.ac.uk/~griff/software.html}} 
(for the Kingman case) and
 \texttt{MetaGeneTree}\footnote{Version~0.1.2, available from 
\url{http://metagenetree.sourceforge.net} }.
However, one should be aware that the $p^0[{\bf t},{\bf n}]$ may not always be interpreted as
probability weights (for example consider 
the star-shaped tree 
${\bf t}=\big((1,0),(2,0),\dots,(d,0)\big)$ with $n=d$, 
${\bf n}=(1,\dots,1)$; for $d=22$, with $\Lambda=\delta_0$ and $r=7$, {\tt genetree} yields
$p^0[{\bf t},{\bf n}] \approx 2.26$).
Still, this method can be used to compute maximum likelihood estimators, and the
correct probability can be recovered by dividing by $c({\bf t},{\bf n})$.

\end{remark}

\section{Derivation of importance sampling schemes}
\label{seq:importance}

\subsection{Simulating sample histories}
\label{sseq:iii}

In the sequel, we will always parametrise a sample as an unlabelled
tree with ordered types $({\bf t},{\bf n})$.  Recursion
\eqref{ordered_ims_recursion} 
can be used directly to calculate sampling probabilities for a given
sample configuration $({\bf t},{\bf n})$, noting that the sample
complexity is reduced by each step. However, for practical purposes
this naive approach is only tractable for samples of rather small
complexity due to the huge number of terms involved
(the coefficient matrices of the right-hand sides of \eqref{recursion}, 
\eqref{real_ims_recursion}, 
resp.\ \eqref{ordered_ims_recursion} are substochastic, hence numerical 
stability itself is not an issue). 

One way to deal with this problem is to consider importance sampling using so-called 
(coalescent-){\em histories}. Informally, describing samples via ordered 
types with multiplicities, such a history 
\[
\mathcal{H} = (H_{-\tau +1},\ldots,H_0)
\]
is the chronologically ordered sequence of the $\tau$ (different)
states in $\mathcal{T}^*$ one observes when tracing the coalescent
tree with superimposed mutations from the root to its leaves (see,
e.g., \cite[Steps~(i)--(vii) in Section~3]{Birkner2008}), where
$H_{-\tau+1} = ((0),(1))$ is the root and $H_0 = ({\bf t},{\bf n})$ is the
observed sample. 

A computationally 
efficient way of generating samples is described in \ref{sect:algo}, adapting 
\cite[Algorithm~1]{Birkner2008}. 
Let $\theta = (r,\Lambda)$ be the underlying `parameter' of our model (mutation rate and
Lambda-coalescent).
For a given sample size $n$, this algorithm 
constructs the path of a Markov chain with law $\P_{\theta,n}$ in $\mathcal{T}^*$ terminating 
in a sample 
of size $n$.
Its transition probabilities are
given by 
(as usual, denoting $|{\bf n}'|$ by $n'$)
\begin{equation}\label{likelihood_eq_markov_transition}
(\mathbf{t}',\mathbf{n}') \to (\mathbf{t}'',\mathbf{n}'') = 
\left\{\negthickspace\negmedspace\begin{array}{lll}
           \partial & \text{\rm w.p.\ } \frac{\tilde{q}^{(n)}_{n',n'}}{r_{n'}} & 
           \text{if $n'=n$},  \\
(\mathbf{t}',\mathbf{n}' + l\mathbf{e}_i\big) & \text{\rm w.p.\ } \frac{1}{r_{n'}}
\frac{n'_i}{n'}\tilde{q}^{(n)}_{n',n'+l} & \text{if $n' + l \le n$ ($l \geq 1$)},\\
\big(\mathfrak{a}_i(\mathbf{t}'),\mathbf{n}'\big) & \text{\rm w.p.\ } \frac{r}{r_{n'}} & 
\text{\rm if } n'_i = 1,\\
\big(\mathfrak{e}_{i,j}(\mathbf{t}'),\mathfrak{e}_j(\mathbf{n}'-\mathbf{e}_i)\big) & \text{\rm w.p.\ } \frac{r}{r_{n'}}\frac{1}{d+1}n'_i & \text{\rm if } n'_i > 1 \; (j=1,\dots,d+1) .\\
\end{array}\right.
\end{equation}
Here, $({\bf t'},{\bf n'})$ denotes the current state (with $d$ types), $i$ denotes the
type that is involved in the transition event with $1\leq i \leq d$,
and $$r_{n'} := n' r + \tilde{q}^{(n)}_{n'n'}.$$ The function
$\mathfrak{a}_i({\bf t'})$ attaches a mutation to the type $i$. The
operator $\mathfrak{e}_{i,j}({\bf t'})$ copies type $i$, attaches a
mutation and inserts the resulting type at position $j$ in the vector ${\bf t}$.  The
expression $\mathfrak{e}_j({\bf n'})$

denotes 
the vector
\[
\mathfrak{e}_j({\bf n'})=\mathfrak{e}_j(n'_1,\ldots,n'_d) := (n'_1,\ldots,n'_{j-1},1,n'_j,\dots,n'_d).
\]
Note that for given $({\bf t'},{\bf n}')$ and a type $i \in \{1, \dots, d\}$, it is in principle possible that 
the values $(\mathbf{t}'',\mathbf{n}'') = \big(\mathfrak{e}_{i,j}(\mathbf{t'}),\mathfrak{e}_j(\mathbf{n'}-\mathbf{e}_i)\big)$ 
are identical for several choices of $j$. The number of such $j$ equals 
\begin{equation*}
\text{nio}(\mathbf{t}',\mathbf{n}',i) := 
1+\#\Big\{ 1 \leq k \leq d : 
\begin{array}{l}
n_k=1 \;  \text{and type $k$ differs from type $i$} \\
\text{by exactly one unique mutation}
\end{array}
\Big\} 
\end{equation*}
(`nio' stands for `number of immediate offspring'). 
The $\tilde{q}^{(n)}_{k,l}$ are the transition 
rates of the time-reversed block counting process of the underlying Lambda-coalescent, 
see \ref{sect:algo}. Finally, $\partial$ denotes a cemetery state. Once reached, the sample
has been generated and is given by the penultimate state $({\bf t'},{\bf n'})$ (from which the cemetery state
had been reached).

It is straightforward to read off the transition probabilities 
from (\ref{likelihood_eq_markov_transition}), observe in particular that
\[
\mathbb{P}_{\theta,n}\big(H_\ell=(\mathbf{t}'',\mathbf{n}'') \mid H_{\ell-1}=(\mathbf{t}',\mathbf{n}')\big) = \frac{r}{r_{n'}}\frac{\text{nio}(\mathbf{t}',\mathbf{n}',i)}{d+1}n'_i
\]
if $(\mathbf{t}'',\mathbf{n}'')=
\big(\mathfrak{e}_{i,j}(\mathbf{t}'),\mathfrak{e}_j(\mathbf{n}'-\mathbf{e}_i)\big)$. 
(No such ambiguities arise for the transitions in the first three lines 
of (\ref{likelihood_eq_markov_transition}).)
\smallskip

We have 
\begin{equation}\label{importance_eq_sum_histories}
p({\bf t},{\bf n}) = \sum_{\mathcal{H}:H_0 = ({\bf t},{\bf n})} \P_{\theta,n}\{\mathcal{H}\},
\end{equation}
were the sum extends over all different histories (of possibly different lengths) 
with terminal state $H_0 = ({\bf t},{\bf n})$. 
Recursion~\eqref{ordered_ims_recursion} is 
just a way to enumerate all consistent histories and compute the sum in~\eqref{importance_eq_sum_histories}. An obvious `naive' approach to estimating 
$p({\bf t},{\bf n})$ is via direct Monte Carlo: Indeed, 
\begin{equation} \label{eq:MCestnaive}
\frac{1}{M} \sum_{i=1}^M \1_{\{(\H^{(i)})_0 =({\bf t},{\bf n})\}}, 
\end{equation}
where $\mathcal{H}^{(1)},\ldots,\mathcal{H}^{(M)}$ are independent
samples from $\P_{\theta,n}(\cdot)$, is an unbiased estimator of $p({\bf t},{\bf n})$. 
Unfortunately, even for small sample sizes, the variance of \eqref{eq:MCestnaive} 
is typically too high for \eqref{eq:MCestnaive} to be of practical value, 
since $p({\bf t},{\bf n})$ can easily be of the order $10^{-15}$ 
(see Table~\ref{tab_real_performance} for examples).

\subsection{Importance sampling and the optimal proposal distribution}
\label{ssn:opt}

Importance sampling is a well-known approach to reducing the variance of estimators of the form 
\eqref{eq:MCestnaive}. In the following, we will think of a \emph{fixed} 
sample size $n$ and will thus lighten notation by denoting $\mathbb{P}_\theta :=
\mathbb{P}_{\theta,n}$. 
Consider a \emph{proposal distribution} $\Q_\theta(\cdot)$ on the space of
histories satisfying 
\begin{equation}\label{importance_eq_support_condition}
\P_\theta\Big\vert_{\big\{ H_0 = ({\bf t}, {\bf n}) \big\}} \ll \Q_\theta,
\end{equation}
and use it to rewrite equation~\eqref{importance_eq_sum_histories} as
\begin{equation}\label{eq_proposal_exact}
p({\bf t},{\bf n})  = \sum_{\mathcal{H}:H_0 = ({\bf t},{\bf n})} \frac{\P_\theta\big(\mathcal{H}\big)}{\Q_\theta\big(\mathcal{H}\big)}\Q_\theta(\mathcal{H}).
\end{equation}
This shows that 
\begin{equation}\label{eq_proposal_estimator}
\frac{1}{M} \sum_{i=1}^M \1_{\{(\H^{(i)})_0 =({\bf t},{\bf n})\}} \frac{d\P_\theta} {d\Q_\theta}(\mathcal{H}^{(i)}) = \frac{1}{M}\sum_{i=1}^M w(\mathcal{H}^{(i)}),
\end{equation}
where $\mathcal{H}^{(1)},\ldots,\mathcal{H}^{(M)}$ are independent
samples from $\Q_\theta(\cdot)$, is also an unbiased (and consistent as $M\to\infty$) 
estimator of $p({\bf t},{\bf n})$. Denote by
\begin{equation}\label{eq:weights}
w(\mathcal{H}^{(i)}) :=  \begin{cases}
\frac{d\P_\theta} {d\Q_\theta}(\mathcal{H}^{(i)}) &\text{if}\; \big(\H^{(i)}\big)_0 =({\bf t},{\bf n})\\
0 & \text{otherwise}
\end{cases}
\end{equation}
the \emph{importance sampling weight} or \emph{IS weight} of the history
$\mathcal{H}^{(i)}$. Our goal now is to
derive proposal distributions
for which the variance of the
estimator~\eqref{eq_proposal_estimator} is small. 

The optimal proposal distribution $\Q^*_\theta$, for which this
variance vanishes, is given by
\begin{equation}\label{importance_eq_observation_optimal}
\Q^*_\theta(\mathcal{H}) = \P_\theta\big\{\mathcal{H} \big\vert H_0 = ({\bf t},{\bf n})\big\},
\end{equation}
the conditional distribution on the histories given the observed data.
Under $\Q^*_\theta$, the importance weight $w(\H)$ equals $p({\bf
t},{\bf n})$ for all histories $\H$ compatible with the data. 
Hence, the (consistent) 
estimator~\eqref{eq_proposal_estimator} becomes deterministic, and its 
variance is thus zero.

{
Since for a given $\mathcal{H}$, $\P_\theta(\mathcal{H})$ is 
straightforward to evaluate, we}
see from \eqref{importance_eq_observation_optimal} that explicit knowledge of 
the optimal proposal distribution is equivalent to knowing $p({\bf
t},{\bf n})$, so not surprisingly $\Q^*_\theta$  cannot be
given explicitly in general. 
{\coleins This also applies to the Kingman case except 
for so-called parent-independent mutation models, as observed in 
\cite{Stephens2000}.} 

It is natural to consider proposal distributions $\Q_\theta$ under which the time-reversal 
of the history is a Markov chain starting from the observed configuration 
$({\bf t},{\bf n})$ and ending at the root $((0), (1))$, thus guaranteeing 
that the weights in \eqref{eq:weights} are strictly positive. 
Indeed, by elementary properties of Markov chains, $\Q^*_\theta$ has this 
property.

Let 
\begin{equation}
\label{eq:Greenfunction}
G^{(n)}({\bf t}',{\bf n}') = \E_{\theta,n}\Big[ \sum_{\ell=-\tau+1}^0 \1_{\{({\bf t}',{\bf n}')\}}(H_\ell) 
\Big]
\end{equation} 
denote the associated Green function, that is the expected
time the Markov chain with transition probabilities
\eqref{likelihood_eq_markov_transition} (for samples of size $n$) spends in the state
$({\bf t}',{\bf n}')$. Note that by the special structure of the transitions in 
\eqref{likelihood_eq_markov_transition} which increase the sample complexity in 
each step, we in fact have {\coleins (for $n \geq |{\bf n}'|$)}
\begin{equation} 
\label{eq:relGandhitting}
G^{(n)}({\bf t}',{\bf n}') = \P_{\theta,n}\big\{ \exists \, \ell \, : H_\ell = ({\bf t}',{\bf n}') \big\}.
\end{equation} 

\begin{lemma} \label{lem:Gandp}
\coleins 
For $({\bf t},{\bf n})$ with $|{\bf n}|=n$ we have 
\begin{equation}
\label{eq:relpandG1}
p({\bf t},{\bf n}) = G^{(n)}({\bf t},{\bf n}) \frac{\tilde{q}^{(n)}_n}{nr+\tilde{q}^{(n)}_n}. 
\end{equation}
More generally, for $({\bf t}',{\bf n}')$ with $|{\bf n}'|=n' < n$, 
\begin{equation}
\label{eq:relpandG2}
p({\bf t}',{\bf n}') = \frac{G^{(n)}({\bf t}',{\bf n}')}{g(n, n') (n'r+\tilde{q}^{(n')}_{n'})}, 
\end{equation}
where $g(n, n')$ is the Green function of the block counting process of the 
underlying Lambda-coalescent, see \eqref{blue}. 
\end{lemma}

\begin{proof}
 \eqref{eq:relpandG1} follows from \eqref{eq:relGandhitting} and the
 fact that under $\P_{\theta,n}$ when the chain is currently in a
 state with sample size $n$ it terminates with probability 
$(\tilde{q}^{(n)}_n)/(nr+\tilde{q}^{(n)}_n)$, see the second case in Step~2 of 
Algorithm~1 in \ref{sect:algo}.

We see from \eqref{eq:reversedjumprate1} and \eqref{likelihood_eq_markov_transition} 
that the probabilities for transitions between states with at most $n'$ samples 
agree under $\P_{\theta,n}$ and $\P_{\theta,n'}$ except for terms involving 
$q^{(n')}_{\cdot,\cdot}$ resp.\ $q^{(n)}_{\cdot,\cdot}$. Using 
\eqref{eq:eq:comparereversedhittingprob} on the product of these terms yields 
\begin{equation} 
\P_{\theta,n}\big\{ \exists \, \ell \, : H_\ell = ({\bf t}',{\bf n}') \big\} 
= \frac{g(n,n')}{g(n',n')} 
\P_{\theta,n'}\big\{ \exists \, \ell \, : H_\ell = ({\bf t}',{\bf n}') \big\}.
\end{equation}
Using \eqref{eq:relGandhitting}, \eqref{eq:relpandG1} and 
observing $g(n',n')=1/(-q_{n'n'})=1/\tilde{q}^{(n')}_{n'}$, we obtain
\begin{align*}
p({\bf t}',{\bf n}') &= \P_{\theta,n'} \big\{ \exists \, \ell \, : H_\ell = ({\bf t}',{\bf n}') \big\} 
\frac{\tilde{q}^{(n')}_{n'}}{n'r+\tilde{q}^{(n')}_{n'}} \\
     &= G^{(n')}({\bf t}',{\bf n}') \frac{\tilde{q}^{(n')}_{n'}}{n'r+\tilde{q}^{(n')}_{n'}} \\
     &= G^{(n)}({\bf t}',{\bf n}') \frac{g(n',n')}{g(n,n')} \frac{\tilde{q}^{(n')}_{n'}}{n'r+\tilde{q}^{(n')}_{n'}} \\
     &= \frac{G^{(n)}({\bf t}',{\bf n}')}{g(n, n') (n'r+\tilde{q}^{(n')}_{n'})},
\end{align*}
which is \eqref{eq:relpandG2}.

\end{proof}

\begin{lemma}\label{importance_lemma_optimal_proposal_transition} 
The time-reversed history 
$\H$ under $\Q_\theta^*$ is a Markov chain started in
$H_0 = ({\bf t},{\bf n})$ 
with transition probabilities given by
\begin{align}\label{importance_eq_optimal_proposal_transition}
\Q_\theta^* \big( & H_{\ell-1}=(\mathbf{t}',\mathbf{n}')  \mid H_{\ell}=(\mathbf{t}'',\mathbf{n}'') \big) \notag \\ 
& = \,  \frac{G^{(n)}({\bf t}',{\bf n}')}{G^{(n)}({\bf t}'',{\bf n}'')}
\mathbb{P}_{\theta,n}\big(H_\ell=(\mathbf{t}'',\mathbf{n}'') \mid H_{\ell-1}=(\mathbf{t}',\mathbf{n}')\big), 
\end{align}
where the transition matrix under $\mathbb{P}_{\theta,n}$ is described in 
Section~\ref{sseq:iii} and $n=|{\bf n}|$. 
The chain is absorbed in the root $((0), (1))$. 
The transition probability in 
\eqref{importance_eq_optimal_proposal_transition} is independent of 
$n$ (provided $n \geq |{\bf n}''|$).  
\end{lemma}

\begin{proof}[Sketch of proof]
The optimal proposal distribution is the distribution of histories simulated
with Algorithm~1 in \ref{sect:algo} with transition \eqref{likelihood_eq_markov_transition} 
conditioned on observing $({\bf t},{\bf n})$ as the penultimate state before 
hitting the `cemetery' $\partial$. 
Nagasawa's formula can thus be applied to obtain the transition
probabilities~\eqref{importance_eq_optimal_proposal_transition} of the
time-reversed chain 
(see e.g.\ \cite{RW87},  Sect.~III.42).
\end{proof}

The fact that \eqref{importance_eq_optimal_proposal_transition} does 
not depend on the target sample size $n$ stems from the consistency 
properties of Lambda-coalescents and is made explicit in 
the following Remark~\ref{rem:Gandp}.

\begin{remark}\label{rem:Gandp}
By Lemma~\ref{lem:Gandp},
we may express the transition probabilities 
of the time-reversed history under $\Q_\theta^*$ explicitly via 
$p$ as follows 
(with notation as above):
\begin{equation}
\label{importance_eq_optimal_proposal_transition_quot_p}
\begin{split}
\Q_\theta^* &\big(H_{\ell-1}=(\mathbf{t}',\mathbf{n}') \mid H_{\ell}=(\mathbf{t}'',\mathbf{n}'') \big)\\
&= 
       \frac{p({\bf t}',{\bf n}')}{p({\bf t}'',{\bf n}'')} 
       \begin{cases}
\frac{1}{r_{n''}}\frac{n''_i-l}{n''-l}{q}_{n'',n''-l} & \text{if } (\mathbf{t}',\mathbf{n}') = \big(\mathbf{t}'',\mathbf{n}'' -l\mathbf{e}_i\big), \\
\frac{r}{r_{n'}} & \text{if } (\mathbf{t}',\mathbf{n}') = (\mathfrak{s}_i(\mathbf{t}''),\mathbf{n}''), \\
\frac{r}{r_{n'}}\frac{\text{nio}(\mathbf{t}',\mathbf{n}',i)}{d} (n''_j+1) & \text{if }(\mathbf{t}',\mathbf{n}') = \big(\mathfrak{r}_i(\mathbf{t}''),\mathfrak{r}_i(\mathbf{n}''+\mathbf{e}_j)\big), \\
0 & \text{otherwise}.
\end{cases}
\end{split}
\end{equation}
\end{remark}
\begin{proof} For the last three lines in the right-hand side of 
\eqref{importance_eq_optimal_proposal_transition_quot_p} note that 
by \eqref{eq:relpandG2}, 
${G^{(n)}({\bf t}',{\bf n}')}/{G^{(n)}({\bf t}'',{\bf n}'')} 
= p({\bf t}',{\bf n}')/p({\bf t}'',{\bf n}'')$ if $|{\bf n}'|=|{\bf n}''|$, 
for the first line observe 
\begin{align*}
&\frac{g(n, n')(n'r+\tilde{q}^{(n')}_{n'})}{
g(n, n'')(n''r+\tilde{q}^{(n'')}_{n''})} 
\frac{1}{r_{n'}}\frac{n''_i-l}{n''-l}\tilde{q}^{(n)}_{n''-l,n''} \\ 
& \hspace{2em} = \frac{1}{r_{n''}} \frac{n''_i-l}{n''-l} 
 \frac{g(n, n') \tilde{q}^{(n)}_{n''-l,n''}}{g(n, n'')} 
= \frac{n''_i-l}{n''-l} \frac{{q}_{n'',n''-l}}{r_{n''}}
\end{align*} 
if $n'=n''-l$ (see \eqref{eq:reversedjumprates}).
\end{proof}

\begin{remark}
\label{rem:SDequiv} 
Lemma~\ref{importance_lemma_optimal_proposal_transition} 
can be seen as 
a starting point for importance sampling: 
Any (quite possibly heuristic) approximation of 
${G^{(n)}({\bf t}'',{\bf n}'')}/{G^{(n)}({\bf t}',{\bf n}')}$ 
leads via (\ref{importance_eq_optimal_proposal_transition}) to an 
approximation of $\Q_\theta^*$ which can be used as 
a proposal  distribution. This is the `$\Lambda$-coalescent equivalent' 
of Stephens \& Donnelly's \cite[Thm.~1]{Stephens2000} observation that the 
optimal distribution in the Kingman context can be characterised in terms 
of the conditional distribution of an $(n+1)$-st sample given 
the types of $n$ samples.
\end{remark}

\subsection{Importance Sampling Schemes}
\label{importance_sec_proposal_schemes}

We have shown that the optimal proposal distribution $\Q_\theta^*(\cdot)$ is a
Markov chain and derived expressions for the transition probabilities in terms
of the Green function
(\ref{eq:Greenfunction}). 
Since recursive evaluation of the
Green function is equivalent to evaluating the likelihood, this is 
more of theoretical than direct practical value.

Still, in the remaining sections we
will present several  
proposal distributions based on Markov chains that approximate
the optimal proposal distribution in reasonable ways so that the variance of
the estimator~\eqref{eq_proposal_estimator} is small. 
We discuss separately situations 
in which the proposal distribution does not take any information about `genealogical distance' between types (which is in principle provided by the
IMS model) into account, and situations in which at least some of this information is retained.

\subsubsection{Importance sampling schemes without regard of `genealogical distance' between types}
\label{importance_sec_griffiths_tavare}
\ \\

{
\noindent
{\em Griffiths \& Tavar\'e's scheme for Lambda-coalescents.}
Griffiths and Tavar\'e in \cite{Griffiths1994} introduced a Monte
Carlo method to estimate the likelihoods of mutation rates under
Kingman's coalescent. This method was generalised in \cite{Birkner2008} to the
multiple merger case
and can be interpreted, as observed by Felsenstein et.~al.~\cite{Felsenstein1999}, also as an importance sampling scheme.}

Indeed, it is easy to derive a proposal distribution from recursion~\eqref{ordered_ims_recursion}, recovering the scheme derived in \cite{Birkner2008}.
For a given configuration $(\mathbf{t}, {\bf n})$ 
with complexity greater than 1 (i.e.\ excluding the root), define (with the usual convention $n=|{\bf n}|, r_n=\lambda_n + nr$)
\begin{align}
\label{eq:defftnallg} 
f_\theta(\mathbf{t},{\bf n}) &:= 
     \frac{1}{r_n} 
                 \Bigg(\sum_{i:n_i\geq 2} \sum_{k=2}^{n_i} {n \choose k} \lambda_{n,k}\frac{n_i-k+1}{n-k+1} 
             + \sum_{i:n_i=1, x_{i0} \text{unique}, \atop \mathfrak{s}({\bf x}_i) \neq {\bf x}_j \forall j} r \notag \\
           &\quad \quad \quad \quad + \frac rd \sum_{i:n_i=1,\atop x_{i0} \text{unique}} 
                 \sum_{j:\mathfrak{s}({\bf x}_i) = {\bf x}_j}  (n_j + 1) \Bigg),
\end{align} 
and put $f_\theta((0), (1)) :=1$ for the root.

\begin{definition}[Proposal distribution $\QGT$]
\label{importance_def_griffiths_tavare}
We denote by $\QGT$ the law of a Markov chain on the space of histories with transitions, given a state
$({\bf t},{\bf n})$, as follows:
\begin{equation}\label{eq_griffiths_tavare_transistion_rates}
({\bf t} ,{\bf n}) \to
\begin{cases}
                      \big(\mathfrak{s}_i({\bf t}),{\bf n}\big)
                                      &\text{w.p. }\frac{r}{r_n f_\theta({\bf t},{\bf n})}\,\,
                                   \text{\scriptsize if $\,n_i=1, t_{i,0}$ unique $\mathfrak{s}_i({\bf t}_i)) \neq {\bf t}_j \forall j$},\\
\big(\mathfrak{r}_i({\bf t}),\mathfrak{r}_i({\bf n} + {\bf e}_j)\big) &\text{w.p. }\frac {r(n_j + 1)}{r_n 
                                      f_\theta({\bf t},{\bf n})}\,\,\text{\scriptsize if $\,n_i=1, t_{i,0}$ unique, $\mathfrak{s}({\bf t}_i) 
                                                    = {\bf t}_j$},\\
\big({\bf t},{\bf n} - (k-1){\bf e}_i\big) &\text{w.p. } \frac{1}{r_n f_\theta({\bf t},{\bf n})} {n \choose k}
                                        \lambda_{n,k}\frac{n_i-k+1}{n-k+1}\,\,\text{\scriptsize if $\,2\leq k\leq n_i$}.
                     \end{cases}
\end{equation}
\end{definition}
To see why this yields a suitable Monte Carlo estimate, let
$\tau$ denote the random number of steps that our Markov chain performs until it hits the root configuration. Then, 
a simple calculation shows (see e.g.~\cite[Lemma~6]{Birkner2008}) that 
we may write 
\begin{equation}\label{eq_griffiths_tavare_monte_carlo}
p({\bf t},{\bf n}) = \mathbb{E}_{({\bf t},{\bf n})} \prod^{\tau-1}_{i=0} f_\theta(H_{i}),
\end{equation}
where the expectation is taken with
respect to
$\QGT$ started in $({\bf t},{\bf n})$.

\begin{remark}
This Monte Carlo method can be interpreted as an importance sampling scheme by choosing
the proposal weights $w(\H)$ according to
\[
w(\H) = \prod^{\tau-1}_{i=0} f_\theta(H_{i}) 
      = \frac{d\mathbb{P}_\theta}{d\QGT}(\H).
\]
Note that this method is a special case of general Monte Carlo
methods for systems of linear equations with non-negative coefficients.
It is therefore referred to as the `canonical candidate' by 
\cite{Griffiths1997a} in the Kingman case and will serve us as 
a benchmark in Section~\ref{importance_sec_performance}.
\end{remark}

\noindent
{\em Stephens \& Donnelly's scheme for Lambda-coalescents.}
\label{importance_sec_stephens_donnelly}
Stephens and Donnelly\ \cite{Stephens2000} motivate and study a proposal distribution 
in a general finitely many alleles model under
Kingman's coalescent.
One can efficiently sample from their proposal
distribution by choosing an individual from the current sample 
uniformly at
random and then decide on the transition for the type of this individual. 
This is indeed optimal in the case of parent-independent mutations 
(see \cite{Stephens2000}, Prop.~1). 
This procedure is adapted by Stephens and Donnelly to the infinitely many sites
model in their Section~5.5. 
Here, not all types are 
eligible for a transition -- only those whose multiplicity is at least two (which will then merge)
or whose outmost mutation, say $x_{k0}$, is unique. 
Denote the number of eligible individuals of a configuration $({\bf t},{\bf n})$ by
\[
z({\bf t},{\bf n}) := \sum_{i: n_i=1 \text{ and } x_{i0} \text{ unique } \atop \text{or}\,n_i 
\geq 2} n_i.
\]
Under Kingman's coalescent, choosing (uniformly) an eligible individual is equivalent to proposing a transition step. Either a singleton is chosen, where the only possible most
recent event is the removal of the outmost mutation, or an individual with a
type that occurs at least twice in the sample is chosen leading to a binary merger. 

To adapt this approach to the Lambda-case note that when choosing an eligible singleton
type the proposed step is unambiguous as in the previous case. However, the
proposal needs additional information if a type with multiplicity greater than two is
chosen, since then typically various multiple mergers of ancestral lines can occur. A natural approach
to this problem is to choose the size of a merger with a probability
proportional to the rates of the block counting process of the
$\Lambda$-coalescent (see e.g.\ \cite[Section~7]{Birkner2008}). Based on this
idea we introduce the following proposal distribution.
\begin{definition}[Proposal distribution $\QSD$]
The proposal distribution 
$\Q_\theta^{\Lambda-\text{SD}}$ is the distribution of the Markov chain on the space of histories 
performing the transitions
\begin{equation}
({\bf t},{\bf n}) \to \begin{cases}
\big(\mathfrak{s}_k({\bf t}),{\bf n}\big) &\text{w.p. }\frac{1}{z({\bf t},{\bf n})}\,
\text{\scriptsize if $k:n_k=1, x_{k0}$ unique $\mathfrak{s}_k({\bf x}_k)) \neq {\bf x}_j \forall j$}\\
\big(\mathfrak{r}_k({\bf t},\mathfrak{r}_k({\bf n} + {\bf e}_j))
&\text{w.p. }\frac{1}{z({\bf t},{\bf n})}\,\text{\scriptsize if $k:n_k=1, x_{k0}$ unique}\\
\big({\bf t},{\bf n} - (k-1){\bf e}_i\big) &\text{w.p. }
\frac{p(k)n_i}{z({\bf t},{\bf n})}\,\text{\scriptsize if $2\leq k\leq n_i$},
\end{cases}
\end{equation}
where 
\begin{equation}
p(k) = p_i^{({\bf t},{\bf n})}(k) = \frac{q_{n,n-k+1}}{\sum_{l=2}^{n_i} q_{n,n-l+1}}
\end{equation}
for $n_i \geq 2$ is the probability derived from the block counting process (see \ref{sect:algo})
that in the most recent merging event $k$ lineages coalesce.
\end{definition}

\begin{remark}[On optimality in the infinite alleles model]
Hobolth et.\ al.\ showed in \cite{Hobolth2008} that the proposal distribution
of Stephens and Donnelly in the Kingman case 
is the optimal proposal distribution in the infinitely many alleles  model (IMA),
which is the prime example of a parent-independent mutation model.
A crucial step in the proof is Ewens' sampling formula, which provides an explicit
expression for the probability of a sample in the IMA. Since such an explicit formula 
is (at present) not available in the IMA for Lambda-coalescents, we may express the optimal 
proposal distribution only implicitly via a recursion of M\"ohle \cite{Moehle2006}.
Indeed, let ${\bf c}=(c_1, \dots, c_k, 0, \dots) \in (\N_0)^\infty$ denote an allelic partition of a sample in the IMA,
that is, $c_i$ is the number of types that occur $i$ times in the sample.
Then, the sampling probability $q({\bf c})$ satisfies
\begin{equation}
\label{eq.moehlerec}
q({\bf c}) = \frac{nr}{r_n} q({\bf c-e_1}) + \sum_{i=1}^{n-1} \frac{{n \choose i+1}\lambda_{n, i+1}}{r_n} \sum_{j=1}^{n-1} \frac{j(c_j+1)}{n-i}
q({\bf c+e_j-e_{i+j}})
\end{equation}
with $n=\sum_i i c_i$. 
The boundary condition is $q((1, 0, \dots)) =1$ and we set $q({\bf c})=0$ if any entry in ${\bf c}$ is negative.
Further, let $\phi: ({\bf t},{\bf n}) \mapsto {\bf c}$ be the function which maps a sample $({\bf t},{\bf n})$ in the infinitely many sites model, which we
think of being generated by Algorithm~1 (the `$\Lambda$-Ethier Griffiths Urn'), to the corresponding allelic partition ${\bf c}$
in the infinite
alleles model (i.e.~where $c_i=\#\{\mbox{types $k$ with $n_k=i$} \}, i=1, 2, \dots$).
Let $P^{\Lambda-\mbox{\tiny EGU}}$ be the image measure of the sample distribution under $\phi$.
Then, using conditional probabilities, the optimal sampling distribution $\Q^{*, \mbox{\tiny IMA}}_\theta$ in the infinite alleles model has transitions
\begin{equation}
\label{eq:trans1} 
\Q^{*, \mbox{\tiny IMA}}_\theta({\bf c}' | {\bf c}) = P^{\Lambda-\mbox{\tiny EGU}}({\bf c}| {\bf c}') \frac{q({\bf c}')}{q({\bf c})},
\end{equation}
where ${\bf c}'={\bf c-e_1}$ or ${\bf c}'={\bf c+e_j-e_{j+i}}$ for some 
$i, j \in \{1, \dots, n-1\}$ are the only possible transitions. 
Unfortunately, unlike the Kingman case (where the Ewens sampling formula is at hand), there is no explicit closed solution to the recursion \eqref{eq.moehlerec}. However, for a given sample size, the solution to \eqref{eq.moehlerec} could be precomputed and stored in a large database
(this is much easier than in the case of our original recursion for $({\bf t},{\bf n})$  since no explicit type 
configurations ${\bf t}$ need to be stored). This would yield a perfect sampler (given a suitable database) for the
infinite alleles model in the $\Lambda$-case. Still, since a lot of information is lost via our map $\phi$, it is unclear if
this would lead to a good sampler for the infinitely many sites case.
We refer to \cite{Moehle2006} and
\cite{Dong2006} for a more thorough investigation of the infinitely many alleles
model in the $\Lambda$-coalescent case.
\end{remark}

\subsubsection{Schemes based on compressed genetrees}
\label{sseq:huw} 
In the following, we will abbreviate the transition matrix 
of the time-reversed history under 
$\Q_\theta^*$ by 
\[
\Q_\theta^*\big((\mathbf{t}'',\mathbf{n}'') 
\rightarrow (\mathbf{t}',\mathbf{n}')\big) := 
\Q_\theta^*\big(H_{\ell-1}=(\mathbf{t}',\mathbf{n}') 
\mid H_{\ell} = (\mathbf{t}'',\mathbf{n}'')\big)  
\]
for any $(\mathbf{t}',\mathbf{n}'), (\mathbf{t}'',\mathbf{n}'') 
\in \mathcal{T}^*$, 
which is well-defined irrespective of the `target' sample size 
$n$ appearing in $\P_{\theta,n}$ 
(see Lemma~\ref{importance_lemma_optimal_proposal_transition}).

In this section, our goal is to derive 
proposal distributions for
the infinitely many sites model, where at least partial information
about the structure of the type configuration 
${\bf t}$ is retained.

In the simplest case the idea (due to \cite{Hobolth2008}) is to 
subsequently focus on a single mutation in the genetree 
and then to consider the corresponding  ``compressed'' genetree, 
in which this is indeed the only mutation at all.
For such a simple compressed tree, the optimal transition probabilities 
can be computed explicitly (at least numerically). 
Summing over the mutations, these probabilities are then composed to a proposal for the original tree. 

This approach will be explained and extended to the Lambda-coalescent in the next subsection. 
After that, we will show how to extend this framework to 
retain more 
information about the tree, in particular
taking pairs of mutations (and potentially even more) into account.\\

\noindent
{\bf Hobolth, Uyenoyama \& Wiuf's Scheme for Lambda-coalescents.}
\label{importance_sec_hobolth}
{\colzwo
Let $({\bf t, n})$ be a sample with ordered types. 
Since we will consider individual mutations, for the purposes of 
this section, we think of a fixed representative under the mutation 
relabelling relation $\sim$ from Section~\ref{ssn:genetree}. 

Pick a segregating site, say $s' \in \{1, \dots, s\}$. We first introduce the `compressed genetree' of $[{\bf t, n}]$ 
with regard to the mutation 
at the segregating site $s'$.
}
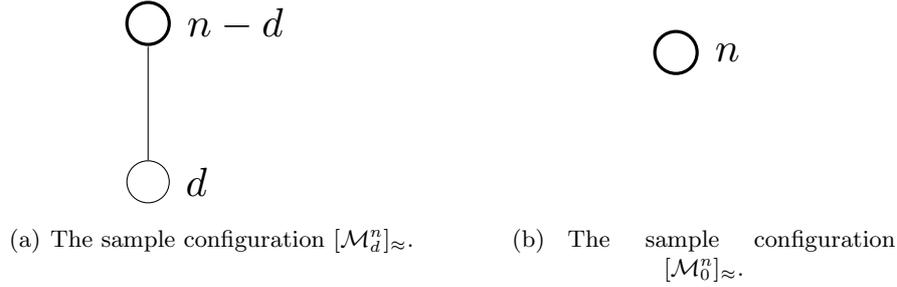
\begin{figure}
\begin{center}
\subfigure[The sample configuration $\lbrack\M^n_d\rbrack_\approx$.]{
\hspace{1.3cm}
\begin{tikzpicture}[scale=1.4, transform shape,inner sep=4pt]
\node[very thick, draw,circle,label=right:{$n-d$}] (root) {}
child {node[draw,circle,label=right:{$d$}] (m1) {}};
\end{tikzpicture}
\hspace{1.3cm}
}
\hspace{1cm}
\subfigure[The sample configuration $\lbrack\M^n_0\rbrack_\approx$.]{
\hspace{1.6cm}
\begin{tikzpicture}[scale=1.4, transform shape,inner sep=4pt]
\node[very thick, draw,circle,label=right:{$n$}] (root) {}
child {edge from parent[draw=none]};
\end{tikzpicture}
\hspace{1.6cm}
}
\caption{The sample configurations $\lbrack\M^n_d\rbrack_\approx$ and
$\lbrack\M^n_0\rbrack_\approx$. The sample has one segregating site respectively
no segregating site.}
\label{importance_fig_one_segregating}
\end{center}
\end{figure}
Denote by 
\[
d(s') = d\big(s',(\mathbf{t},\mathbf{n})\big) = 
\sum_{i \, : \,  \mbox{\tiny type $i$ carries} \atop \mbox{\tiny a mutation at $s'$}} n_i
\]
the number of individuals in the sample bearing a
mutation at {\colzwo the segregating} site {\colzwo$s'$}. 
Let
\begin{equation}
\M^n_d := 
       \Big( \big( (0),(1,0)\big), (n-d,d)\Big), \quad n \in \N, \, d\in \{0, \dots, n\},
\end{equation}
be the genetree where $d$ individuals bear a mutation and $n-d$ do not.
Note that 
\begin{equation}
\M^n_0 := \Big( \big( (0) \big) , (n)\Big)
\end{equation}
is the configuration where all $n$ individuals share the same type. 
See Figure~\ref{importance_fig_one_segregating}. 

\begin{definition}[compressed genetree]
Let $({\bf t, n})$ be a sample of size $n$ 
with $s\ge 1$ segregating sites. Let $s' \in \{1, \dots, s\}$. 
Then, we define the `compressed genetree'
$({\bf t, n})(s')$ with respect to the segregating site $s'$ as
\[
({\bf t, n})(s') := \M^n_{d(s')} = 
\Big( \big( (0),(1,0)\big) , (n-d(s'),d(s'))\Big),
\]
where $d(s')$ is the number of individuals carrying mutation $s'$ in the sample. 
\end{definition}

We now explain how to derive from the optimal proposal distribution for the corresponding compressed trees
a proposal distribution for the original genetree. To this end, fix $s' \in \{1, \dots, s\}$ and let $p_\theta(n,d)$ be the
probability that the most recent mutation in 
$({\bf t, n})$ affected an individual out of the $d=d(s')$
individuals exhibiting a mutation at segregating site $s'$, that is
\begin{equation}
p_\theta(n,d) = \begin{cases}
\sum_{l=1}^{d-1} \Q_\theta^* (\mathcal{M}^{n}_{d} \to \mathcal{M}^{n-l}_{d-l}) & \text{if}\; d>1,\\
\Q_\theta^* (\mathcal{M}^{n}_{1} \to \mathcal{M}^{n}_{0}) & \text{if}\; d=1.
\end{cases}
\end{equation}
In the first case the most recent event was a merger of any size involving
individuals bearing the mutation, whereas in the second case the last event was
the origin of the mutation. Further, define
\begin{equation}\label{defuthetasi}
u_\theta^{(s')}(i) := \begin{cases}
                   p_\theta (n,d_{s'}) \frac{n_i}{d_{s'}} & \text{if}\; \mbox{ type $i$ carries a mutation at $s'$} \\ 
                   (1-p_\theta (n,d_{s'})) \frac{n_i}{n-d_{s'}} & \text{if}\;\mbox{ type $i$ does not carry a mutation at $s'$.}
                  \end{cases}
\end{equation}
Note that $n_i/d$ is the fraction of genes of type $i$ 
among those genes carrying a mutation at segregating site $s'$, thus 
$u_\theta^{(s')}(i)$ would be the exact probability that the most recent 
event in the history involves type $i$ if $s'$ were the only segregating 
site.

\begin{definition}[Eligible types]
Let $({\bf t, n})$ be a genetree with $d$ types. We say that the $k$-th type, where $k \in  \{1, \dots, d\}$, is {\em eligible for transition} 
(or short: {\em eligible}),
if either $n_k \ge 2$ or $n_k=1$ and $x_{k0}$ is unique. 
\end{definition}

We are now ready to state a Lambda-coalescent extension of the \cite{Hobolth2008} proposal distribution (for $\Lambda=\delta_0$, it 
agrees with that from \cite{Hobolth2008}).

\begin{definition}[Proposal distribution 
$\QSingleHUW$]
\label{importance_def_single_hobolth}
We denote by $\QSingleHUW$ the law of a Markov chain on the space of (time-reversed) histories $\H$, starting from 
samples of size $n$, 
if its transition probabilities from a state $({\bf t}',{\bf n}')$ can be described as follows:
\begin{itemize}
\item Pick a type, say $k$, from the set of eligible types of $({\bf t}',{\bf n}')$ at probability
     \[
\frac{
\sum_{s'=1}^s u_\theta^{(s')}(k)
}
{\sum_{\mbox{\tiny $i$ eligible}} \sum_{s'=1}^s u_\theta^{(s')}(i).
}
     \]
\item If the multiplicity of the chosen type $k$ is one remove the outmost mutation.
\item If the multiplicity $n'_k$ is larger than one, perform a merger inside this
group. The size of the merger is determined as follows: 
\begin{itemize}
\item If type $k$ does not bear a mutation, then, an $l+1$ merger, for $1 \le l < n_k$, happens with probability 
proportional to $\Q_\theta^* (\mathcal{M}^{n'}_0 \to \mathcal{M}^{n'-l}_0)$, 
where  
$\Q_\theta^* (\mathcal{M}^{n'}_0 \to \mathcal{M}^1_0) = 
g(n,n')q_{n'1}/G^{(n)}(\mathcal{M}^{n'}_0)$ is the probability of jumping to the terminal state.
\item If type $k$ bears at least one mutation, let $s'$ be the segregating site corresponding to its outmost mutation $x_{k0}$. 
Let $d(s')$ be the number of individuals in the sample bearing a mutation at this segregating site.
Then, an $l+1$ merger, for $1 \le l < n'_k$, happens with probability 
proportional to $\Q_\theta^*(\mathcal{M}^{n'}_{d(s')} \to
\mathcal{M}^{n'-l}_{d(s')-l})$.
\end{itemize}

\end{itemize}

\end{definition}

\begin{remark}
\label{rem_hobolth_one}
(i) The quantities $p_\theta(n,d)$ and the proposal of the merging size involve
the optimal proposal distribution for samples with at most one segregating site
so these quantities can be easily computed numerically and kept in a 
lookup table. \\
(ii) Hobolth et.\ al.\ showed in \cite[Theorem~2]{Hobolth2008} that if the
sample is of size 2, then the optimal proposal distribution chooses one of the
two types proportional to the number of mutations it differs from the root of
the genetree. They note in \cite[Remark~3]{Hobolth2008} that their
proposal distribution equals the optimal one in that case. 
The same statement is true for $\QSingleHUW$, since the
dynamics of a sample of size two does depend on $\Lambda$ only through the total
mass. For more general samples this effect should also favour types that have a
large number of mutations. 
\end{remark}

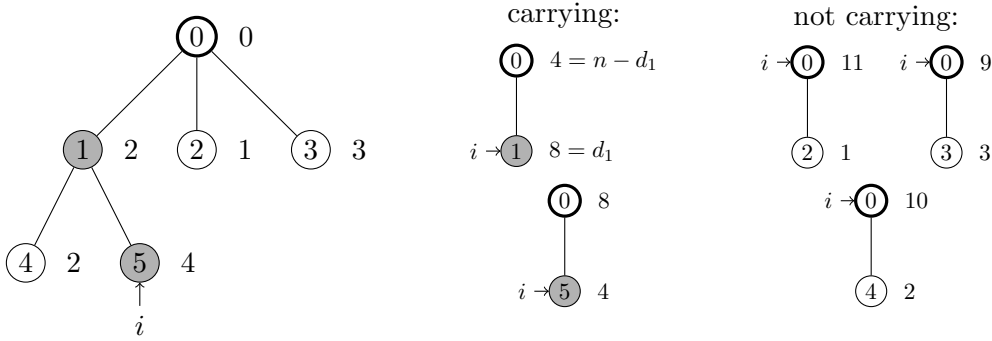
\begin{figure}
\begin{center}
\begin{minipage}[b]{.4\textwidth}
\begin{center}
\setlength{\nodedist}{40pt}
\begin{tikzpicture}[remember picture,scale=1,label distance=.1\nodedist,transform shape,node distance=\nodedist]
\node[very thick, draw,circle,inner sep=2pt,label=right:0] (root) {0}
child {node[fill=black!30,draw,circle,inner sep=2pt,label=right:2] (m1) {1}
child {node[draw,circle,inner sep=2pt,label=right:2] (m4) {4}}
child {node[fill=black!30,draw, circle,inner sep=2pt,label=right:4] (m5) {5}}}
child {node[draw,circle,inner sep=2pt,label=right:1] (m2) {2}}
child {node[draw,circle,inner sep=2pt,label=right:3] (m3) {3}};

\draw[<-] ([yshift=-.05\nodedist]m5) -- +(0,-.5) node[below] {$i$};
\end{tikzpicture}
\end{center}
\end{minipage}
\begin{minipage}[b]{.2\textwidth}
\begin{center}
carrying:\\[1ex]
\setlength{\nodedist}{40pt}
\begin{tikzpicture}[scale=.8,label distance=.1\nodedist,transform shape,node distance=\nodedist]
\node[very thick, draw,circle,inner sep=2pt,label=right:{$4=n-d_1$}] (root) {0}
child {node[fill=black!30,draw, circle,inner sep=2pt,label=right:{$8=d_1$}] (m) {1}};
\draw[<-] (m) -- +(-.5,0) node[left] {$i$};
\end{tikzpicture}\\[1ex]
\setlength{\nodedist}{40pt}
\begin{tikzpicture}[scale=.8,label distance=.1\nodedist,transform shape,node distance=\nodedist]
\node[very thick, draw,circle,inner sep=2pt,label=right:8] (root) {0}
child {node[fill=black!30,draw, circle,inner sep=2pt,label=right:4] (m) {5}};
\draw[<-] (m) -- +(-.5,0) node[left] {$i$};
\end{tikzpicture}\\\vspace{.5cm}
\end{center}
\end{minipage}
\begin{minipage}[b]{.3\textwidth}
\begin{center}
not carrying:\\[1ex]
\setlength{\nodedist}{40pt}
\begin{tikzpicture}[scale=.8,label distance=.1\nodedist,transform shape,node distance=\nodedist]
\node[very thick, draw, circle,inner sep=2pt,label=right:11] (root) {0}
child {node[draw,circle,inner sep=2pt,label=right:1] (m) {2}};
\draw[<-] (root) -- +(-.5,0) node[left] {$i$};
\end{tikzpicture}
\setlength{\nodedist}{40pt}
\begin{tikzpicture}[scale=.8,label distance=.1\nodedist,transform shape,node distance=\nodedist]
\node[very thick, draw, circle,inner sep=2pt,label=right:9] (root) {0}
child {node[draw,circle,inner sep=2pt,label=right:3] (m) {3}};
\draw[<-] (root) -- +(-.5,0) node[left] {$i$};
\end{tikzpicture}\\[1ex]
\setlength{\nodedist}{40pt}
\begin{tikzpicture}[scale=.8,label distance=.1\nodedist,transform shape,node distance=\nodedist]
\node[very thick, draw, circle,inner sep=2pt,label=right:10] (root) {0}
child {node[draw,circle,inner sep=2pt,label=right:2] (m) {4}};
\draw[<-] (root) -- +(-.5,0) node[left] {$i$};
\end{tikzpicture}\\\vspace{.5cm}
\end{center}
\end{minipage}\\
\end{center}
\vspace{-.5cm}
\caption{A sample configuration is depicted on the left and a type $i$ is
marked. On the right all possible compressed genetrees are listed. The type
corresponding to $i$ is marked in the compressed genetrees. Either type $i$
corresponds to the type carrying the mutation or not.}
\label{importance_fig_all_mutations}
\end{figure}

Figure~\ref{importance_fig_all_mutations} depicts a sample configuration and all
corresponding compressed genetrees with one mutation. Hobolth et.\ al.\ provide
in \cite{Hobolth2008} explicit formulae for the optimal transition probabilities
for samples with just one visible mutation if the underlying genealogy is given
by Kingman's coalescent.\\

\noindent
{\bf Schemes regarding Pairs of Mutations} We now extend the
approach of \cite{Hobolth2008} to consider compressed genetrees which allow two mutations.
First note that 
there are two kinds of structurally distinct genetrees with two mutations.
\begin{definition}\label{importance_def_two_mutation_genetrees}
Let
\begin{equation}
\mathcal{M}^n_{d_1,d_2} := \Big( \big[ (0),(1,0),(2,0)\big]_{\colzwo\approx}, (n-d_1-d_2,d_1,d_2)\Big)
\end{equation}
be the genetree where the two mutations are on different branches. The number
of individuals carrying mutation $m$ is $d_m$. Denote by
\begin{equation}
\mathcal{N}^n_{d_1,d_2} := \Big( \big[ (0),(1,0),(2,1,0)\big]_{\colzwo\approx}, (n-d_1,d_1-d_2,d_2)\Big)
\end{equation}
the sample configuration where the mutations are on the same branch. The number
of individuals carrying only mutation 1 is $d_1-d_2$ and both mutations are
carried by $d_2$ individuals.
\end{definition}
The two possible types of genetrees are depicted in
Figure~\ref{importance_pic_two_samples}.
\begin{remark}
(i) Note that $\mathcal{M}^n_{d_1,d_2} = \mathcal{M}^n_{d_2,d_1}$ 
(as equivalence classes under $\approx$) holds for $d_1
+d_2 \leq n$. 
Furthermore, note that $\mathcal{M}^n_{d_1,d_2}$ with $d_1 + d_2 =
n$, $\mathcal{N}^n_{0,d_2}$ and $\mathcal{N}^n_{n,d_2}$ 
denote valid genetrees with two mutations (even though for the latter two, 
mutation 1 is then not segregating), 
whereas by a slight abuse of
notation $\mathcal{M}^n_{d_1,0} = \mathcal{M}^n_{0,d_1} = \M^n_{d_1}$ and
$\mathcal{N}^n_{d_1,0} = \M^n_{d_1}$ denote genetrees with only one segregating
site. We denote by $\M^n_{0,0} = \mathcal{N}^n_{0,0} = \M^n_0$ the
sample of size $n$ with no segregating sites.
\end{remark}

\begin{figure}\centering
\subfigure[The genetree for the sample configuration $\lbrack\M^n_{d_1,d_2}\rbrack_\approx$]{
\label{importance_pic_m_n_d2}
\hspace{.8cm}
\begin{tikzpicture}[scale=0.8, transform shape]
\node[very thick, draw,circle,inner sep=1pt,label=right:{$n-d_1-d_2$}] (root) {0}
child {node[draw,circle,inner sep=1pt,label=right:{$d_1$}] (m1) {1}}
child {node[draw,circle,inner sep=1pt,label=right:{$d_2$}] (m2) {2}
child {edge from parent[draw=none]}};
\end{tikzpicture}
\hspace{.3cm}
}
\hspace{1cm}
\subfigure[The genetree for the sample configuration $\lbrack\mathcal{N}^n_{d_1,d_2}\rbrack_\approx$]{
\label{importance_pic_n_n_d2}
\hspace{1cm}
\setlength{\nodedist}{30pt}
\begin{tikzpicture}[scale=0.8, transform shape,node distance=\nodedist]
\node[very thick, draw,circle,inner sep=1pt,label=right:{$n-d_1$}] (root) {0}
child {node[draw,circle,inner sep=1pt,label=right:{$d_1-d_2$}] (m1) {1} child {node[draw,circle,inner sep=1pt,label=right:{$d_2$}] (m2) {2}}};
\end{tikzpicture}
\hspace{1cm}
}
\vspace{-.3cm}
\caption{The two different sample configurations of size $n$ with two
segregating sites (or mutations).
$d_i$ individuals carry mutation $i$, $i=1,2$. 
}
\label{importance_pic_two_samples}
\end{figure}
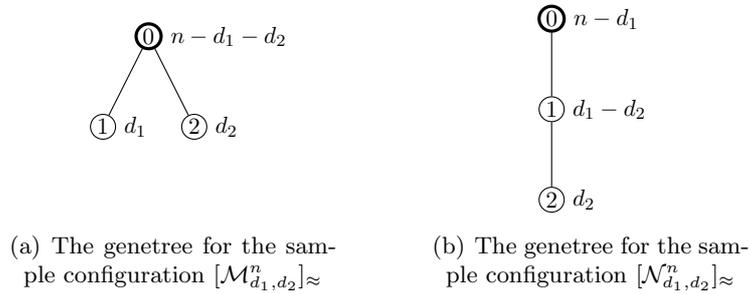

We now introduce the notion of a compressed genetree with regard to pairs of mutations.

{\colzwo
\begin{definition}[compressed genetree]
Let $[{\bf t, n}]$ be a genetree with $s\ge 2$ segregating sites. Let $s', s'' \in \{1, \dots, s\}$. Then, we denote the `compressed genetree'
with respect to the segregating sites $s', s''$ by $[{\bf t, n}](s', s'')$, where
\[
[{\bf t, n}](s', s'') = \M^n_{d(s'), d(s'')}
\]
if there is no type in $[{\bf t, n}]$ which carries mutations at both $s'$ and $s''$, and 
\[
[{\bf t, n}](s', s'') = \mathcal N^n_{d(s'), d(s'')}
\]
if there is at least one type in $[{\bf t, n}]$ which carries mutations at both $s'$ and $s''$, and there is no type which carries a mutation at $s''$ but not at $s'$.
\end{definition}
}

To consider pairs of mutations determining the probabilities of performing a
step involving type $k$ in a general sample configuration $({\bf t},{\bf n})$ it
is necessary to know the relation of the outmost mutation $x_{k0}$ of type $k$ (if it carries a mutation at all) 
to the given pair of mutations (resp.~the corresponding segregating sites) in the genetree. 
In other words, the type in the
compressed genetree corresponding to type $k$ needs to be determined. Based on
this information the appropriate most recent event in the history of the compressed tree can be
chosen. Figure~\ref{importance_fig_two_mutation_compression} shows two examples
of compressed genetrees for two given segregating sites. By symmetry, this
relation can be described by one of five distinct cases.

\begin{definition}
Let $[{\bf t},{\bf n}]$ be a genetree with $s\ge 2$ segregating sites 
and let $k\in \{1,\dots, d\}$. Let $s', s'' \in \{ 1, \dots, s\}$ be two segregating sites.
Then, we distinguish the following cases:\\[1ex]
\begin{tabular}{ll}
Case I &if type $k$ bears mutations at $s'$ and $s''$, \\
Case II &if type $k$ bears a mutation at $s'$, but not at $s''$, \\ & and there exists a type carrying both mutations, \\
Case III &if type $k$ bears a mutation at $s'$, but not at $s''$, \\ & and there exists no type carrying both mutations, \\
Case IV &if type $k$ does not bear any mutation at $s'$ or $s''$, \\ & and there exists a type carrying both mutations, \\
Case V &if type $k$ does not bear any mutation at $s'$ or $s''$, \\ & there exists no type carrying both mutations. \\
\end{tabular}
\end{definition}

The five cases are
depicted in Figure~\ref{importance_pic_five_possibilities}.

\begin{figure}
\begin{center}
\subfigure[The two mutations considered in this example are mutation 1 and 8. 
The compressed genetree is of the form $\mathcal{N}^n_{d_1,d_2}$ and type $i$ corresponds to mutation 1.] {
\setlength{\nodedist}{40pt}
\begin{tikzpicture}[scale=0.8, transform shape,node distance=\nodedist]
\node[very thick, fill=black!50,draw,circle,inner sep=1pt,label=right:3] (root) at (0,0) {0}
child {node[fill=black!50,draw,circle,inner sep=1pt] (m1) {1}
child {node[draw,circle,inner sep=1pt,label=right:1] (m4) {4}}
child {node[draw,circle,inner sep=1pt] (m5) {5}
child {node[fill=black!50,draw,circle,inner sep=1pt,label=right:2] (m8) {8}}}}
child {node[draw,circle,inner sep=1pt,label=right:5] (m2) {2}}
child {node[draw,circle,inner sep=1pt] (m3) {3}
child {node[draw,circle,inner sep=1pt,label=right:1] (m6) {6}}};

\draw[<-] (m5) -- +(.5,0) node[right] {$i$};

\node[draw=none] (blub) at (2.4,-2.2) {{\Large $\Rightarrow$}};

\node[very thick, fill=black!50,draw,circle,inner sep=1pt,label=right:9] (and_root) at (3.9,-0.5) {0}
child {node[fill=black!50,draw,circle,inner sep=1pt,label=right:1] (and_m1) {1}
child {node[fill=black!50,draw,circle,inner sep=1pt,label=right:2] (and_m8) {8}}};

\draw[<-] (and_m1) -- +(-.5,0) node[left] {$i$};
\end{tikzpicture}
}
\hspace{.5cm}
\subfigure[In this example mutation 3 and 5 are considered. The compressed genetree is of the form 
$\mathcal{M}^n_{d_1,d_2}$ and type $i$ corresponds to the root type.] {
\setlength{\nodedist}{40pt}
\begin{tikzpicture}[scale=0.8, transform shape,node distance=\nodedist]
\node[very thick, fill=black!50,draw,circle,inner sep=1pt,label=right:3] (root) at (0,0) {0}
child {node[draw,circle,inner sep=1pt] (m1) {1}
child {node[draw,circle,inner sep=1pt,label=right:1] (m4) {4}}
child {node[fill=black!50,draw,circle,inner sep=1pt] (m5) {5}
child {node[draw,circle,inner sep=1pt,label=right:2] (m8) {8}}}}
child {node[draw,circle,inner sep=1pt,label=right:5] (m2) {2}}
child {node[fill=black!50,draw,circle,inner sep=1pt] (m3) {3}
child {node[draw,circle,inner sep=1pt,label=right:1] (m6) {6}}};

\draw[<-] (m1) -- +(-.5,0) node[left] {$i$};

\node[draw=none] (blub) at (2.4,-2.2) {{\Large $\Rightarrow$}};
\node[very thick, fill=black!50,draw,circle,inner sep=1pt,label=right:9,line width=1.5pt] (and_root) at (3.9,-1.35) {0}
child {node[fill=black!50,draw,circle,inner sep=1pt,label=right:2] (and_m5) {5}}
child {node[fill=black!50,draw,circle,inner sep=1pt,label=right:1] (and_m3) {3}};

\draw[<-] (and_root) -- +(-.5,0) node[left] {$i$};
\end{tikzpicture}
}
\end{center}
\vspace{-.5cm}
\caption{Two examples of genetree compressions. The type $i$ is identified with one of the three types in the compressed sample by this procedure.}
\label{importance_fig_two_mutation_compression}
\end{figure}
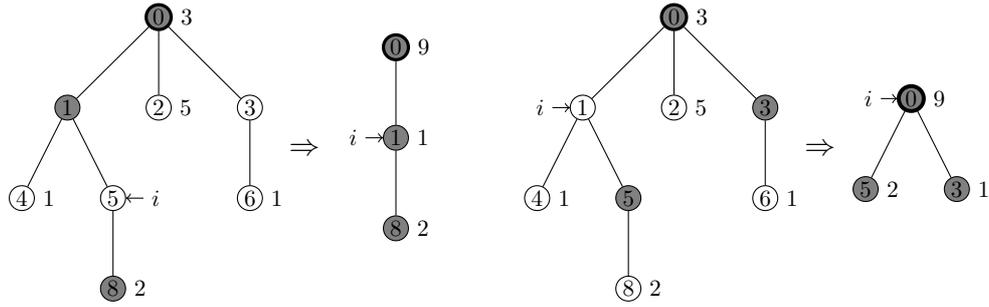

\begin{figure}\centering
\subfigure[Case I]{
\begin{tikzpicture}[scale=0.78, transform shape]
\node[very thick, draw,circle,inner sep=1pt,label=right:{$n-d(s')$}] (root) {0}
child {node[draw,circle,inner sep=1pt,label=right:{$d(s')-d(s'')$}] (m1) {$s'$} child {node[fill=black!50,draw,circle,inner sep=1pt,label=right:{$d(s'')$}] (m2) {$s''$}}};
\end{tikzpicture}
}
\hspace{.1cm}
\subfigure[Case II]{
\begin{tikzpicture}[scale=0.78, transform shape]
\node[very thick, draw,circle,inner sep=1pt,label=right:{$n-d(s')$}] (root) {0}
child {node[fill=black!50,draw,circle,inner sep=1pt,label=right:{$d(s')-d(s'')$}] (m1) {$s'$} child {node[draw,circle,inner sep=1pt,label=right:{$d(s'')$}] (m2) {$s''$}}};
\end{tikzpicture}
}
\hspace{.1cm}
\subfigure[Case III]{
\begin{tikzpicture}[scale=0.78, transform shape]
\node[very thick, draw,circle,inner sep=1pt,label=right:{$n-d(s')-d(s'')$}] (root) {0}
child {node[fill=black!50,draw,circle,inner sep=1pt,label=right:{$d(s')$}] (m1) {$s'$}}
child {node[draw,circle,inner sep=1pt,label=right:{$d(s'')$}] (m2) {$s''$}};
\end{tikzpicture}
}
\hspace{.1cm}
\subfigure[Case IV]{
\begin{tikzpicture}[scale=0.78, transform shape]
\node[very thick, fill=black!50,draw,circle,inner sep=1pt,label=right:{$n-d(s')$}] (root) {0}
child {node[draw,circle,inner sep=1pt,label=right:{$d(s')-d(s'')$}] (m1) {$s'$} child {node[draw,circle,inner sep=1pt,label=right:{$d(s'')$}] (m2) {$s''$}}};
\end{tikzpicture}
}
\hspace{.1cm}
\subfigure[Case V]{
\begin{tikzpicture}[scale=0.78, transform shape]
\node[very thick, fill=black!50,draw,circle,inner sep=1pt,label=right:{$n-d(s')-d(s'')$}] (root) {0}
child {node[draw,circle,inner sep=1pt,label=right:{$d(s')$}] (m1) {$s'$}}
child {node[draw,circle,inner sep=1pt,label=right:{$d(s'')$}] (m2) {$s''$}};
\end{tikzpicture}
}
\caption{The five different cases of types being affected by the most recent event in the two genetrees corresponding to configurations $\mathcal{M}^n_{d(s'),d(s'')}$ (Case III and V) and $\mathcal{M}^n_{d(s'),d(s'')}$ (Case I, II and IV). The shaded node refers to the proposed type.}
\label{importance_pic_five_possibilities}
\end{figure}
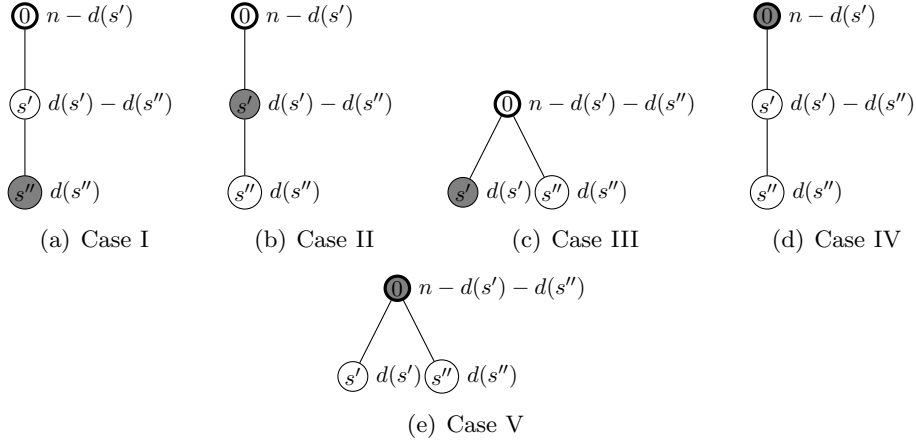

Again, we will now derive proposal distributions based on optimal proposals for the compressed
trees. To this end, note that the optimal transition probabilities for samples with two mutations can be
calculated numerically. To determine the transitions of the proposal Markov
chain until it hits the root configuration corresponding to $\mathcal{M}^1_0$ the transition
probabilities for the cases with one mutation or zero mutations also have to be
precomputed.
Thus, we shall set the probability weights for transitions involving 
samples with at most two mutations equal to the optimal weights in all 
the proposal distributions defined below (at no extra computational cost).

Fix a sample $[{\bf t},{\bf n}]$  with $d$ different types and at least two segregating sites $s', s''$. 
Note that a possible transition of the proposal Markov
chain can be characterised by a pair $(i,l)$ with $1\leq i \leq d$ and $0 \leq l
\leq n_i-1$, where $i$ denotes the type that is involved in the most recent
event and $l$ denotes the amount by which the multiplicity is decreased. Denote
by $l=0$ the case that the outmost mutation of type $i$ is removed from the
genetree (if type $i$ is an eligible singleton).

Now define, for $l \ge 1$, the quantity
\begin{equation}
\mathfrak{u}_\theta^{\{s', s''\}} (i,l) := \begin{cases}
\frac{n_i}{d(s'')} \Q_\theta^* (\mathcal{N}^n_{d(s'),d(s'')} \to \mathcal{N}^{n-l}_{d(s')-l,d(s'')-l}) & \text{in Case I},\\
\frac{n_i}{d(s') - d(s'')} \Q_\theta^* (\mathcal{N}^n_{d(s'),d(s'')} \to \mathcal{N}^{n-l}_{d(s')-l,d(s'')}) & \text{in Case II},\\
\frac{n_i}{d(s')} \Q_\theta^* (\mathcal{M}^n_{d(s'),d(s'')} \to \mathcal{M}^{n-l}_{d(s')-l,d(s'')}) & \text{in Case III},\\
\frac{n_i}{n - d(s')} \Q_\theta^* (\mathcal{N}^n_{d(s'),d(s'')} \to \mathcal{N}^{n-l}_{d(s'),d(s'')}) & \text{in Case IV},\\
\frac{n_i}{n - d(s') - d(s'')} \Q_\theta^* (\mathcal{M}^n_{d(s'),d(s'')} \to \mathcal{M}^{n-l}_{d(s'),d(s'')}) & \text{in Case V}.\\
\end{cases}
\end{equation}
For $l=0$ let
\begin{equation}
\label{importance_eq_pair_valuate_mutation}
\mathfrak{u}^{\{s', s''\}}_\theta (i,0) = \begin{cases}
                    \frac{n_i}{d(s'')}\Q^*_\theta (\mathcal{N}^n_{d(s'),1} \to \mathcal{N}^{n}_{d(s'),0}) \1_{\{d(s'')=1\}} & \text{in Case I},\\
                    0 & \text{in Case II},\\
                     \frac{n_i}{d(s')}\Q^*_\theta (\mathcal{M}^n_{d(s''),1} 
\to \mathcal{M}^{n}_{d(s''),0}) \1_{\{d(s') = 1\}} & \text{in Case III},\\
                    0 & \text{in Case IV},\\
                    0 & \text{in Case V}.\\
                  \end{cases}
\end{equation}
Note that in Case I, III and IV the order of the mutations in the compressed
tree $\mathcal{N}^n_{d(s'),d(s'')}$ is determined by their order in the original
genetree. Analogous to (\ref{defuthetasi}), 
$\mathfrak{u}^{\{s', s''\}}_\theta (i,l)$ would be the 
optimal probability weight of transition $(i,l)$ if only the two mutations 
$s'$ and $s''$ existed in the data. 
\smallskip

Finally, we define for each type $i, 1\leq i\leq d$, and segregating sites $s', s''$, 
\begin{equation}
u_\theta^{\{s', s''\}} (i) := \begin{cases}
\sum_{l=0}^{d(s'')-1} \mathfrak{u}_\theta^{\{s', s''\}} (i,l) & \text{in Case I},\\
\sum_{l=0}^{d(s')-d(s'')-1} \mathfrak{u}_\theta^{\{s', s''\}} (i,l) & \text{in Case II},\\
\sum_{l=0}^{d(s')-1} \mathfrak{u}_\theta^{\{s', s''\}}(i,l) & \text{in Case III},\\
\sum_{l=0}^{n - d(s')-1} \mathfrak{u}_\theta^{\{s', s''\}}(i,l) & \text{in Case IV},\\
\sum_{l=0}^{n-d(s')-d(s'')-1} \mathfrak{u}_\theta^{\{s', s''\}}(i,l) & \text{in Case V}.\\
\end{cases}
\end{equation}
We will now use these quantities as probability weights for the event that in the compressed genetree $[{\bf t}, {\bf n}](s', s'')$,
the last event in the history involved type $k$ (under our new proposal distributions).

\begin{definition}[Probability weights for picking eligible types]
\label{def:prop}
Given $[{\bf t}, {\bf n}]$ with $d$ types and $s\ge 2$ segregating sites, let, for each {\em eligible} $k \in \{1, \dots d\}$,
\[
Q^1_\theta([{\bf t}, {\bf n}])(k) := \frac{\sum_{\{1 \le s'< s'' \le s \}} u_\theta^{\{s', s''\}} (k)}{\sum_{i=1}^{d}
    \sum_{\{1 \le s' < s'' \le s \}} u_\theta^{\{s', s''\}} (i)}.
\]
If $k$ is not eligible, put $Q^1_\theta([{\bf t}, {\bf n}])(k) =0$.
\end{definition}

This distribution can be used to
propose a type to be involved in the most recent event. In a second step one may then choose the size of
the possible merger, similar as before in that again the probabilities of the merger sizes in a specific sample, now
with two mutations, are considered.

\begin{definition}[Proposal distribution $\QPairHUWtwoStepA$] 
\label{importance_def_double_hobolth_one_one}
We define a distribution $\QPairHUWtwoStepA$ on the 
space of histories $\H$ as the law of a Markov chain with transitions as follows.
Let $[{\bf t},{\bf n}]$ be a sample configuration with at least $s \ge 2$ segregating sites and $d\ge 1$ types. 
\begin{itemize} 
\item Choose a type $i \in \{1, \dots, d\}$ to
be involved in the most recent event in history according to $Q^1_\theta([{\bf t}, {\bf n}])(k)$ from Definition~\ref{def:prop}. 
\item If $n_i = 1$, remove the outmost mutation of type $i$ (noting that a.s.~only eligible types can be chosen).
\item If $n_i \geq 2$, and $i$ bears at least one mutation, let $s'$ be the segregating site corresponding to the outmost mutation of type $i$. 
   Reduce the multiplicity of type $i$ by $l$ with probability $\Q_\theta^* (\mathcal{N}^n_{d(s'),d(s')-n_i} \to
           \mathcal{N}^{n-l}_{d(s')-l,d(s')-n_i})$. If type $i$ is the root type, reduce the
           multiplicity by $l$ with probability $\Q_\theta^* (\mathcal{M}^n_{n-n_i} \to
                  \mathcal{M}^{n-l}_{n-l-n_i})$.
\end{itemize}

\end{definition}

Alternatively, one may consider all mutations present in the sample.

\begin{definition}[Proposal distribution $\QPairHUWtwoStepB$]
\label{importance_def_double_hobolth_two_two}
We define a distribution $\QPairHUWtwoStepB$ on the 
space of histories $\H$ as the law of a Markov chain with transitions as follows.
Let $[{\bf t},{\bf n}]$ be a sample configuration with at least $s \ge 2$ segregating sites and $d\ge 1$ types. 
\begin{itemize} 
\item Choose a type $i \in \{1, \dots, d\}$ to
be involved in the most recent event in history according to $Q^1_\theta([{\bf t}, {\bf n}])(k)$ from Definition~\ref{def:prop}. 
\item If $n_i = 1$, remove the outmost mutation of type $i$ (noting that a.s.~only eligible types can be chosen).
\item If $n_i \geq 2$, choose to decrease $n_i$ by
 $l$ with probability proportional to
\begin{equation}
\sum_{s', s''} \mathfrak{u}_\theta^{\{s',s''\}} (i,l),
\end{equation}
where the sum extends over all pairs of segregating sites present in the current sample.
\end{itemize}

\end{definition}

It might appear artificial to consider choosing transition by such a two-step procedure instead of choosing
all at once. Indeed, the method of \cite{Hobolth2008} can be
extended in another direction by choosing the type involved in the most recent
event and the size of the possible merger in {\em one step}. 
We present two proposal distributions that let pairs of mutations valuate {\em
all} possible transitions and then the most recent step is chosen proportionally
to these weights.

\begin{definition}[Proposal Distribution $\QPairHUWoneStepA$]
\label{importance_def_double_hobolth_two_one}
We define a distribution $\QPairHUWoneStepA$ on the 
space of histories $\H$ as the law of a Markov chain with transitions as follows.
Let $[{\bf t},{\bf n}]$ be a sample configuration with at least $s \ge 2$ segregating sites and $d\ge 2$ types. 
We propose the event $(i,l)$ for
$1\leq i \leq d$ and $0 \leq l \leq n_i-1$ to be the most recent evolutionary
event with probability proportional to
\begin{equation}
       \begin{cases}
\sum_{\{s', s''\}} \mathfrak{u}^{\{s', s''\}}_\theta(i,l) & \text{if $i$ is eligible}\\
0 & \text{otherwise}.
\end{cases}
\end{equation}

\end{definition}

Note that in a given sample $({\bf t},{\bf n})$, by~\eqref{importance_eq_pair_valuate_mutation} the contribution of the presence of a pair of mutations
at $\{s', s''\}$ to the event $(i,0)$ of removing the outmost mutation of a
leaf type $i$ is zero if the corresponding $d(s')$ resp.~$d(s'')$ is greater than one. In a
generic genetree this case appears rather frequently and thus we argue that the
proposal distribution 
$\QPairHUWoneStepA$ underrates mutation events. 
\label{page:someone_underrates_mutation_events}
This effect is illustrated in Figure~\ref{importance_fig_underate_mutations}. 

{\ck
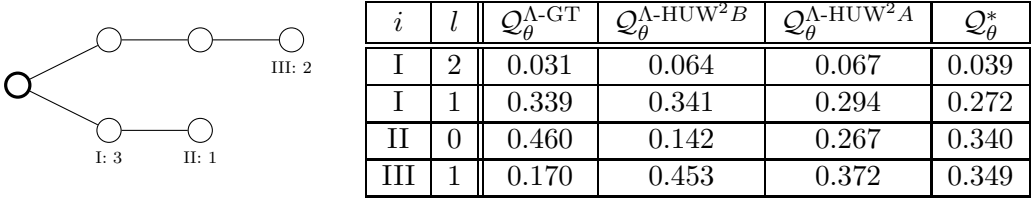
\begin{figure}
\begin{center}
\hspace{-1.8cm}\parbox{3.5cm}{%
\begin{tikzpicture}[scale=0.8, transform shape,inner sep=4pt]{\ck
\node[very thick, draw,circle] (root) {}
[grow=right]
child {node[draw,circle,label=below:\scriptsize{I: 3}] (m1) {}
child {node[draw,circle,label=below:\scriptsize{II: 1}] (m2) {}}}
child {node[draw,circle] (m3) {}
child {node[draw,circle] (m4) {}
child {node[draw,circle,label=below:\scriptsize{III: 2}] (m5) {}}}};
}
\end{tikzpicture}} \hspace{1.0cm}
\parbox{7cm}{%
\begin{tabular}{|c|c||c|c|c|c|} 
\hline
$i$ & $l$ & $\QGT$ & $\QPairHUWoneStepA$ & $\QPairHUWoneStepB$ & $\Q_\theta^*$ \\ 
\hline\hline
I & 2 & 0.031 & 0.064 & 0.067 & 0.039 \\
\hline
I & 1 & 0.339 & 0.341 & 0.294 & 0.272 \\
\hline
II & 0 & 0.460 & 0.142 & 0.267 & 0.340 \\
\hline
III & 1 & 0.170 & 0.453 & 0.372 & 0.349 \\
\hline
\end{tabular}}
\end{center}
\caption{Type II has multiplicity one and is a descendant of type I. Thus all
pairs of mutations that do not include the outmost mutation of type II weigh
the step removing the outmost mutation of type two with zero. The table on the
right shows that {\ck $\QPairHUWoneStepB$} is closer to the optimal
distribution than {\ck $\QPairHUWoneStepA$}}
\label{importance_fig_underate_mutations}
\end{figure}
}

To circumvent this problem, one may modify the proposal distribution by summing only over
those pairs of mutations where one of the mutations coincides with the outmost
mutation of the current type. This should reduce the number of pairs that put too
much emphasis on the merging events and establish a more balanced proposal
distribution.

\begin{definition}[Proposal Distribution $\QPairHUWoneStepB$]
\label{importance_def_double_hobolth_two_two_B}
We define a distribution $\QPairHUWoneStepB$ on the 
space of histories $\H$ as the law of a Markov chain with transitions as follows.
Let $[{\bf t},{\bf n}]$ be a sample configuration with at least $s \ge 2$ segregating sites and $d\ge 2$ types. 
We propose the event $(i,l)$ for
$1\leq i \leq d$ and $0 \leq l \leq n_i-1$ to be the most recent evolutionary
event with probability proportional to 
\begin{equation}
       \begin{cases}
       \sum_{s'} \frac{n_i}{n-d_{s'}} \Q_\theta^* (\mathcal{M}^{n}_{d(s')} \to \mathcal{M}^{n-l}_{d(s')}) & \text{if $i$ is eligible and the root type}\\
\sum_{\{s' \neq s_i \}} \mathfrak{u}^{\{s_i, s'\}}_\theta(i,l) & \text{if $i$ is eligible}\\
0 & \text{otherwise},
\end{cases}
\end{equation}
where $s_i$ is the segregating site corresponding to the outmost mutation $x_{i0}$ of type $i$.
\end{definition}

\begin{remark}
(i) Another positive side effect of this method is that it reduces the
complexity of proposing a step from quadratic to linear in the number of
mutations.\\
(ii) In Hobolth \& Wiuf \cite{Hobolth2009}, Section~4, 
explicit expressions for the sampling probabilities 
in the case of Kingman's coalescent for samples with two (nested) segregating sites are presented. 
The authors note in
Section~7 that their results `could potentially be used to further improve the
proposal distribution for inference in coalescent models.' Indeed, 
via Remark~\ref{rem:Gandp}, their
results can be applied to derive explicit formulae for the quantities
$\mathfrak{u}^{\{s_i, s'\}}_\theta(i,l)$ that govern the proposal distributions regarding
pairs of mutations (for $\Lambda=\delta_0$).\\
(iii) Note that the idea to let mutations valuate {\em all} possible transitions
can also be applied for the case when just one mutation at a time is considered
in the sense of $\QSingleHUW$.
\end{remark}

Our last proposal distribution combines the single-mutation approach with the
pair approach.
Indeed, note that the complexity of the proposal distributions
regarding pairs of mutations is quadratic in the number of mutations, whereas
the proposal distributions regarding all mutations have linear complexity. We will
see in Section~\ref{importance_sec_performance} that the real-time to compute
steps for the distributions differ. However, we find that the method determining
the size of the merger in proposal distribution 
$\QPairHUWtwoStepA$ 
from Definition~\ref{importance_def_double_hobolth_one_one} performs well. Thus
a promising candidate concerning speed and performance should be given by the
combination of proposing a type in the first step considering all mutations
separately and then choosing the merging size in the second step by the method
from $\QPairHUWtwoStepA$. 

\begin{definition}[Proposal Distribution $\QPairHUWmixed$]
\label{prop_QPairHUWmixed}
We define a distribution $\QPairHUWmixed$ on the 
space of histories $\H$ as the law of a Markov chain with transitions as follows.
Let $({\bf t},{\bf n})$ be a sample configuration with at least $s \ge 2$ segregating sites and $d\ge 2$ types. 
Choose type $i$ to be involved in the most recent event considering all
mutations according to the same method used for the distribution $\QSingleHUW$
from Definition~\ref{importance_def_single_hobolth}. 
If a singleton type is chosen,
remove the outmost mutation, whereas in the case of a non-singleton type $i$ with
$n_i\geq2$ the multiplicity is decreased by $l$ with probability $\Q_\theta^*
(\mathcal{N}^n_{d(o),d(o)-n_i} \to \mathcal{N}^{n-l}_{d(o)-l,d(o)-n_i})$, where
$1\leq l\leq n_i-1$.
\end{definition}

\begin{remark} \ck For the analysis of a sample of size $n$, 
the proposal schemes from 
(\ref{importance_def_double_hobolth_one_one}, 
\ref{importance_def_double_hobolth_two_two}, 
\ref{importance_def_double_hobolth_two_one},
\ref{importance_def_double_hobolth_two_two_B}, 
\ref{prop_QPairHUWmixed}) all require the numerical computation 
of the solution of (\ref{ordered_ims_recursion}) for all samples 
of size $m \leq n$ with at most two segregating sites. This can be 
precomputed, but should be kept in the computer's main memory during the 
(many) repeated runs. Thus, memory requirements can be a limiting 
factor prohibiting the analysis of large samples. 

Since in a sample 
of size $m \leq n$ with at most two mutations under the IMS there are 
at most three types (of several possible multiplicities), memory of the order $n^3$ will be required. For further speed-up one could also store the transition probabilities for all possible moves for each sample, which would result in a requirement of the order $n^4$.
\end{remark}

\section{Performance Comparison}
\label{importance_sec_performance}

In this section we investigate and compare the performance of the different proposal
distributions, introduced in Section~\ref{importance_sec_proposal_schemes}, 
in various scenarios by means of a (not necessarily comprehensive) simulation study.

Such a study faces two {\ck particular issues which need to be addressed}. 
First, one needs to identify (preferably parametric) sub-families of Lambda-coalescents 
which might be of biological relevance (i.e.~arise from microscopic 
modelling of the behaviour of the underlying population).
{\ck We will focus our attention to so-called {\em Beta-coalescents}, recalled below.}
A second {\ck issue} is owed to the fact that tractable sample complexities are still in the low three-digit numbers ($\approx 100$).
If one wishes to compare the performance of our sampling schemes one has to use either 
a few less generic scenarios where the samples have relatively large complexities or many samples of small complexity

This section can be outlined as follows: First, we introduce and discuss the class of Beta-coalescents. Then, we measure empirically the 
total variation distance between our proposal distributions and the optimal 
distribution for a small sample complexity. 
Next, we compare the concrete performance of our schemes for several randomly generated samples 
of small size for various scenarios, and for several relatively large real DNA sequence data samples.
Finally, we will discuss our results and try to come up with recommendations for the practitioner.

\subsection{{\ck Beta-coalescents}}

Recall that our `parameter' $\theta=(r, \Lambda)$ consists of the mutation rate $r$ 
and the underlying Lambda-coalescent with coalescent measure $\Lambda$. 
The case where $\Lambda = \delta_0$ is the classical Kingman case describing 
populations with constant population size and reproduction events
which are small when compared to the total population size.
Here, we will consider the case where $\Lambda= B(2-\alpha, \alpha)$, 
with $\alpha \in (0,2)$, {\ck that is,} so-called `Beta-coalescents' introduced by \cite{Schweinsberg2003}, whose density is given by
\[
\Lambda(dx)= \frac{\Gamma(2)}{\Gamma(2-\alpha)\Gamma(\alpha)} x^{1-\alpha} (1-x)^{\alpha-1} \, dx.
\]
Note that the Kingman-coalescent corresponds to the weak limit as $\alpha \to 2$. See, e.g., \cite{Schweinsberg2003} or \cite{Birkner2009} and the references there for a discussion 
of possible biological motivations of this class. 

\subsection{Distance to the optimal proposal distribution}

For small {\ck sample} complexities, 
it is possible to solve our recursions (\ref{recursion}), (\ref{real_ims_recursion}) and (\ref{ordered_ims_recursion})
numerically and hence to compute optimal proposal weights directly.
It is therefore natural to measure the distance between the optimal proposal distribution and 
our candidate distributions for such small complexities.
We consider a selection of parameter values for the Beta-coalescent (including the Kingman-coalescent) in Table~\ref{tb:dist} 
and present the total variation distance of the optimal weights of the possible steps and the 
weights given by the candidate distribution averaged over all possible samples of complexity 15. In enumerating all these samples, viewed as trees, we have found 
algorithms from \cite{Knuth2005} very helpful.

The relative ranking of the different candidates implied by the total variation distance 
is similar when using the mean-squared distance or the relative entropy (data not shown). 
The respective minimisers are printed in bold.

\begin{table}
\begin{center}
\scriptsize
\begin{tabular}{||l||l|l|l||l|l|l||l|l|l||}
\hline\hline
& \multicolumn{3}{c||}{$r = 0.5$}& \multicolumn{3}{c||}{$r = 1$}& \multicolumn{3}{c||}{$r = 2$}\\
\hline
& $\alpha=1$& $\alpha=1.5$& $\alpha=2$& $\alpha=1$& $\alpha=1.5$& $\alpha=2$& $\alpha=1$& $\alpha=1.5$& $\alpha=2$\\
\hline\hline
$\QGT$ & 0.166 & 0.118 & 0.080 & 0.172 & 0.134 & 0.088 & 0.127 & 0.114 & 0.084 \\
\hline
$\QSD$ & 0.226 & 0.114 & 0.060 & 0.220 & 0.142 & 0.088 & 0.151 & 0.115 & 0.084 \\
\hline
$\QSingleHUW$ & 0.115 & 0.077 & 0.045 & 0.119 & 0.102 & 0.074 & 0.083 & 0.082 & 0.071 \\
\hline
$\QPairHUWoneStepA$ & 0.069 & 0.058 & 0.039 & 0.088 & 0.096 & 0.084 & 0.068 & 0.082 & 0.091 \\
\hline
$\QPairHUWoneStepB$ & {\bf 0.054} & 0.047 & 0.038 & {\bf 0.064} & {\bf 0.065} & 0.063 & {\bf 0.053} & {\bf 0.055} & 0.060 \\
\hline
$\QPairHUWtwoStepA$ & 0.063 & 0.044 & {\bf 0.026} & 0.081 & 0.072 & {\bf 0.053} & 0.060 & 0.062 & {\bf 0.055} \\
\hline
$\QPairHUWtwoStepB$ & 0.058 & {\bf 0.041} & {\bf 0.026} & 0.076 & 0.069 & {\bf 0.053} & 0.058 & 0.059 & {\bf 0.055} \\
\hline
$\QPairHUWmixed$ & 0.092 & 0.063 & 0.038 & 0.111 & 0.097 & 0.071 & 0.081 & 0.081 & 0.071 \\
\hline\hline
\end{tabular}
\caption{\small Total variation distance between optimal proposal distribution and importance sampling schemes, {\ck averaged over all} samples of complexity 15. }
\label{tb:dist}
\end{center}
\end{table}

The best results are consistently provided by methods based on compressed genetrees with 
two mutations, {\ck namely}  $\QPairHUWoneStepB$, $\QPairHUWtwoStepA$ and $\QPairHUWtwoStepB$.
This is true not only for the Beta-coalescent, but in particular for Kingman's coalescent, so that our
{\ck new} methods seem to outperform even the classical methods known so far,  {\ck at least with respect to this rather theoretical criterion}.

\subsection{Performance comparison for different specif{\ck ic} tree structures}
\label{sec_specific_structures}

In this subsection, we aim to investigate strengths and weaknesses 
of our methods depending on the structure of the genetrees encoded by the datasets.

To this end, we simulated 500 genetrees under given 
parameters (for Beta-coalescents) of {\em sample size} 15. Note that
the corresponding tree complexities vary and can be much {\ck bigger} than 15.
From these 500 trees, we $a)$ uniformly pick one tree with an `average' number of mutations
(note that the distribution of the number of mutations can easily be 
computed recursively) and 
$b)$ choose a tree with a number of mutations according to the empirical $80 \%$ 
quantile of the 500 simulated trees (i.e.~a tree with `many' mutations). 
.
Sample trees chosen according to other criteria of `atypically high 
sample complexity' yielded similar results to those from case~$b$ (data 
not shown). The computations were carried out using {\tt MetaGeneTree} on computers with a standard performance (using AMD Opteron CPUs with 2.6~GHz).

We begin with $a)$, an {\em average tree} (with respect to number of mutations, for the given parameters), and investigate the performance 
of our methods for three different parameter values.
Figure~\ref{fig_avg_trees} 
shows the genetrees and the respective parameters used for its 
generation. Figure~\ref{fig_avg} shows the respective number of runs 
and computing time needed so that the relative empirical error
of the likelihood estimate becomes smaller than $1\%$. 
{\ck Again, our proposal distributions based on compressed genetrees fare rather well, with the notable exception of $\QPairHUWoneStepA$. }

\begin{figure}
\begin{center}
\psfrag{1}{$\scriptstyle 1$}
\psfrag{2}{$\scriptstyle 2$}
\psfrag{3}{$\scriptstyle 3$}
\psfrag{4}{$\scriptstyle 4$}
\psfrag{9}{$\scriptstyle 9$}
\psfrag{11}{$\scriptstyle 11$}
\subfigure[$r=0.5,\alpha=2$] {\label{fig_avg_trees_m05a2}
\fbox{\includegraphics{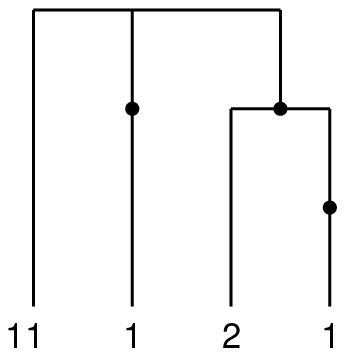}}
}
\subfigure[$r=0.5,\alpha=1$] {\label{fig_avg_trees_m05a1}
\fbox{\includegraphics{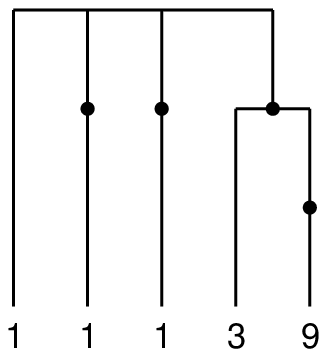}} 
}
\subfigure[$r=1.5,\alpha=1.5$] {\label{fig_avg_trees_m15a15}
\fbox{\includegraphics{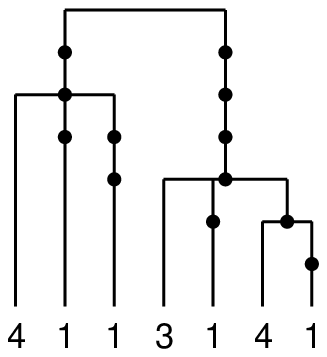}} 
}
\end{center}
\caption{Trees showing an average number of mutations out of 500 simulated trees 
under the respective parameters (leaf labels correspond to type multiplicities).}
\label{fig_avg_trees}
\end{figure}

\begin{figure}
\begin{center}
\psfrag{1}{$\scriptstyle 1$}
\psfrag{2}{$\scriptstyle 2$}
\psfrag{3}{$\scriptstyle 3$}
\psfrag{4}{$\scriptstyle 4$}
\psfrag{6}{$\scriptstyle 6$}
\psfrag{10}{$\scriptstyle 10$}
\subfigure[$r=0.5,\alpha=1$] {\label{fig_mut_trees_m05a1}
\fbox{\includegraphics{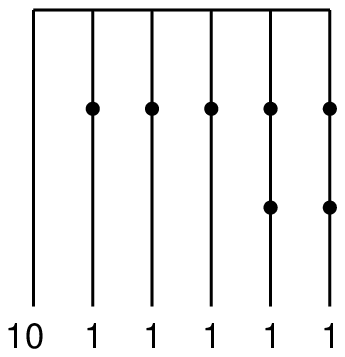}}
}
\subfigure[$r=1.5,\alpha=2$] {\label{fig_mut_trees_m15a2}
\fbox{\includegraphics{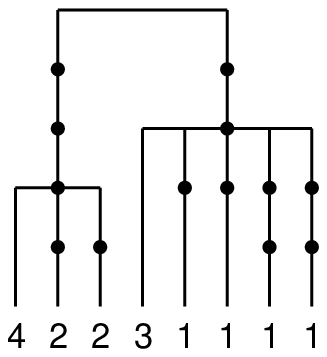}} 
}
\subfigure[$r=1.5,\alpha=1$] {\label{fig_mut_trees_m15a1}
\fbox{\includegraphics{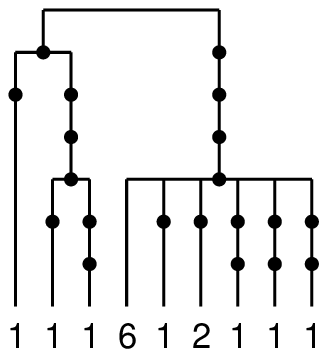}} 
}
\end{center}
\caption{Trees showing a number of mutations that equals the empirical 80\% quantile of 500 simulated trees under the respective parameters.}
\label{fig_mut_trees}
\end{figure}

\begin{figure}
\begin{center}
\psfrag{p1}[cr][c][1][270]{{\scriptsize $\QGT$}}
\psfrag{p7}[cr][c][1][270]{{\scriptsize $\QSD$}}
\psfrag{p8}[cr][c][1][270]{{\scriptsize $\QSingleHUW$}}
\psfrag{p10}[cr][c][1][270]{{\scriptsize $\QPairHUWoneStepA$}}
\psfrag{p12}[cr][c][1][270]{{\scriptsize $\QPairHUWoneStepB$}}
\psfrag{p14}[cr][c][1][270]{{\scriptsize $\QPairHUWtwoStepA$}}
\psfrag{p15}[cr][c][1][270]{{\scriptsize $\QPairHUWtwoStepB$}}
\psfrag{p18}[cr][c][1][270]{{\scriptsize $\QPairHUWmixed$}}
\psfrag{m05a2}{{\tiny $r=0.5,\alpha=2$}}
\psfrag{m15a15}{{\tiny $r=1.5,\alpha=1.5$}}
\psfrag{m05a1}{{\tiny $r=0.5,\alpha=1$}}
\psfrag{xlab}[cc]{{\footnotesize Distribution}}
\psfrag{ylab}[cc]{{\footnotesize log(\#)}}
\subfigure[Base-10 logarithm of the number of runs needed ($\log(\#)$).] {\label{fig_avg_runs}
\includegraphics[scale=.40]{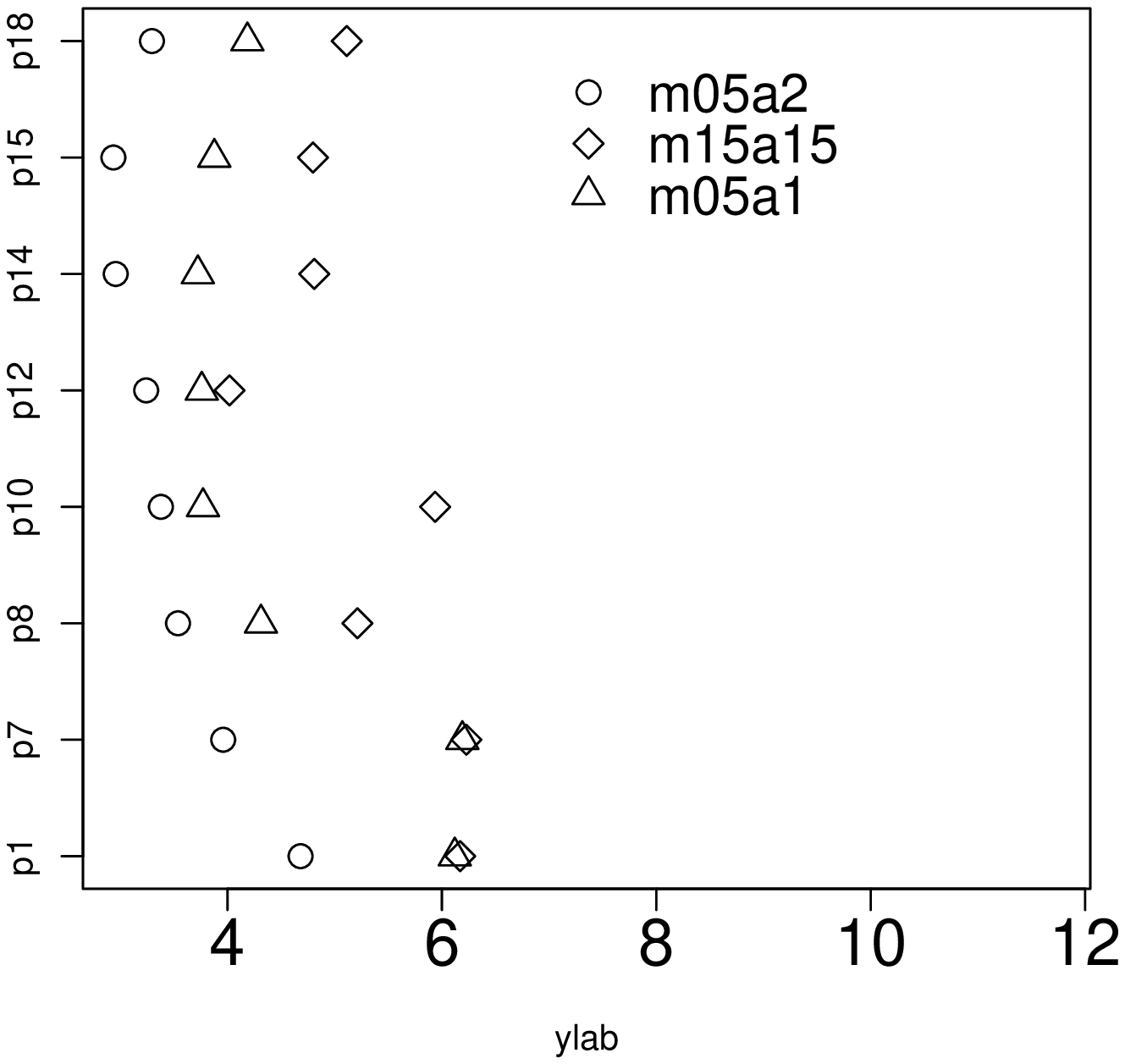}}
\psfrag{ylab}[cc]{{\footnotesize log($t$)}}
\subfigure[Base-10 logarithm of the computing time needed in seconds ($\log(t)$).] {\label{fig_avg_times}\ \hspace{0.5cm}\includegraphics[scale=.40]{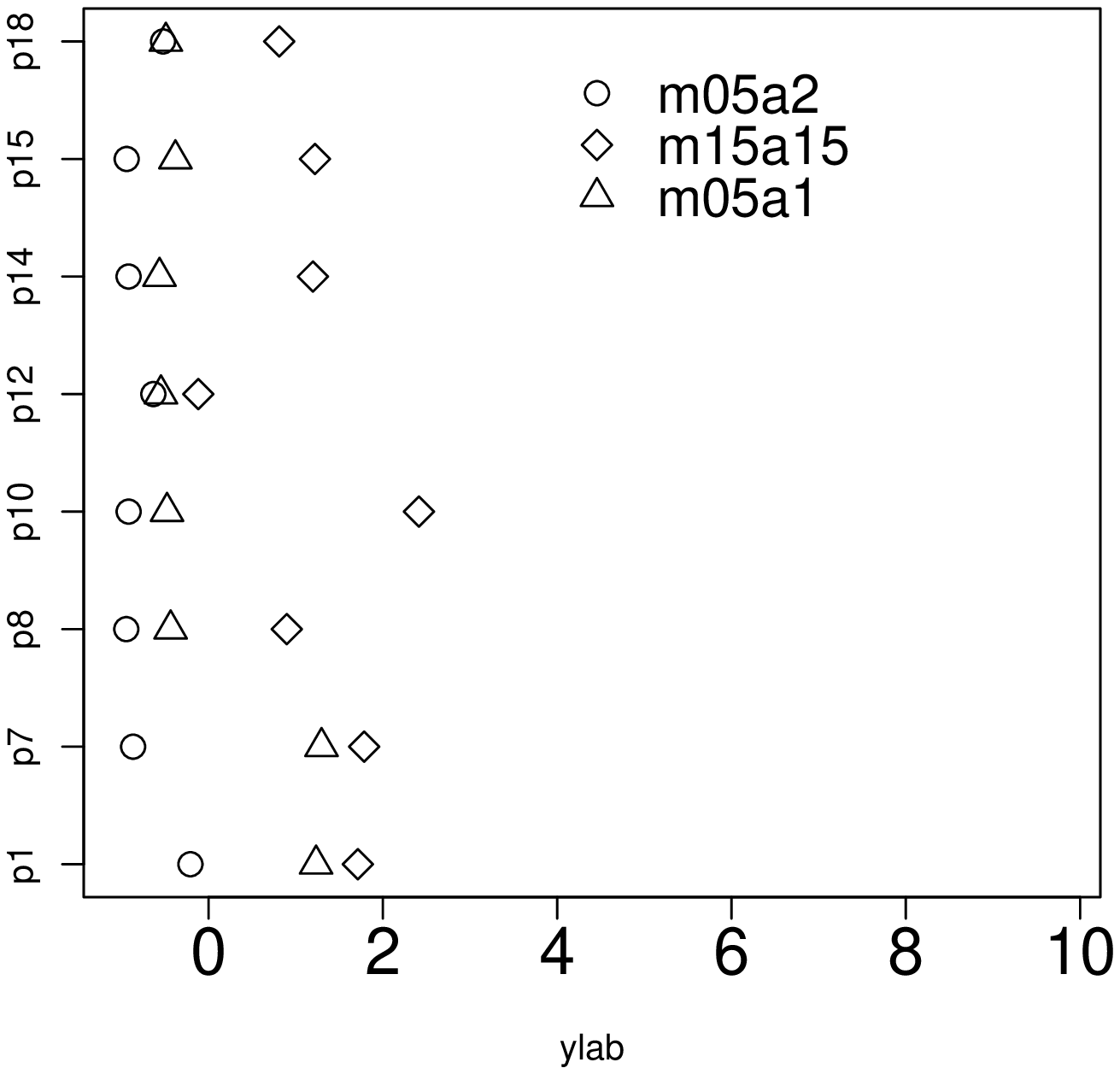}}
\end{center}
\caption{\small Number of runs and computing time needed to obtain a relative error below 1\% for the {\ck `average trees' given in Figure~\ref{fig_avg_trees}.}}
\label{fig_avg}
\end{figure}

$b)$\ Our next set of {\ck genetrees corresponds to} the $80 \%$ quantile 
with respect to the number of mutations on the tree (i.e.~trees with an
exceptionally large number of mutations, and therefore {\ck relatively} high tree complexity). 
Figure~\ref{fig_mut_trees} 
shows the genetrees and the 
respective parameters and Figure~\ref{fig_mut} gives the number of runs and 
computing time needed so that the relative empirical error
of the likelihood estimate becomes smaller than $1\%$.
{\ck As expected, the average computational time, due to increased complexity, increases significantly 
in comparison to an `average' tree. The relative performance of our methods, however, remains
similar -- in particular, $\QPairHUWoneStepB$ performs best.}

\begin{figure}
\begin{center}
\psfrag{p1}[cr][c][1][270]{{\scriptsize $\QGT$}}
\psfrag{p7}[cr][c][1][270]{{\scriptsize $\QSD$}}
\psfrag{p8}[cr][c][1][270]{{\scriptsize $\QSingleHUW$}}
\psfrag{p10}[cr][c][1][270]{{\scriptsize $\QPairHUWoneStepA$}}
\psfrag{p12}[cr][c][1][270]{{\scriptsize $\QPairHUWoneStepB$}}
\psfrag{p14}[cr][c][1][270]{{\scriptsize $\QPairHUWtwoStepA$}}
\psfrag{p15}[cr][c][1][270]{{\scriptsize $\QPairHUWtwoStepB$}}
\psfrag{p18}[cr][c][1][270]{{\scriptsize $\QPairHUWmixed$}}
\psfrag{m05a2}{{\tiny $r=0.5,\alpha=2$}}
\psfrag{m15a15}{{\tiny $r=1.5,\alpha=1.5$}}
\psfrag{m15a1}{{\tiny $r=1.5,\alpha=1$}}
\psfrag{m05a1}{{\tiny $r=0.5,\alpha=1$}}
\psfrag{m15a2}{{\tiny $r=1.5,\alpha=2$}}
\psfrag{xlab}[cc]{{\footnotesize Distribution}}
\psfrag{ylab}[cc]{{\footnotesize log(\#)}}
\subfigure[Base-10 logarithm of the number of runs needed ($\log(\#)$).] {\label{fig_mut_runs}\includegraphics[scale=.40]{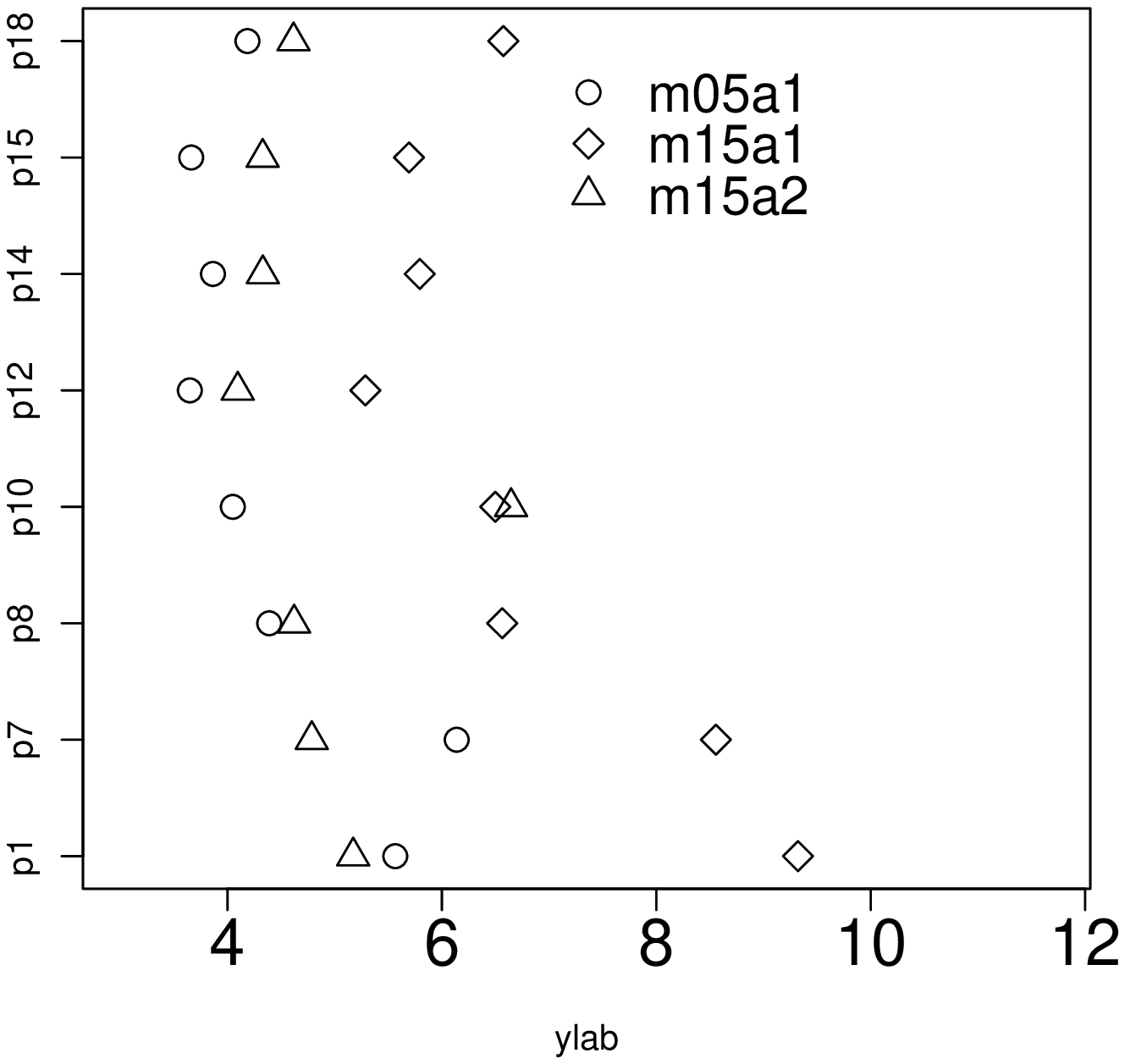}}
\psfrag{ylab}[cc]{{\footnotesize log($t$)}}
\subfigure[Base-10 logarithm of the computing time needed in seconds ($\log(t)$).] {\label{fig_mut_times}\ \hspace{0.5cm}\includegraphics[scale=.40]{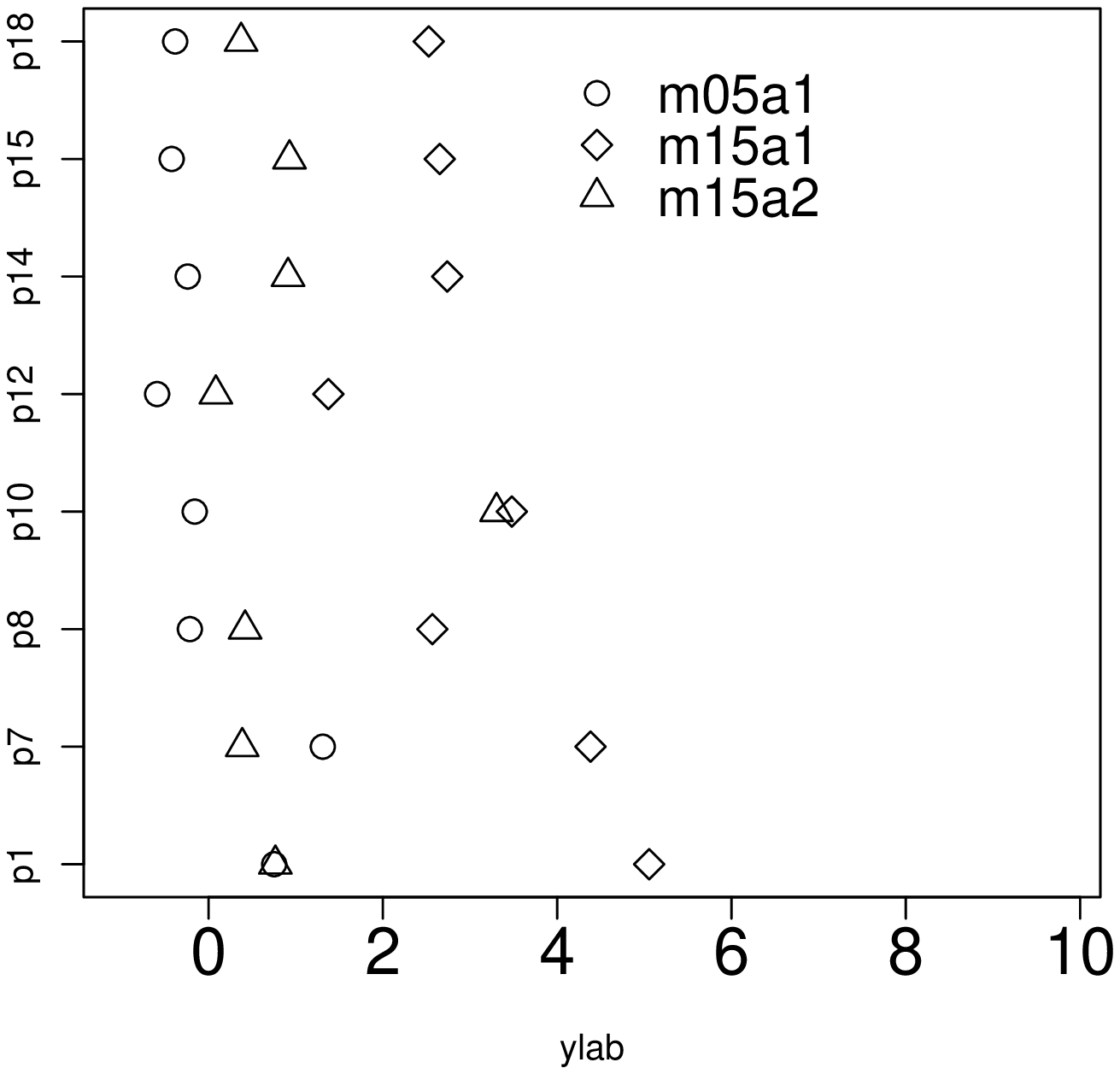}}
\end{center}
\caption{Number of runs and computing time needed to obtain a relative error below 1\% for the {\ck trees of high relative complexity} given in Figure~\ref{fig_mut_trees}.}
\label{fig_mut}
\end{figure}

\subsection{Average performance over many samples}
\label{sec_average_performance}

{\ck We simulated} 100
samples under a given pair of parameters and estimate{\ck d} the likelihood {\ck of} these
samples {\ck for the same} parameters.
Whereas for the analysis in the previous section we provided the exact number of runs, we now 
cumulated additional simulation runs 
until the relative error dropped below $1\%$, increasing the number 
of new runs by a factor of $4$ in each step. Density plots for 
the number of runs needed to achieve this are given in
Figure~\ref{importance_fig_m1_a15_hist_runs} for the parameters (1, 1.5) and in
Figure~\ref{importance_fig_m1_a2_hist_runs} for the parameters (1, 2) {\ck for selected} proposal distributions.

As before, we also measured the time required to achieve a
relative error below $1\%$ in term of the actual computing time in seconds.
The base-10 logarithms of the corresponding times are given in
Figure~\ref{importance_fig_m1_a15_hist_times} and
Figure~\ref{importance_fig_m1_a2_hist_times} for selected proposal distributions.
Since one simulated sample for $\alpha=2$ showed no mutations, we assumed a
duration of zero. For $\alpha=1.5$ (Figure~\ref{importance_fig_m1_a15_hist}) the proposal distribution $\QPairHUWoneStepB$ again performs better than the others. However, for $\alpha=2$ (Figure~\ref{importance_fig_m1_a2_hist}) 
performances are very similar 
with even a slight disadvantage for $\QPairHUWoneStepB$ in terms of computing time.

\begin{figure}
\begin{center}
\subfigure[Histogram of the base-10 logarithmic number of runs needed to obtain a 
relative error less than 1\%.]{\label{importance_fig_m1_a15_hist_runs}
\psfrag{0}{{\tiny 0}}
\psfrag{10}{{\tiny 10}}
\psfrag{20}{{\tiny 20}}
\psfrag{30}{{\tiny 30}}
\psfrag{40}{{\tiny 40}}
\psfrag{50}{{\tiny 50}}
\psfrag{60}{{\tiny 60}}
\psfrag{70}{{\tiny 70}}
\psfrag{-0.4}{{\tiny -0.4}}
\psfrag{1}{{\tiny 1}}
\psfrag{2.3}{{\tiny 2.3}}
\psfrag{3.6}{{\tiny 3.6}}
\psfrag{4.8}{{\tiny 4.8}}
\psfrag{6}{{\tiny 6}}
\psfrag{7.2}{{\tiny 7.2}}
\psfrag{porp1}{{\scriptsize $\QGT$}}
\psfrag{porp8}{{\scriptsize $\QSingleHUW$}}
\psfrag{porp12}{{\scriptsize $\QPairHUWoneStepB$}}
\psfrag{porp14}{{\scriptsize $\QPairHUWtwoStepA$}}
\psfrag{xlab}{log(\#)}
\psfrag{ylab}{count}
\includegraphics[scale=.3]{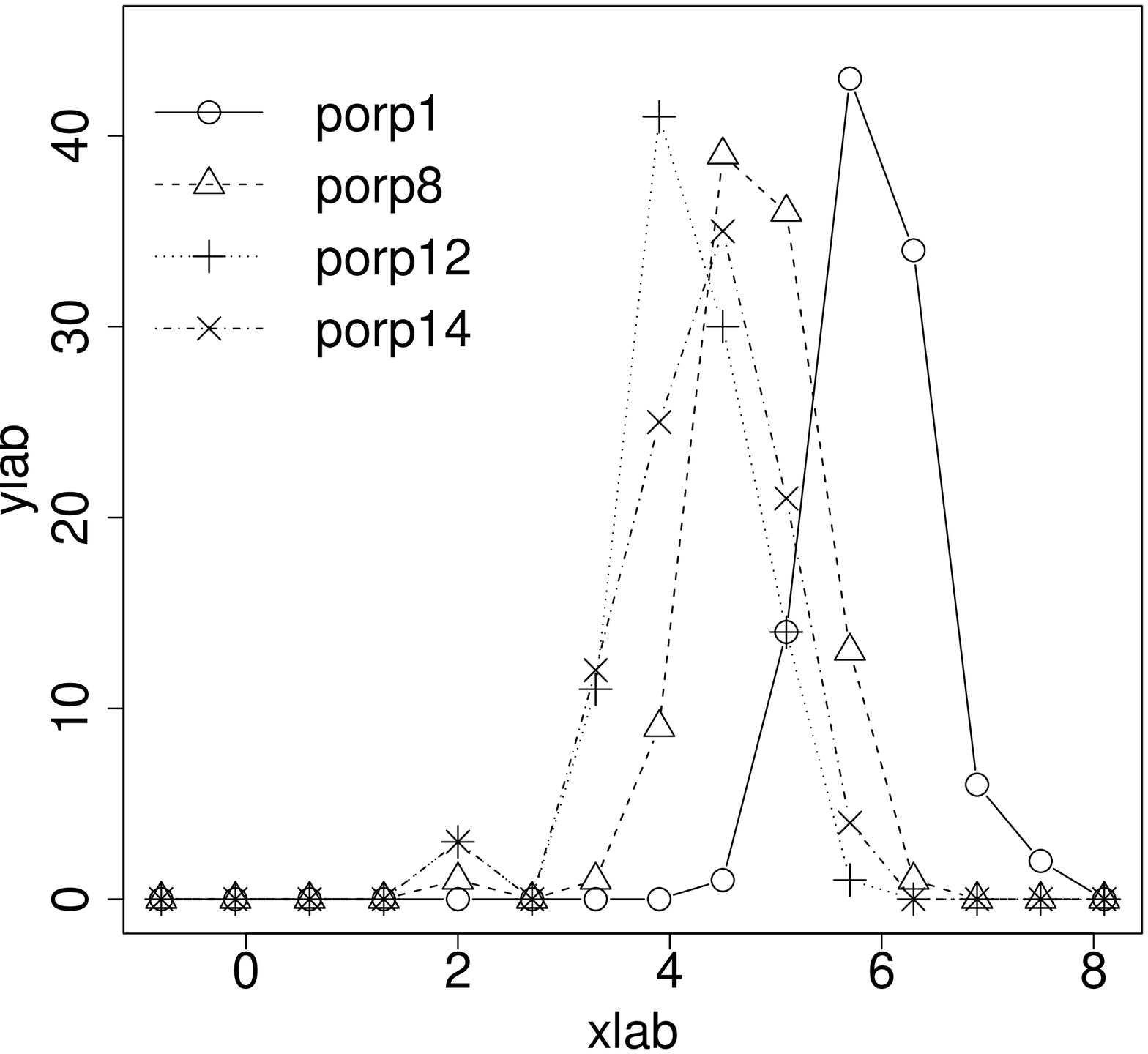}
}
\hspace{1cm}
\subfigure[Histogram of the base-10 logarithmic real-time needed to obtain a 
relative error less than 1\%.]{\label{importance_fig_m1_a15_hist_times}
\psfrag{0}{{\tiny 0}}
\psfrag{10}{{\tiny 10}}
\psfrag{20}{{\tiny 20}}
\psfrag{30}{{\tiny 30}}
\psfrag{40}{{\tiny 40}}
\psfrag{50}{{\tiny 50}}
\psfrag{60}{{\tiny 60}}
\psfrag{70}{{\tiny 70}}
\psfrag{-1.5}{{\tiny -1.5}}
\psfrag{-0.5}{{\tiny -0.5}}
\psfrag{0.5}{{\tiny 0.5}}
\psfrag{1.5}{{\tiny 1.5}}
\psfrag{2.5}{{\tiny 2.5}}
\psfrag{porp1}{{\scriptsize $\QGT$}}
\psfrag{porp8}{{\scriptsize $\QSingleHUW$}}
\psfrag{porp12}{{\scriptsize $\QPairHUWoneStepB$}}
\psfrag{porp14}{{\scriptsize $\QPairHUWtwoStepA$}}
\psfrag{xlab}{log($t$)}
\psfrag{ylab}{count}
\includegraphics[scale=.3]{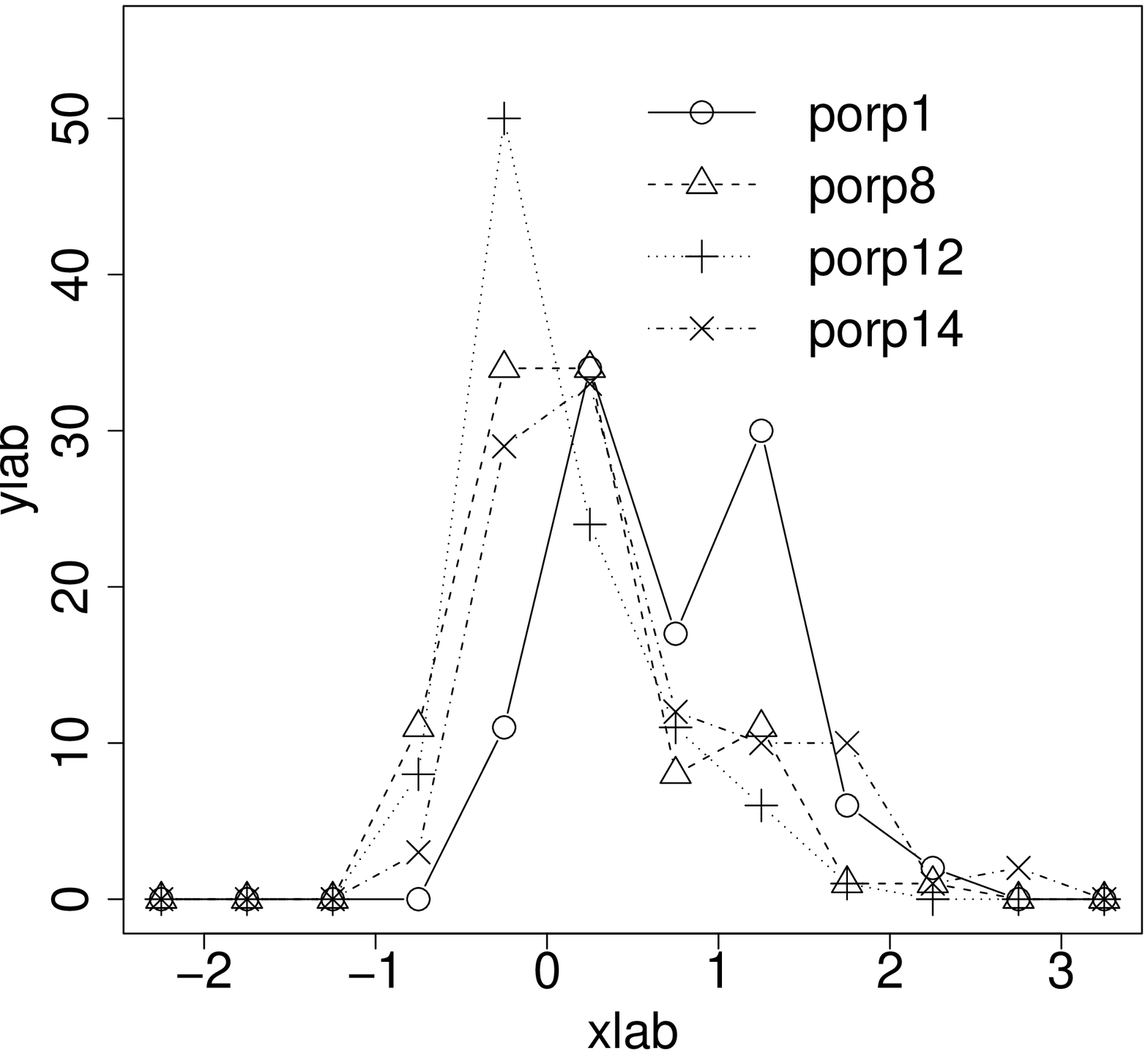}
}
\end{center}
\caption{Empirical distributions for the number of runs and the real-time for 100 
samples of size 15, simulated with $r=1$ and $\alpha=1.5$. The likelihood was computed for the same parameters.}
\label{importance_fig_m1_a15_hist}
\end{figure}

\begin{figure}
\begin{center}
\subfigure[Histogram of the base-10 logarithmic number of runs needed to obtain 
a relative error less than 1\%.]{\label{importance_fig_m1_a2_hist_runs}
\psfrag{0}{{\tiny 0}}
\psfrag{10}{{\tiny 10}}
\psfrag{20}{{\tiny 20}}
\psfrag{30}{{\tiny 30}}
\psfrag{40}{{\tiny 40}}
\psfrag{50}{{\tiny 50}}
\psfrag{60}{{\tiny 60}}
\psfrag{70}{{\tiny 70}}
\psfrag{-0.4}{{\tiny -0.4}}
\psfrag{1}{{\tiny 1}}
\psfrag{2.3}{{\tiny 2.3}}
\psfrag{3.6}{{\tiny 3.6}}
\psfrag{4.8}{{\tiny 4.8}}
\psfrag{6}{{\tiny 6}}
\psfrag{7.2}{{\tiny 7.2}}
\psfrag{porp1}{{\scriptsize $\QGT$}}
\psfrag{porp8}{{\scriptsize $\QSingleHUW$}}
\psfrag{porp12}{{\scriptsize $\QPairHUWoneStepB$}}
\psfrag{porp14}{{\scriptsize $\QPairHUWtwoStepA$}}
\psfrag{xlab}{log(\#)}
\psfrag{ylab}{count}
\includegraphics[scale=.3]{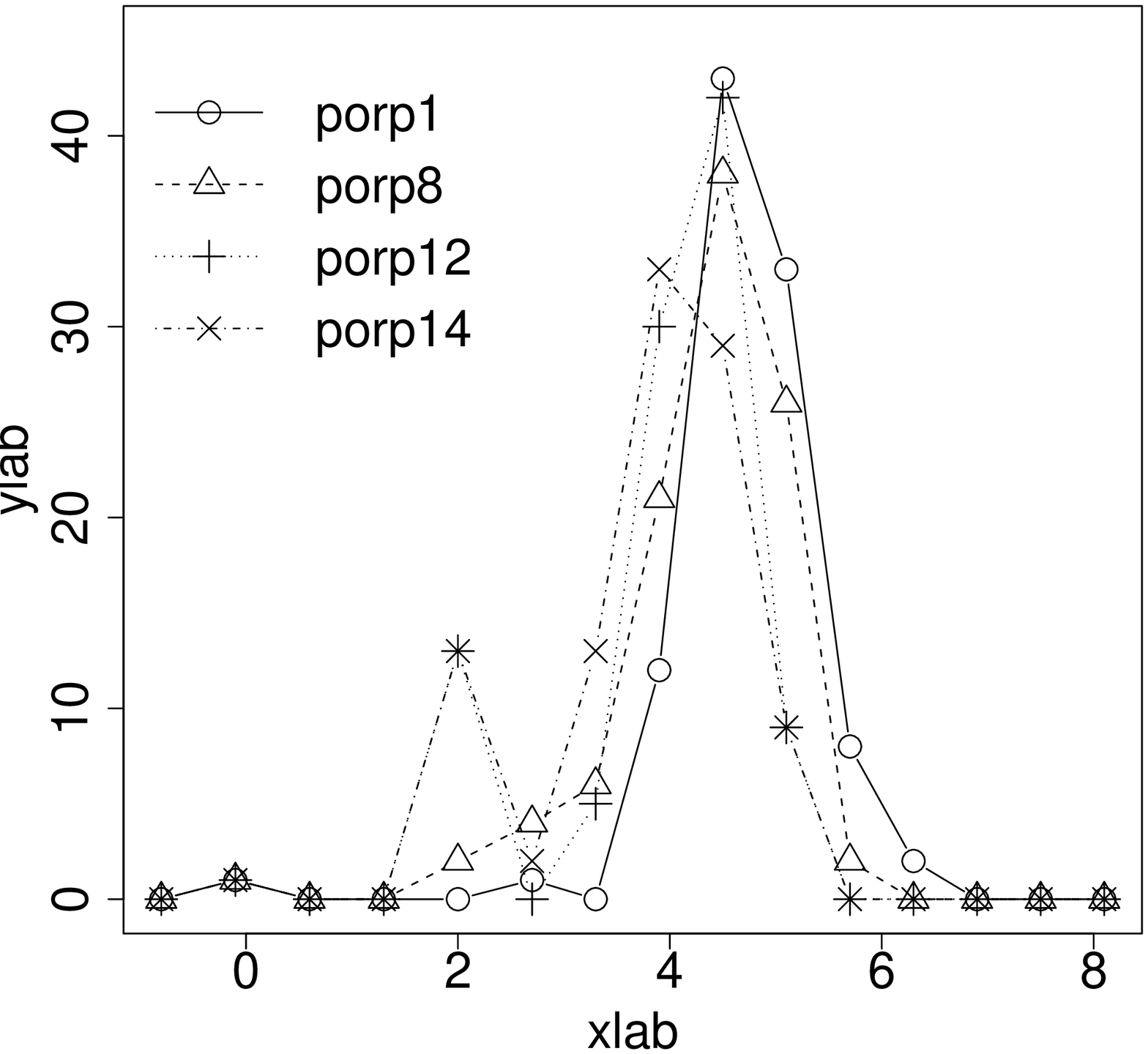}
}
\hspace{1cm}
\subfigure[Histogram of the base-10 logarithmic real-time needed to 
obtain a relative error less than 1\%.]{\label{importance_fig_m1_a2_hist_times}
\psfrag{0}{{\tiny 0}}
\psfrag{10}{{\tiny 10}}
\psfrag{20}{{\tiny 20}}
\psfrag{30}{{\tiny 30}}
\psfrag{40}{{\tiny 40}}
\psfrag{50}{{\tiny 50}}
\psfrag{60}{{\tiny 60}}
\psfrag{70}{{\tiny 70}}
\psfrag{-1.5}{{\tiny -1.5}}
\psfrag{-0.5}{{\tiny -0.5}}
\psfrag{0.5}{{\tiny 0.5}}
\psfrag{1.5}{{\tiny 1.5}}
\psfrag{2.5}{{\tiny 2.5}}
\psfrag{porp1}{{\scriptsize $\QGT$}}
\psfrag{porp8}{{\scriptsize $\QSingleHUW$}}
\psfrag{porp12}{{\scriptsize $\QPairHUWoneStepB$}}
\psfrag{porp14}{{\scriptsize $\QPairHUWtwoStepA$}}
\psfrag{xlab}{log($t$)}
\psfrag{ylab}{count}
\includegraphics[scale=.3]{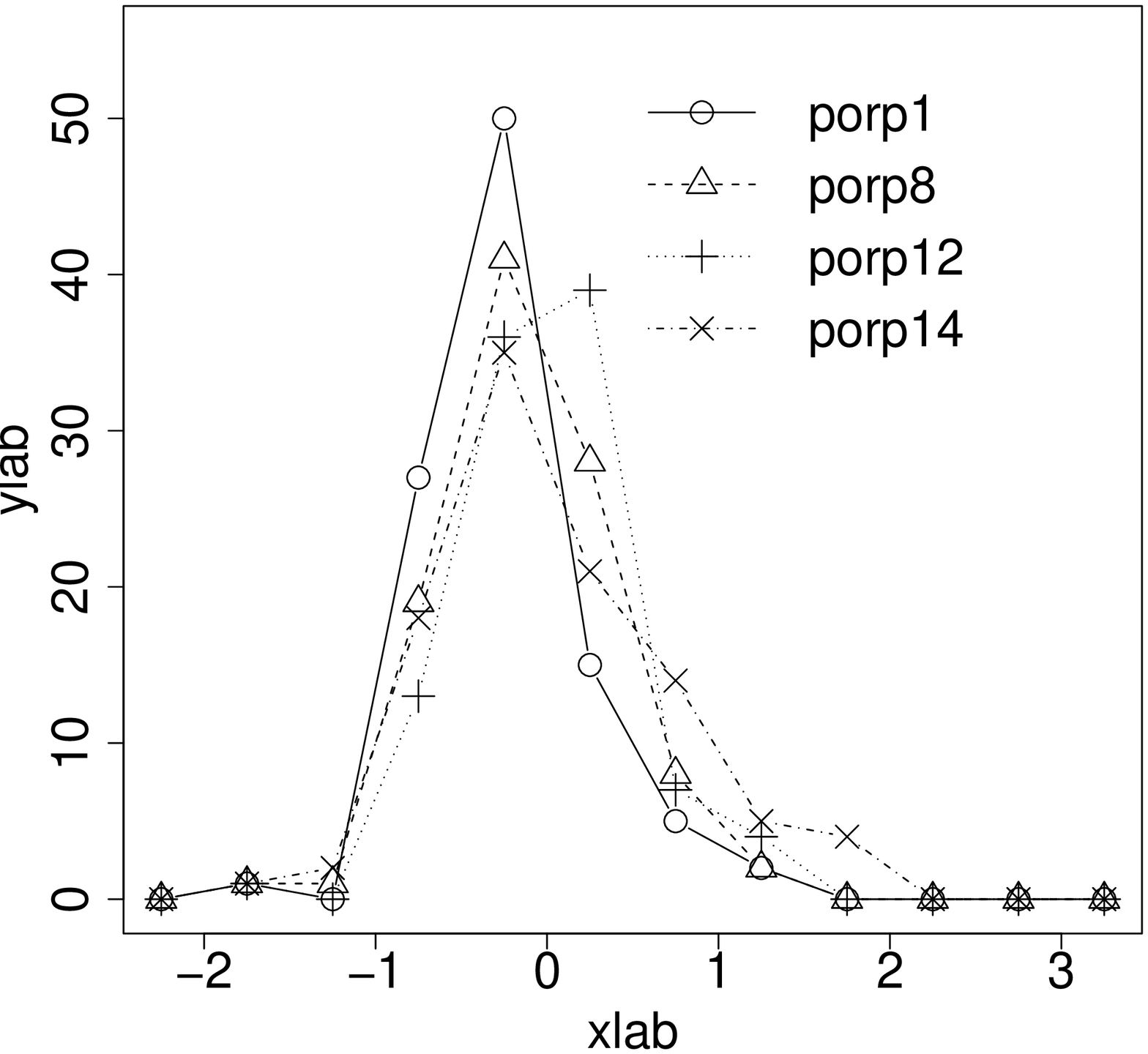}
}
\end{center}
\caption{Empirical distributions for the number of runs and the real-time for
100 samples of size 15 with, simulated with $r=1$ and $\alpha=2$. Again, the
likelihood was computed for the same parameters.}
\label{importance_fig_m1_a2_hist}
\end{figure}

\subsection{Performance on real data sets}
\label{sec_real_data}

So far we have only dealt with simulated datasets of relatively small complexity. 
We now analyse the performance of our methods on various real datasets.

We begin with a famous and well-studied dataset consisting of 
mitochondrial data sampled by Ward et.~al.\ (\cite{Ward1991})
from the North American {\em Nuu Chah Nulth} tribe. The corresponding genetree is given in Figure~\ref{fig_ncn_gt94}. 
These samples were analysed in a framework similar to ours in \cite{Griffiths1994} and \cite{Hobolth2008}, 
and we use the data in the form edited by Griffiths and Tavar\'e in \cite[Figure~3]{Griffiths1994}. 
We first estimated the maximum likelihood values for the mutation rate $r$ 
and the parameter for the Beta-coalescent $\alpha$ on a discrete grid. The values are given in Table~\ref{tab_real_performance}.
Details of this method
and possible biological implications will be discussed elsewhere. 
We then used the estimated parameters to perform the same analysis as in Section~\ref{sec_specific_structures}, that is we determined the number of independent runs and the computing time to estimate the likelihood value at this point in the parameter space with a relative error below $1\%$. The result is given in Figure~\ref{fig_first} (cf. the symbol related to
\cite{Griffiths1994}).
Again the proposal distributions using pairs of mutations show good performance when the number of runs is considered. However, this advantage almost vanishes when the total computation time is considered.
Still, $\QPairHUWoneStepB$ performs best. 

Currently, evolutionary mechanisms to describe actual biological populations which might 
give rise to Lambda-coalescent like genealogies (see e.g.~\cite{Eldon2006}) 
are being discussed. In this subsection we will further compare the performance of our methods on the 
datasets considered in \cite{Arnason2004}, namely mitochondrial cytochrome {\em b} DNA variation data sampled from various subpopulations of Atlantic Cod {\em (Gadus Morhua)}. These datasets, depicted in Figure~\ref{fig_cod1} and Figure~\ref{fig_cod2}, are taken from \cite{Arnason1996}, \cite{Arnason1998} (only from the Baltic transition area), \cite{Arnason2000} (only the Greenland subsample), \cite{Carr1991}, \cite{Pepin1993} and \cite{Sigurgislason2003} (only cyt {\em b} data).

Again, we estimated the maximum likelihood values for the mutation rate $r$ and the 
parameter for the Beta-coalescent $\alpha$ on a discrete grid and proceeded in a similar 
way as for the Nuu Chah Nulth data. The estimated parameter values are given in Table~\ref{tab_real_performance} and Figure~\ref{fig_first}\footnote{The analysis for \cite{Griffiths1994} under $\QPairHUWoneStepA$ showed a relative error of 7 \% after 27 million runs taking 32 days.} and Figure~\ref{fig_second} show the results of the runtime analysis.

Again the proposal distributions using pairs of mutations show a strong performance when the number of runs is considered. However, this advantage vanishes when the computation time is considered, 
where for some samples $\QSD$ and $\QSingleHUW$ even perform better. 
To some extend this can be attributed to the increased effort the proposal 
distributions using pairs of mutations have to invest in the precalculation.

\begin{figure}
\begin{center}
\psfrag{1}{$\scriptstyle 1$}
\psfrag{2}{$\scriptstyle 2$}
\psfrag{3}{$\scriptstyle 3$}
\psfrag{4}{$\scriptstyle 4$}
\psfrag{5}{$\scriptstyle 5$}
\psfrag{8}{$\scriptstyle 8$}
\psfrag{19}[Bl][Bc]{$\scriptstyle 19$}
\fbox{\includegraphics[height=3cm, width=4.5cm]{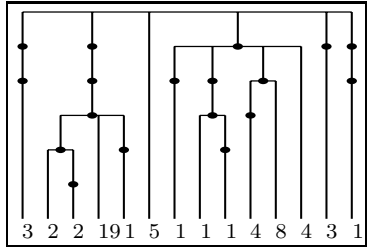}}
\end{center}
\caption{The genetree corresponding to the dataset from \cite{Griffiths1994}.}
\label{fig_ncn_gt94}
\end{figure}

\begin{figure}
\begin{center}
\psfrag{1}{$\scriptstyle 1$}
\psfrag{2}{$\scriptstyle 2$}
\psfrag{3}{$\scriptstyle 3$}
\psfrag{4}{$\scriptstyle 4$}
\psfrag{6}{$\scriptstyle 6$}
\psfrag{7}{$\scriptstyle 7$}
\psfrag{8}{$\scriptstyle 8$}
\psfrag{12}[Bl][Bc]{$\scriptstyle 12$}
\psfrag{14}[Bl][Bc]{$\scriptstyle 14$}
\psfrag{19}{$\scriptstyle 19$}
\psfrag{25}[Bl][Bc]{$\scriptstyle 25$}
\psfrag{35}{$\scriptstyle 35$}
\psfrag{48}{$\scriptstyle 48$}
\psfrag{62}[Bl][Bc]{$\scriptstyle 62$}
\subfigure[\lbrack AP96\rbrack] {\label{fig_cod1_arnason1996}
\fbox{\includegraphics[height=3cm, width=4.5cm]{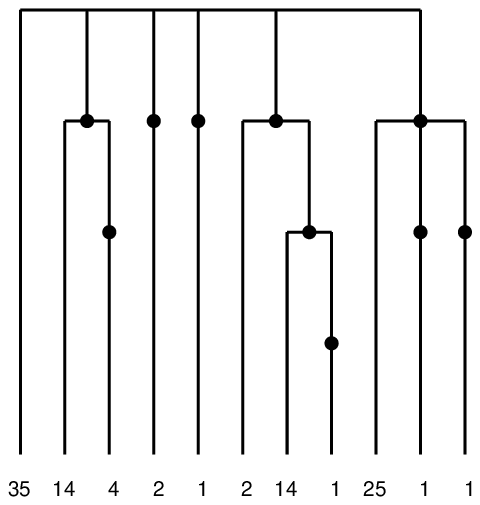}}
}
\subfigure[\lbrack APP98\rbrack] {\label{fig_cod1_arnason1998}
\fbox{\includegraphics[height=3cm, width=4.5cm]{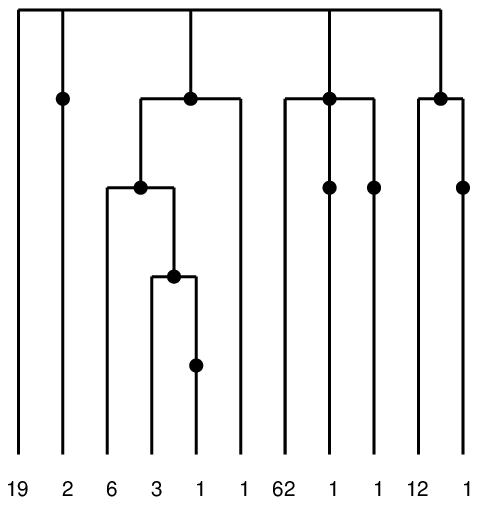}} 
}
\subfigure[\lbrack APKS00\rbrack] {\label{fig_cod1_arnason2000}
\fbox{\includegraphics[height=3cm, width=4.5cm]{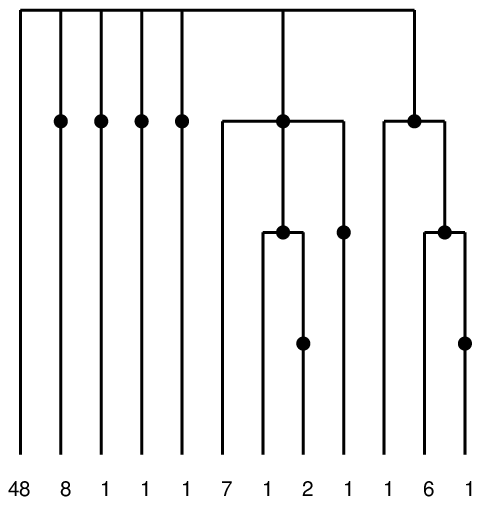}} 
}
\end{center}
\caption{The genetrees corresponding to the datasets from \cite{Arnason1996}, \cite{Arnason1998} and \cite{Arnason2000}.}
\label{fig_cod1}
\end{figure}

\begin{figure}
\begin{center}
\psfrag{1}{$\scriptstyle 1$}
\psfrag{2}{$\scriptstyle 2$}
\psfrag{3}{$\scriptstyle 3$}
\psfrag{4}{$\scriptstyle 4$}
\psfrag{6}{$\scriptstyle 6$}
\psfrag{10}[Bl][Bc]{$\scriptstyle 10$}
\psfrag{11}[Bl][Bc]{$\scriptstyle 11$}
\psfrag{13}[Bl][Bc]{$\scriptstyle 13$}
\psfrag{26}{$\scriptstyle 26$}
\psfrag{36}{$\scriptstyle 36$}
\psfrag{84}{$\scriptstyle 84$}
\subfigure[\lbrack CM91\rbrack] {\label{fig_cod2_carr1991}
\fbox{\includegraphics{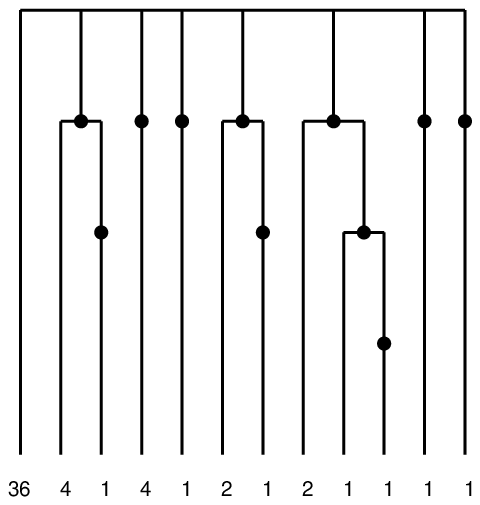}}
}
\subfigure[\lbrack PC93\rbrack] {\label{fig_cod2_pepin1993}
\fbox{\includegraphics{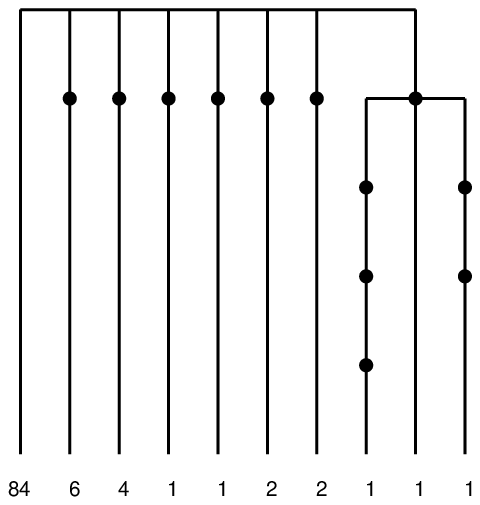}} 
}
\subfigure[\lbrack SA03\rbrack] {\label{fig_cod2_sigurgislason2003}
\fbox{\includegraphics{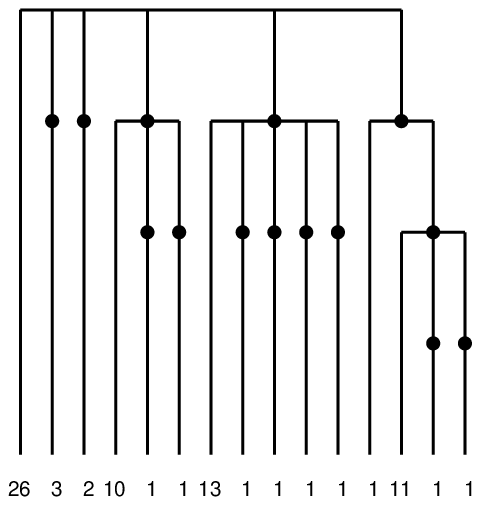}} 
}
\end{center}
\caption{The genetrees corresponding to the datasets from \cite{Carr1991}, \cite{Pepin1993} and \cite{Sigurgislason2003}.}
\label{fig_cod2}
\end{figure}

\begin{table}[h!]
\begin{center}
\scriptsize
\begin{tabular}{||c||r|r|r|r|r|r|r||}
\hline\hline
& \cite{Griffiths1994} & \cite{Arnason1996} & \cite{Arnason1998} & \cite{Arnason2000} & \cite{Carr1991} & \cite{Pepin1993} & \cite{Sigurgislason2003} \\\hline\hline
$n$ & 55 & 100 & 109 & 78 & 55 & 103 & 74\\
\hline
$(\hat{r},\hat{\alpha})$ & (2.4, 2.0) & (0.7, 1.65) & (0.6, 1.55) & (0.7, 1.65) & (0.8, 1.4) & (0.6, 1.4) & (0.7, 1.3) \\
\hline
$p^0[{\bf t},{\bf n}]$ & $9.02 \cdot 10^{-20}$ & $2.25 \cdot 10^{-13}$ & $2.19 \cdot 10^{-14}$ & $2.26 \cdot 10^{-12}$ & $3.80 \cdot 10^{-9}$ & $1.64 \cdot 10^{-10}$ & $6.44 \cdot 10^{-13}$ \\
\hline
$c({\bf t},{\bf n})$ & 1 & 2 & 2 & 6 & 6 & 4 & 96 \\
\hline
$p[{\bf t},{\bf n}]$ & $9.02 \cdot 10^{-20}$ & $1.13 \cdot 10^{-13}$ & $1.10 \cdot 10^{-14}$ & $3.77 \cdot 10^{-13}$ & $6.33 \cdot 10^{-10}$ & $4.10 \cdot 10^{-11}$ & $6.71 \cdot 10^{-15}$ \\
\hline\hline
\end{tabular}
\caption{True probabilities $p[{\bf t},{\bf n}]$ under 
estimated ML parameters (within the Beta$(2-\alpha,\alpha)$-class; 
MLE on a discrete grid)
combinatorial factors $c({\bf t},{\bf n})$, and likelihoods $p[{\bf t},{\bf n}]$ for the real datasets.}
\label{tab_real_performance}
\end{center}
\end{table}

\begin{figure}
\begin{center}
\psfrag{p1}[cr][c][1][270]{{\scriptsize $\QGT$}}
\psfrag{p7}[cr][c][1][270]{{\scriptsize $\QSD$}}
\psfrag{p8}[cr][c][1][270]{{\scriptsize $\QSingleHUW$}}
\psfrag{p10}[cr][c][1][270]{{\scriptsize $\QPairHUWoneStepA$}}
\psfrag{p12}[cr][c][1][270]{{\scriptsize $\QPairHUWoneStepB$}}
\psfrag{p14}[cr][c][1][270]{{\scriptsize $\QPairHUWtwoStepA$}}
\psfrag{p15}[cr][c][1][270]{{\scriptsize $\QPairHUWtwoStepB$}}
\psfrag{p18}[cr][c][1][270]{{\scriptsize $\QPairHUWmixed$}}
\psfrag{nuh_chah_nulth}{{\tiny \cite{Griffiths1994}}}
\psfrag{arnason1996}{{\tiny \cite{Arnason1996}}}
\psfrag{arnason1998}{{\tiny \cite{Arnason1998}}}
\psfrag{xlab}[cc]{{\footnotesize Distribution}}
\psfrag{ylab}[cc]{{\footnotesize log(\#)}}
\subfigure[Base-10 logarithm of the number of runs needed ($\log(\#)$).] {\label{fig_first_runs}
\includegraphics[scale=.40]{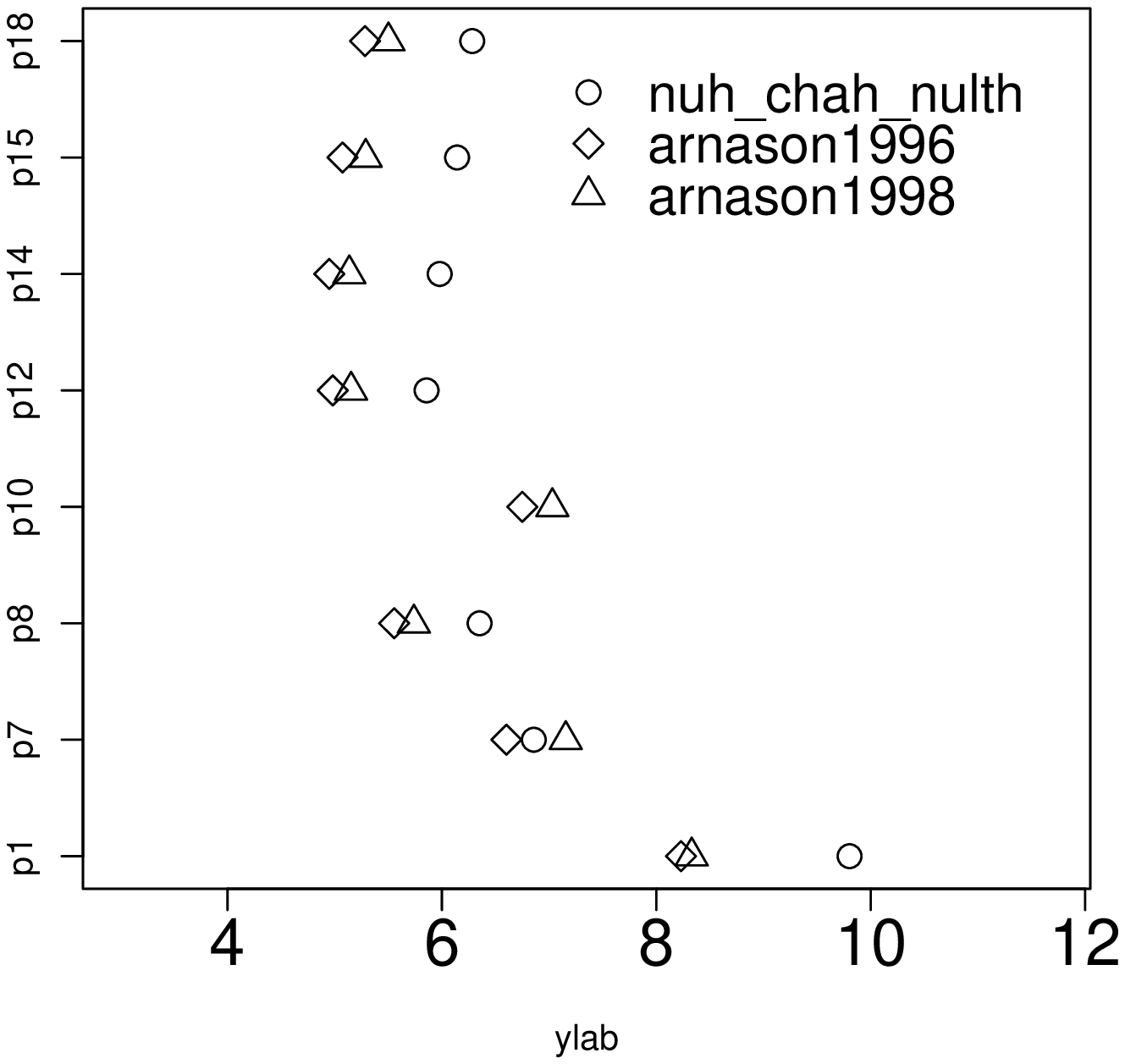}}
\psfrag{ylab}[cc]{{\footnotesize log($t$)}}
\subfigure[Base-10 logarithm of the computing time needed in seconds ($\log(t)$).] {\label{fig_first_times}\ \hspace{0.5cm}\includegraphics[scale=.40]{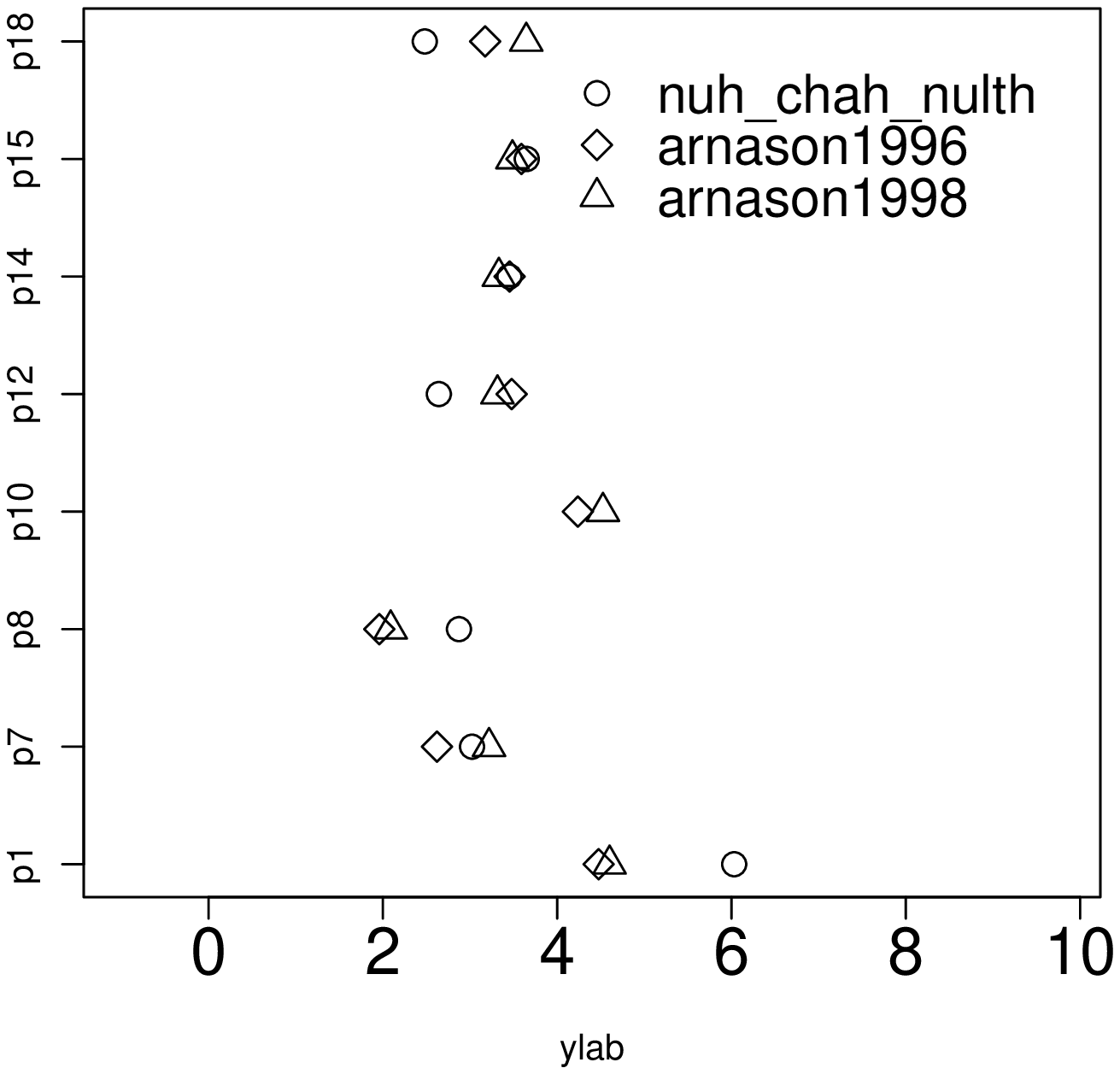}}
\end{center}
\caption{\small Number of runs and computing time needed to obtain a relative error below 1\% for the genetrees corresponding to the datasets from \cite{Griffiths1994}, \cite{Arnason1996} and \cite{Arnason1998} given in Figure~\ref{fig_ncn_gt94} and~\ref{fig_cod1}.}
\label{fig_first}
\end{figure}

\begin{figure}
\begin{center}
\psfrag{p1}[cr][c][1][270]{{\scriptsize $\QGT$}}
\psfrag{p7}[cr][c][1][270]{{\scriptsize $\QSD$}}
\psfrag{p8}[cr][c][1][270]{{\scriptsize $\QSingleHUW$}}
\psfrag{p10}[cr][c][1][270]{{\scriptsize $\QPairHUWoneStepA$}}
\psfrag{p12}[cr][c][1][270]{{\scriptsize $\QPairHUWoneStepB$}}
\psfrag{p14}[cr][c][1][270]{{\scriptsize $\QPairHUWtwoStepA$}}
\psfrag{p15}[cr][c][1][270]{{\scriptsize $\QPairHUWtwoStepB$}}
\psfrag{p18}[cr][c][1][270]{{\scriptsize $\QPairHUWmixed$}}
\psfrag{arnason2000}{{\tiny \cite{Arnason2000}}}
\psfrag{carr1991}{{\tiny \cite{Carr1991}}}
\psfrag{pepin1993}{{\tiny \cite{Pepin1993}}}
\psfrag{sigurgislason2003}{{\tiny \cite{Sigurgislason2003}}}
\psfrag{xlab}[cc]{{\footnotesize Distribution}}
\psfrag{ylab}[cc]{{\footnotesize log(\#)}}
\subfigure[Base-10 logarithm of the number of runs needed ($\log(\#)$).] {\label{fig_second_runs}
\includegraphics[scale=.40]{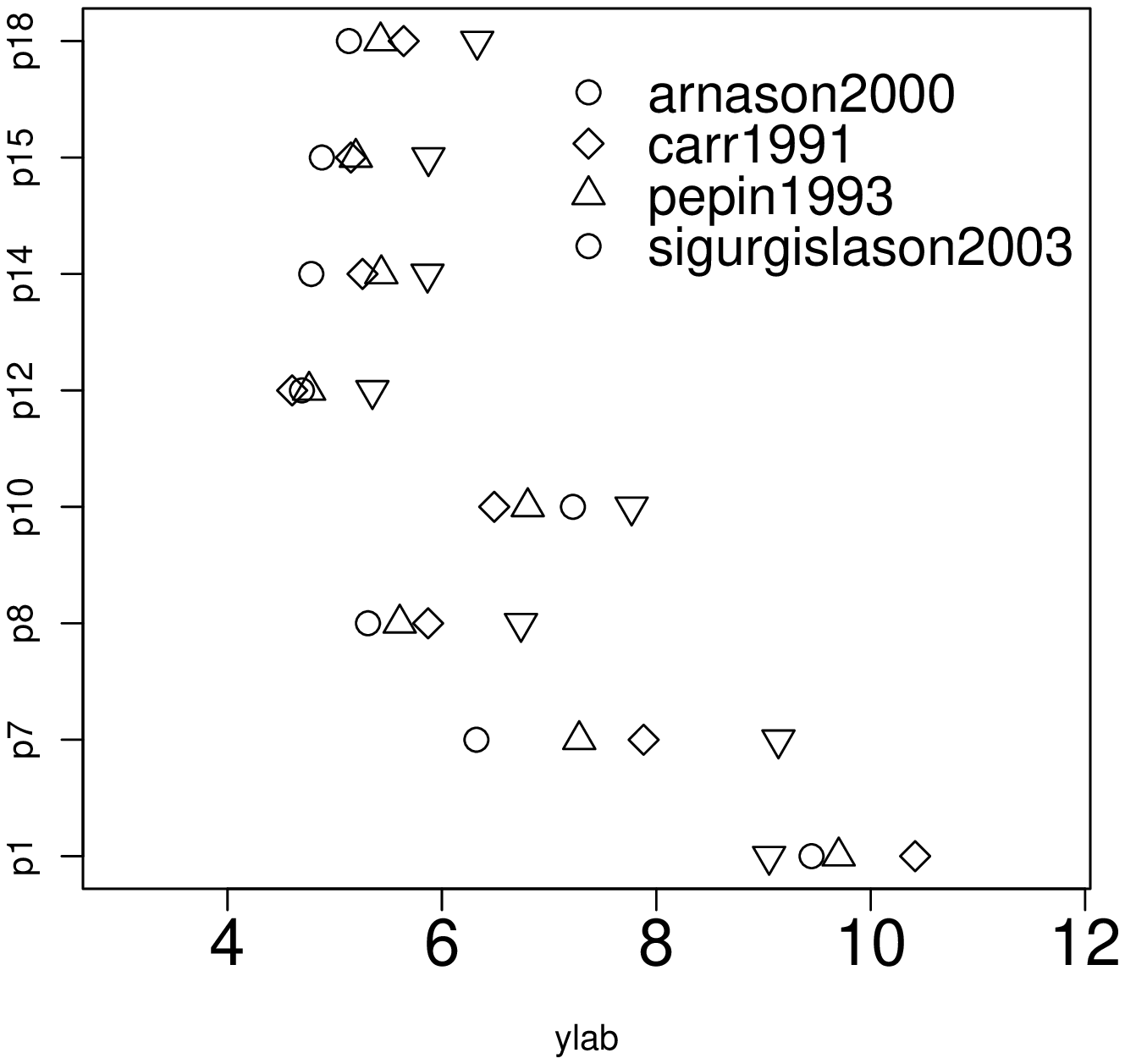}}
\psfrag{ylab}[cc]{{\footnotesize log($t$)}}
\subfigure[Base-10 logarithm of the computing time needed in seconds ($\log(t)$).] {\label{fig_second_times}\ \hspace{0.5cm}\includegraphics[scale=.40]{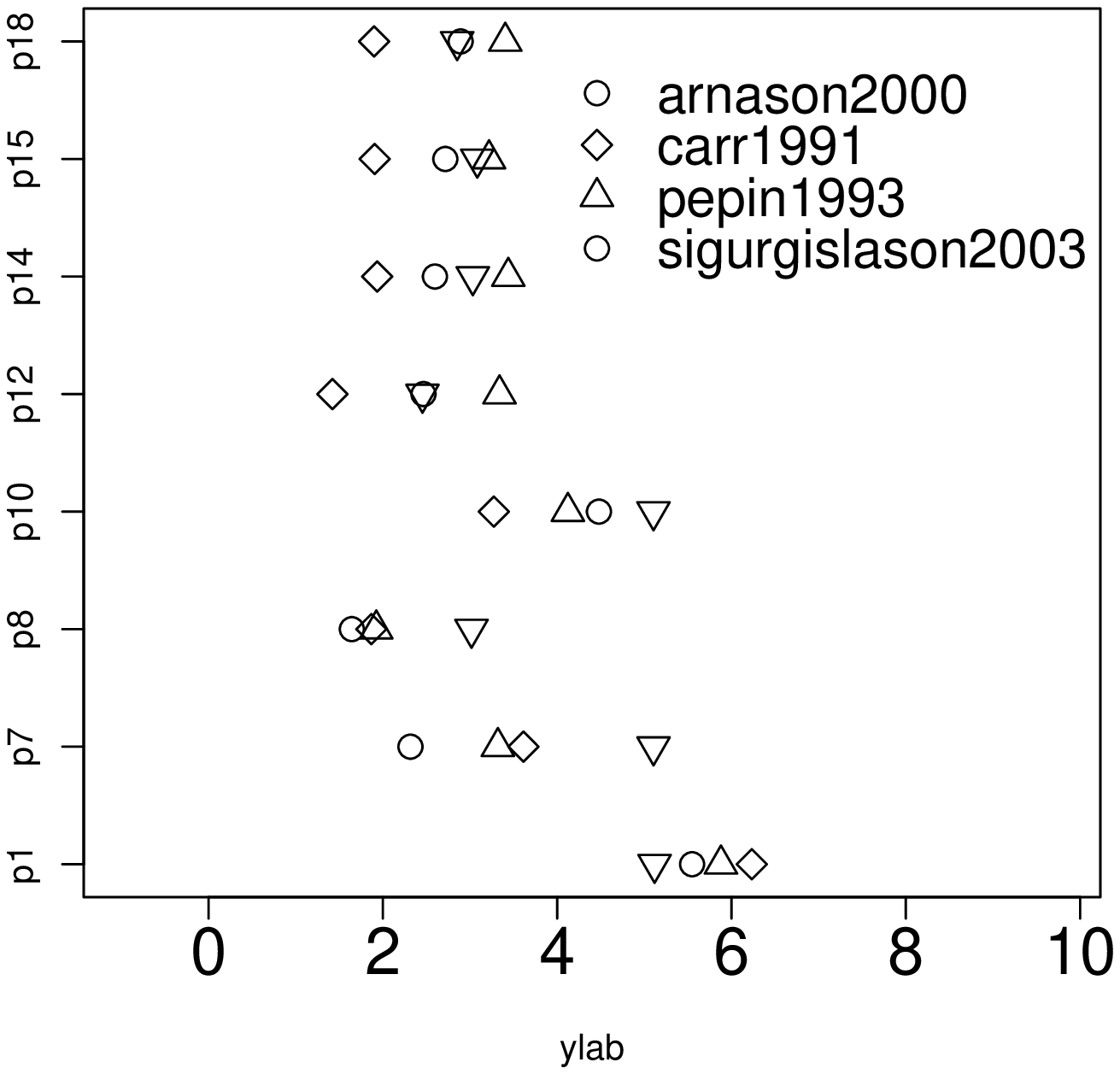}}
\end{center}
\caption{\small Number of runs and computing time needed to obtain a relative error below 1\% for the genetrees corresponding to the datasets from \cite{Arnason2000}, \cite{Carr1991}, \cite{Pepin1993} and \cite{Sigurgislason2003} given in Figure~\ref{fig_cod1} and~\ref{fig_cod2}.}
\label{fig_second}
\end{figure}

\subsection{Conclusion and guidelines for the practitioner}

Table~\ref{tb:dist} shows that for a wide range of parameters and samples with complexity~15, the proposal distributions using pairs of mutations are typically closer to the optimal proposal distribution than
the proposal distributions using less detailed information from the sample. 
The distribution $\QSD$ performs better then the 'standard' $\QGT$, but is in turn outperformed by $\QSingleHUW$. 
This relative ranking of distributions in principle holds throughout the subsequent 
analysis in Sections~\ref{sec_specific_structures}, \ref{sec_average_performance} and \ref{sec_real_data}. 
The proposal distributions using pairs of mutations perform consistently better than the others 
when the number of independent runs is considered. Note that $\QPairHUWoneStepA$ is an exception to this rule 
(this fits to the observation from p.~\pageref{page:someone_underrates_mutation_events} that $\QPairHUWoneStepA$ underrates mutation events). 

However, when considering overall computation time, this clear picture changes. 
Though the methods using compressed genetrees still outperform $\QGT$ and $\QSD$ in most cases,
$\QSingleHUW$ shows a performance comparable to our methods using pairs of mutations. 
On the one hand this can be attributed to the actual
implementation, on the other hand the computational complexities per proposal step for different
proposal distributions do differ, ranging from 
constant 
($\QGT$ and $\QSD$) 
to linear ($\QSingleHUW$ and $\QPairHUWoneStepB$) or quadratic in the number
of segregating sites. When real computation time is considered, the proposal 
distribution $\QSingleHUW$ seems to make up for the lack in accuracy by smaller computation time for each step when compared to the pair-wise methods.

A further increase in the 
runtime of the proposal distributions regarding pairs of mutations needs to be attributed  
to the fact that they require precalculation
of all steps in all samples with up to two segregating sites. 
When analysing the more complex real data sets of the previous section, 
this precalculation becomes a substantial component of the total computing time. 
For example, 
our current implementation needed about 38 seconds for the precalculation of samples of size 50, 
but this rapidly increases to 4250 seconds for samples of size 100. 
In contrast, the proposal distribution $\QSingleHUW$ only
requires precalculations for samples with one or zero segregating sites, 
which is negligible for samples of size 100.

Still, if a sample configuration can be analysed by the proposal
distribution $\QPairHUWoneStepB$, then this proposal distribution
yields a good performance. Furthermore, when several datasets are to be 
analysed, the program {\tt MetaGeneTree} allows to save computing time by storing 
the precalculated optimal proposal weights in a file.

In conclusion one can say that the methods using compressed genetrees present 
an improvement over 
the `canonical candidate' $\QGT$ or the heuristic generalisation of Stephens and Donnelly's idea for the Lambda-case, $\QSD$. 
For small to moderate sizes the pair-wise methods perform rather well with $\QPairHUWoneStepB$ outperforming every other method. 

In general, which proposal distribution works best in terms of real-time requirements depends on the particular data set and the parameters. 
Thus, for larger datasets, we recommend a small preparatory study to 
test the performance of the various methods. 
This can easily be done with {\tt MetaGeneTree}.

\color{black}

\section*{Acknowledgement}

The research of 
M.S.\ was supported in part by a DFG IRTG 1339 scholarship 
and NIH grant R00-GM080099. 
M.B.\ would like to thank Asger Hobolth for a very stimulating discussion 
which initiated this research. 


\color{black}


\appendix
\vspace{3ex}\noindent\Large{{\bf Appendix}}
\setcounter{section}{1}
\label{s:app}

\subsection{Generating samples: Details} 
\label{sect:algo}

The following is adapted from \cite[Section~7]{Birkner2008}. 
Let $\{ \Pi_t\}_{t \ge 0}$ be a $\Lambda$-coalescent. We denote by
$\{Y_t\}_{t \ge 0}$ the corresponding {\em block counting process},
i.e.\ $Y_t= \#\{\text{blocks of $\Pi_t$}\}$  is a continuous-time Markov chain on $\N$ with jump rates
\[ 
q_{ij} = {i \choose i-j+1} \lambda_{i, i-j+1}, \quad i > j \ge 1.
\]
The total jump rate 
while in $i$ is of course $-q_{ii} = \sum_{j=1}^{i-1} q_{ij}$. 
We write 
\begin{equation} 
p_{ij} := \frac{q_{ij}}{-q_{ii}}
\end{equation}
for the jump probabilities of the {\em skeleton chain}, noting that
$(p_{ij})$ is a stochastic matrix.
Note that in order to reduce $i$ classes to $j$ classes, an
$i-j+1$-merger has to occur.
Let 
\begin{equation} 
\label{blue}
g(n,m) := \E_n \bigg[ \int_0^\infty {\bf 1}_{\{Y_s = m \}} \, ds \bigg] 
\quad \mbox{for $n \ge m \ge 2$} 
\end{equation}
be the expected amount of time that $Y$, starting from $n$, spends in $m$.
Decomposing according to the first jump of $Y$, we find the following 
set of equations for $g(n,m)$: 
\begin{eqnarray} 
\label{eq:GFrec1}
g(n,m) & = & \sum_{k=m}^{n-1} p_{nk}
g(k,m), \quad 
n > m \ge 2, \\
\label{eq:GFrec2}
g(m,m) & = & \frac{1}{-q_{mm}}, \quad m \ge 2.  
\end{eqnarray} 
Let
us write $Y^{(n)}$ for the process starting from $Y^{(n)}_0 = n$.  Let
$\tau := \inf\{ t : Y^{(n)}_t = 1\}$ be the time required to come down
to only one class, and let
\[
\tilde{Y}^{(n)}_t := Y^{(n)}_{(\tau-t)-}, \quad 0 \le t < \tau
\] 
be the time-reversed path, where we define $\tilde{Y}^{(n)}_t =
\partial$, some cemetery state, when $t \ge \tau$.

With the above definitions, $\tilde{Y}^{(n)}$ is a continuous-time
Markov chain on $\{2,\dots,n\} \cup \{\partial\}$ with jump rates
\begin{equation} 
\label{eq:reversedjumprates}
\tilde{q}^{(n)}_{ji} = \frac{g(n,i)}{g(n,j)}q_{ij}, \quad j < i \le n, 
\end{equation}
and $\tilde{q}^{(n)}_{n\partial} = -q_{nn}$, where $g(n,m)$ is as in (\ref{blue}).
The {\em  starting distribution} of $\tilde{Y}^{(n)}$ is given by 
\[ 
\Pr \{ \tilde{Y}^{(n)}_0 = k \} = g(n,k) q_{k1},
\]
for each $k$.
We write $|{\bf n}|:=\sum_{i=1}^d n_i$, and denote
$\tilde q_k^{(n)} := -\tilde q_{kk}^{(n)}$. 

{\coleins \smallskip
Note that 
\begin{equation} 
\label{eq:reversedjumprate1}
{\colzwo \tilde q}_k^{(n)} = -q_{kk}, \quad 2 \leq k \leq n, 
\end{equation} 
i.e., the total jump rate of $\tilde{Y}^{(n)}$ in state $k \leq n$
does not depend on $n$.  \eqref{eq:reversedjumprate1} follows from the
observation that by monotonicity of paths, the set of times that $Y^{(n)}$
(and thus $\tilde{Y}^{(n)}$) spends in a given state $k$ is a.s.\ an
interval (possibly empty), thus
\begin{align*} 
\frac{1}{-q_{kk}} = \frac{\E \Big[ \int_0^\infty {\bf 1}_{\{Y^{(n)}_s = k \}} 
\, ds \Big]}{\P\big\{ \exists \,s \; : \: Y^{(n)} = k \big\}} 
= \frac{\E \Big[ \int_0^\infty {\bf 1}_{\{\tilde{Y}^{(n)}_s = k \}} 
\, ds \Big]}{\P\big\{ \exists \,s \; : \: \tilde{Y}^{(n)} = k \big\}} 
= \frac{1}{{\colzwo \tilde q}_k^{(n)}}
\end{align*}
because the hitting probability and the length of the time interval 
spent in $k$ are the same for the path $Y^{(n)}$ and its time-reversal.
\smallskip

Let $(\bar{Y}^{(n)}_\ell)_{\ell=0,1,2,\dots}$ be the skeleton chain of
the time-reversed block counting process. We parametrise time for
$\bar{Y}^{(n)}$ in such a way that $\bar{Y}^{(n)}_0=1$ and
$\bar{Y}^{(n)}_1=\tilde{Y}^{(n)}_0$.  Thus, $\bar{Y}^{(n)}$ is a
Markov chain on $\{1,2,\dots,n\} \cup \{\partial\}$ with transition
matrix $\bar{p}^{(n)}_{1 k}=g(n,k) q_{k1}$ ($2 \leq k \leq n$),
$\bar{p}^{(n)}_{j i}= \tilde{q}^{(n)}_{ji}/\tilde q_j^{(n)}$ ($2 \leq
j < i \leq n$),
$\bar{p}^{(n)}_{n,\partial}=1=\bar{p}^{(n)}_{\partial,\partial}$.
\smallskip

The time-reversed block counting processes corresponding 
to different `target' sample sizes are related as follows: 
For $n_1 < n_2$ and any 
$\ell_0=1< \ell_1 < \cdots < \ell_m \leq n_1$, we have 
\begin{equation}
\label{eq:comparereversedpathprob1}
\P\big\{\bar{Y}^{(n_1)}_i = \ell_i, i=0,\dots,m\big\} = 
\frac{g(n_1,\ell_m)}{g(n_2,\ell_m)}
\P\big\{\bar{Y}^{(n_2)}_i = \ell_i, i=0,\dots,m\big\},
\end{equation}
in particular, for $\ell \leq n_1 \leq n_2$, 
\begin{equation} 
\label{eq:eq:comparereversedhittingprob}
g(n_2, \ell)
\P\big\{\bar{Y}^{(n_1)} \; \mbox{hits $\ell$} \big\} 
= g(n_1,\ell)
\P\big\{\bar{Y}^{(n_2)} \; \mbox{hits $\ell$} \big\}.
\end{equation}
To see \eqref{eq:comparereversedpathprob1} note that 
for $1=\ell_0 < \cdots < \ell_m \leq n$ 
\begin{align*}
g(n,\ell_1) q_{\ell_1,1} \prod_{i=1}^{m-1} \tilde{q}^{(n)}_{\ell_i \ell_{i+1}}
= g(n,\ell_1) q_{\ell_1,1} \prod_{i=1}^{m-1} 
\frac{g(n,\ell_{i+1})}{g(n,\ell_{i})} q_{\ell_{i+1} \ell_i} 
= g(n,\ell_{m}) \prod_{i=0}^{m-1} q_{\ell_{i+1} \ell_i}, 
\end{align*}
dividing both sides by $\prod_{i=1}^{m-1} \tilde{q}^{(n)}_{\ell_i} 
= \prod_{i=1}^{m-1} (-q_{\ell_i \ell_i})$ gives 
\begin{align*} 
\prod_{i=0}^{m-1} \bar{p}^{(n)}_{\ell_i \ell_{i+1}} = 
g(n,\ell_{m}) q_{\ell_{m} \ell_{m-1}} \prod_{i=0}^{m-2} p_{\ell_{i+1} \ell_i} 
= \frac{g(n,\ell_{m})}{-q_{\ell_m \ell_m}} 
\prod_{i=0}^{m-1} p_{\ell_{i+1} \ell_i}. 
\end{align*}
}

\medskip

The law of the sequence $(Z_0:=([(0)]_\sim, (\{1\})), Z_1,\dots,Z_c)$
generated by Algorithm~\ref{appendix_algorithm_sample} is that of the sample histories described in
Section~\ref{sseq:iii}. Note that it agrees with
\cite[Algorithm~1]{Birkner2008} except for the way the ordering of the
types is generated.
\medskip 

\begin{algorithm}
\begin{minipage}{\textwidth}
\noindent 
\begin{itemize}
\item[1)] Draw $K$ according to the law of $\tilde Y_0^{(n)}$, i.e.
 $\Pr \{K=k\}=g(n,k)q_{k1}$.  Begin with a single `ancestral
 type' with multiplicity $K$, i.e.~$\mathbf{t}=({\bf x}_1), {\bf
   x}_1=0, 
 {\bf n}=(K)$, and so $d=1$.
 Set $s:=1$. 

 $c:=1$, $Z_c := ({\bf t}, (K))$. 
 \smallskip

\item[2)] Given $Z_c=(\mathbf{t},{\bf n})$ with $d$ types, let $k:=|{\bf n}|$, and draw a
 uniform random variable $U$ on $[0,1]$. 
\begin{itemize}
\item[$\circ$] If $U \le \frac{k r}{k r + \tilde{q}^{(n)}_{k}}$, then 
 draw one type, say $I$, according to the present frequencies. 
 \begin{itemize}
 \item[-] If $n_I = 1$, $Z_{c+1}$ arises from $Z_c$ by replacing $\mathbf{x}_I$ by 
   $(s, x_{I0}, \dots, x_{Ij(I)})$. Increase $s$ by $1$. 
 \item[-] If $n_I > 1$, $Z_{c+1}$ arises from $Z_c$ as follows: 
   Copy $Z_c$, decreasing $n_I$ by one. 
   Then define a new type 
   $\mathbf{x}' = (s, x_{I0}, \dots, x_{Ij(I)})$, 
   draw $J$ uniformly from $\{1,\dots,d+1\}$ and insert $\mathbf{x}'$ with 
   multiplicity one into $Z_{c+1}$ just before the previous type $J$ 
   (with the convention that the new type is placed at the end of $Z_{c+1}$ 
   when $J=d+1$). 

   Increase $s$ and $d$ each by one. 
 \end{itemize}
\item[$\circ$] If $U > \frac{k r}{k r + \tilde{q}^{(n)}_{k}}$, then:
\begin{itemize}
\item[-] If $|{\bf n}|=n$, stop.
\item[-] Otherwise, pick $J \in \{k+1,\dots,n\}$ with $\Pr\{J=j\} =
 {\tilde{q}^{(n)}_{\#{\bf n},j}}/{\tilde{q}^{(n)}_{\#{\bf n}}}$. 
 Copy $Z_{c+1}$ from $Z_c$. 
 Choose one of the present types $I$ (according to their present
 frequency), and add $J-|{\bf n}|$ copies of this type, i.e.~replace
 $n_i := n_i + J-|{\bf n}|$ in $Z_{c+1}$.
\end{itemize}
\end{itemize}
\smallskip 

\item[3)] Increase $c$ by one, repeat 2). 
\smallskip 

\end{itemize}
\end{minipage}
\caption{Algorithm to generate a sample under the $\Lambda$-coalescent in the infinitely many sites model.}
\label{appendix_algorithm_sample}
\end{algorithm}

\subsection{A discussion of the combinatorial factor $c({\bf t},{\bf n})$ appearing in \eqref{eq_relation_ordered_unordered}}
\label{sect:combfactor}

Let ${\bf t}$, ${\bf a}$, ${\bf n}^{({\bf a})} = {\bf n}$, 
and thus also the sample size $n=|{\bf n}|$, the number of 
segregating sites $s$ and the number of different types $d$ 
visible in the sample be given. 
We evaluate $c({\bf t},{\bf n})$ more explicitly, 
using ideas from Griffiths \cite{Griffiths1987}. 

Recall that an unordered unlabelled sample configuration with 
unordered types $[{\bf t},{\bf n}]$ is equivalent to a non-planted 
rooted unlabelled graph-theoretic tree $\tau$ with $n$ leaves and 
$s+1$ internal vertices (a rooted graph-theoretic tree is called \emph{planted} 
if the root node has degree one and \emph{non-planted} otherwise), see 
\cite[Theorem~1]{Griffiths1987}. In this parametrisation, 
the leaves of $\tau=\tau([{\bf t},{\bf n}])$ correspond to the 
(unnumbered) samples, the internal nodes to segregating sites 
(except for the root of $\tau$) and types to internal nodes with at least 
one subtended leaf. By contrast, a given $({\bf t},{\bf n})$ with 
$d$ ordered types can be viewed as such a tree in which the $d$ internal 
nodes with at least one subtended leaf carry distinct numbers from 
$\{1,\dots,d\}$, namely the type numbers. 

The basic observation behind the following lemma is that 
removing the root node (and connecting edges) from a rooted tree leaves a number
of (possibly planted) rooted trees that can be grouped into classes of
isomorphic trees.

\begin{lemma} \label{computingct}
Order the types in $[{\bf t},{\bf n}]$ in some 
arbitrary fashion, yielding $({\bf t},{\bf n})$. 
 Let the root of $\tau=\tau([{\bf t},{\bf n}])$ have 
$k>0$ descendants, $0 \leq \ell \leq k$ of which are leaves. Group the 
subtrees founded by the descendants which are not leaves into
isomorphy classes (isomorphy as rooted trees). Write $r$ for the number of
non-leaf classes and $g_1, \dots, g_r$ for their sizes (in some arbitrary
ordering). Necessarily $g_1+\cdots+g_r=k-\ell$. Call representatives of the 
$r$ different classes $\tau_1$, \dots, $\tau_r$. 
There are
\begin{equation}\label{eq_number_of_permutations}
c({\bf t},{\bf n}) = c(\tau) = \prod_{i=1}^r c(\tau_i)^{g_i} g_i!
\end{equation}
permutations of the type numbers that do not change $\tau$, 
with the empty product interpreted as~$1$, and 
$c({\bf t},{\bf n})$ is defined in \eqref{eq_ctn_definition}.
\end{lemma}

\begin{proof}
We prove the statement by induction on the number of nodes in $\tau$
(equivalently, the sample complexity). For a tree with $3$ nodes, 
corresponding to a sample of size $2$ with no mutations, 
Equation~\eqref{eq_number_of_permutations} yields the correct answer $1$. 

Now consider $\tau$,
where the root has $k-\ell$ non-leaf descendants in $r$ classes of sizes
$g_1,\dots,g_r$. For each $i=1,\dots,r$ there are $c(\tau_i)$ ways to permute
the type names without changing $\tau_i$ (viewed as an unnumbered unlabelled 
sample with ordered types). Since there are $g_i$ representatives
of this class attached to the root, this yields $c(\tau_i)^{g_i}$ possibilities.
Additionally, we can interchange the complete set of type names between the
subtrees in class $i$, giving another factor $g_i!$. Since the type name 
changes in a given class do not affect the changes in the other classes, 
the factors from each class have to be multiplied to obtain the result.
\end{proof}

\begin{remark} 
(1)~See Figure~\ref{combinatorics_fig_invariant_genetree}, 
\ref{combinatorics_fig_invariant_graph_theoretic} for two representations 
of 
\[({\bf t},{\bf n})=\big(((2,1,0),(3,1,0),(4,0)), (2,2,3)\big)
\] 
which has $c({\bf t},{\bf n})=2$. \\[0.5ex]
(2)~When implementing the recursion \eqref{eq_number_of_permutations} on 
a computer, one obviously has to compute isomorphy classes of subtrees 
of a given tree. There, we have found it useful to pass to planar 
representatives of the given graph-theoretic rooted trees and implement 
a total order on such trees (for which there are various possibilities). 
\end{remark}

\begin{figure}
\begin{center}
\subfigure[Exchanging type I and II does not alter the genetree.] {
\label{combinatorics_fig_invariant_genetree}
\setlength{\nodedist}{40pt}
\begin{tikzpicture}[scale=0.8, transform shape,node distance=\nodedist]
\tikzstyle{level 1}=[level distance=0.8\nodedist]
\node[very thick, draw,circle] (root) at (0,0) {}
child {node[draw,circle] (m1) {}
child {node[draw,circle,label=right:\footnotesize{I:2}] (m4) {}}
child {node[draw,circle,label=right:\footnotesize{II:2}] (m5) {}
child {edge from parent[draw=none]}}}
child {node[draw,circle,label=right:\footnotesize{III:3}] (m2) {}};
\end{tikzpicture}
\hspace{1cm}
}
\subfigure[Exchanging type I and II does not alter the graph-theoretic tree.] {
\label{combinatorics_fig_invariant_graph_theoretic}
\setlength{\nodedist}{40pt}
\hspace{1cm}
\begin{tikzpicture}[scale=0.8, transform shape,node distance=\nodedist]
\tikzstyle{level 1}=[level distance=0.8\nodedist,sibling distance=2.5\nodedist]
\tikzstyle{level 2}=[sibling distance=1\nodedist]
\tikzstyle{level 3}=[sibling distance=0.5\nodedist]

\node[very thick, draw,circle] (root) at (0,0) {}
child {node[draw,circle] (m1) {}
child {node[draw,circle,label=right:\footnotesize{I}] (m4) {}
child {node[draw,circle] {}}
child {node[draw,circle] {}}}
child {node[draw,circle,label=right:\footnotesize{II}] (m5) {}
child {node[draw,circle] {}}
child {node[draw,circle] {}}}}
child {node[draw,circle,label=right:\footnotesize{III}] (m2) {}
child {node[draw,circle] {}}
child {node[draw,circle] {}}
child {node[draw,circle] {}}};
\end{tikzpicture}
}
\end{center}
\caption{The effect that reordering does not change the tree visualised in both corresponding representations.}
\label{combinatorics_fig_invariant}
\end{figure}
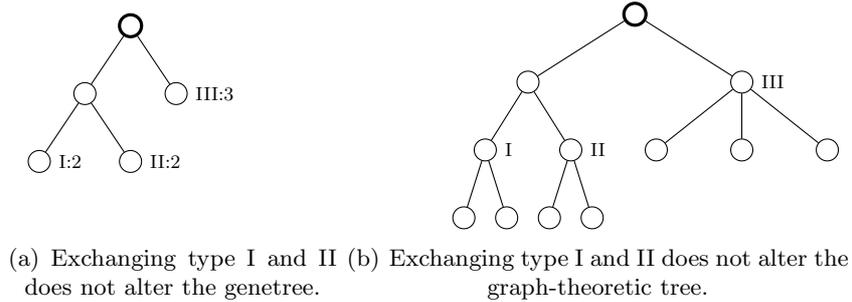

\subsection{Speed-up: Precomputations and multiple parameter sets}

\label{sect:precomp}

Assume that for some $\mathcal{A} \subset \mathcal{T}^*$, 
$p_\theta(\mathbf{t}',\mathbf{n}')$ is (numerically) known 
for all $(\mathbf{t}',\mathbf{n}') \in \mathcal{A}$. 
In practice, this can be achieved by including in $\mathcal{A}$ only 
such samples for which (\ref{ordered_ims_recursion}) can be solved 
numerically on the given computer architecture. 

This information can be combined with importance sampling schemes as
discussed above by running the proposal chains only until they hit
$\mathcal{A}$, thus reducing the variance of the estimators: 
Let $\tilde{\mathcal{H}} = (\tilde{H}_i) := (H_{-i})$ be the time-reversed history, 
$(\mathbf{t},\mathbf{n}) \in \mathcal{T}$ with $|\mathbf{n}|=n$ be given 
and let $\Q$ be a proposal distribution 
(compatible with \eqref{importance_eq_support_condition}) 
under which $(\tilde{H}_i)$ is a Markov chain, starting from 
$\tilde{H}_0=(\mathbf{t},\mathbf{n})$. Then we have 
\begin{align} 
\label{eq_proposal_exp_precomp}
p_\theta(\mathbf{t},\mathbf{n}) 
= \frac{\lambda_n}{rn+\lambda_n} 
\E_\Q\bigg[ \Big( \prod_{i=0}^{\tau_\mathcal{A}-1} 
\frac{\P_{\theta,n}(\tilde{H}_{i+1}\to \tilde{H}_{i})}{
\Q(\tilde{H}_{i}\to \tilde{H}_{i+1})} \Big) 
g(n,|\tilde{H}_{\tau_\mathcal{A}}|) 
(|\tilde{H}_{\tau_\mathcal{A}}|r+\lambda_{|\tilde{H}_{\tau_\mathcal{A}}|}) p_\theta(\tilde{H}_{\tau_\mathcal{A}})
\bigg],
\end{align}
where $\tau_{\mathcal{A}}:=\min\{ i : \tilde{H}_i \in \mathcal{A}\}$ and 
$|\tilde{H}_{\tau_\mathcal{A}}|$ denotes the number of samples in 
$\tilde{H}_{\tau_\mathcal{A}}$. 
Analogous to \eqref{eq_proposal_estimator}, by averaging 
the term inside the $\Q$-expectation in \eqref{eq_proposal_exp_precomp} 
over independent draws from $\Q$, this yields an 
unbiased estimator of $p_\theta(\mathbf{t},\mathbf{n})$ whose variance 
will be smaller than that of \eqref{eq_proposal_estimator}. 

For given $(\mathbf{t},\mathbf{n})=h_0, h_1,\dots,h_s \in \mathcal{T}^*$ 
with $h_i \not\in \mathcal{A}$, $i=0,1,\dots,s-1$, $h_s\in \mathcal{A}$, 
we have 
\begin{align*} 
\P_{\theta,n}\big( & (H_{-s},H_{-s+1},\dots,H_0) = (h_s,\dots,h_0)\big) \\
& = \, \P_{\theta,n}( \mathcal{H} \: \text{hits}\: h_s) 
\Big( \prod_{i=0}^{s-1}\P_{\theta,n}(h_{i+1}\to h_{i}) \Big) 
\frac{\lambda_n}{rn+\lambda_n} 
\end{align*}
by the Markov property under $\P_{\theta,n}$, 
thus \eqref{eq_proposal_exp_precomp} follows 
from \eqref{eq:relGandhitting}, Lemma~\ref{lem:Gandp} and the 
Markov property under $\Q$. 

Note that \eqref{eq_proposal_estimator} and the analogous estimator
built from \eqref{eq_proposal_exp_precomp} can be used to
simultaneously estimate $p_\theta(\mathbf{t},\mathbf{n})$ for various
values of $\theta$ from the \emph{same} runs under a given $\Q$ (of
course, yielding correlated estimators). This can be computationally
more efficient for example when computing likelihood surfaces.
See, e.g., \cite{Tavare2004}, Sect.~6.3 on how to combine estimators from 
different runs.

\subsection{Estimating times and aspects of the genealogy given the data} 

\label{sect:histestimation}

The time-reversed history $(\tilde{H}_i) = (H_{-i})$
describes the skeleton chain of \mbox{a(n $n$-)$\Lambda$}-coalescent with 
mutations according to the IMS model. It is straightforward to augment this 
with `real times' (on the coalescent time scale): 
Given $\tilde{\H}=(\tilde{H}_0,\dots,\tilde{H}_{\tau-1})$, 
the coalescent process will spend time $V_i$ in the $i$-th state, 
where the $V_i$ are conditionally independent with 
$\mathcal{L}(V_i|\tilde{H}) = 
\text{Exp}(r|\tilde{H}_i|+\lambda_{|\tilde{H}_i|})$, thus 
$T_i:=V_0+\cdots+V_{i-1}$, the time of the $i$-th event, can be readily 
simulated given $\tilde{\H}$. Furthermore, for any function 
$f\big((\tilde{H}_i),(T_i)\big)$ of the reversed history and its (coalescent) 
time embedding, we have 
\begin{align} 
\label{eq:est_exp_hist_functional}
\E_{\theta,n}& \Big[f\big((\tilde{H}_i),(T_i)\big) 
\1_{\{H_0=(\mathbf{t},\mathbf{n})\}}\Big] \\ \notag
& = \, 
\E_\Q\Big[ \frac{\P_{\theta,n}(\tilde{\H})}{\Q(\tilde{\H})} 
f\big((\tilde{H}_i),(T_i)\big) \1_{\{H_0=(\mathbf{t},\mathbf{n})\}} \Big]
\end{align}
for any proposal distribution $\Q$ satisfying 
\eqref{importance_eq_support_condition}, where implicitly, 
the conditional law of $(T_i)$ given $\tilde\H=(\tilde{H}_i)$ is the same 
under $\Q$ and under $\P_{\theta,n}$. Thus, in analogy with 
\eqref{eq_proposal_estimator}, 
\begin{align} 
\frac{1}{M} \sum_{j=1}^M \1_{\{(\tilde\H^{(j)})_0 =({\bf t},{\bf n})\}} \frac{d\P_{\theta,n}} {d\Q}(\tilde\H^{(j)}) f\big(\tilde\H^{(j)},(T_i^{(j)})\big)
\end{align}
is an unbiased and consistent estimator of 
\eqref{eq:est_exp_hist_functional}, where $\tilde\H^{(1)},\dots,
\tilde\H^{(M)}$ and the corresponding $(T_i^{(1)}),\dots,(T_i^{(M)})$ 
are independently drawn from $\Q$. 
\smallskip

For example, using $f\big((\tilde{h}_i),(t_i)\big)=t_1+\cdots+t_{\tau-1}$ 
or $f\big((\tilde{h}_i),(t_i)\big)=\1(t_1+\cdots+t_{\tau-1} \leq x)$, 
combined with an estimate of $p_\theta(\mathbf{t},\mathbf{n})$, this 
approach can be used to estimate the conditional mean or even the 
conditional distribution of the time to the most recent ancestor of the sample, given the observed data. Similarly, the conditional age 
of a particular mutation can be estimated (when undoing the 
equivalence relation $\sim$). This extends the line of thought 
from \cite{Griffiths1994} to the Lambda-coalescent context.

\end{document}